\documentclass[11pt]{amsart}
\usepackage[margin=1in]{geometry}
\usepackage{amsmath,amssymb,amsthm, hyperref, amscd}

\usepackage{enumitem}
\usepackage[T1]{fontenc}
\usepackage{textcomp}
\usepackage[scale=0.95,ttdefault]{AnonymousPro}
\hypersetup{
   colorlinks=true,
   linkcolor=blue,
    urlcolor=cyan,
    }

\usepackage[
    backend=biber,
    style=alphabetic,
    useprefix=true,
    maxnames=100,
    maxcitenames=100,
    sorting=anyt
]{biblatex}
\newcommand{\fa}{\mathfrak{a}}
\newcommand{\fb}{\mathfrak{b}}
\newcommand{\fc}{\mathfrak{c}}
\newcommand{\fr}{\mathfrak{r}}
\newcommand{\fm}{\mathfrak{m}}
\newcommand{\fn}{\mathfrak{n}}
\newcommand{\fA}{\mathfrak{A}}
\newcommand{\fB}{\mathfrak{B}}

\newcommand{\fN}{\mathfrak{N}}
\newcommand{\fp}{\mathfrak{p}}
\newcommand{\fq}{\mathfrak{q}}
\newcommand{\fP}{\mathfrak{P}}
\newcommand{\fQ}{\mathfrak{Q}}
\newcommand{\Spec}{\operatorname{Spec}}
\newcommand{\Ass}{\operatorname{Ass}}
\newcommand{\Sing}{\operatorname{Sing}}
\newcommand{\Frac}{\operatorname{Frac}}
\newcommand{\Disc}{\operatorname{Disc}}
\newcommand{\Tr}{\operatorname{Tr}}
\newcommand{\HT}{\operatorname{ht}}
\newcommand{\rank}{\operatorname{rank}}
\newcommand{\depth}{\operatorname{depth}}
\newcommand{\colim}{\operatorname{colim}}

\newcommand{\cC}{\mathcal{C}}

\newcommand{\bF}{\mathbf{F}}
\newcommand{\bR}{\mathbf{R}}

\newcommand{\bZ}{\mathbf{Z}}

\newcommand{\Tor}
{\operatorname{Tor}}
\newcommand{\Hom}
{\operatorname{Hom}}
\newcommand{\Ext}{\operatorname{Ext}}

\newcommand{\Tag}[1]{\href{https://stacks.math.columbia.edu/tag/#1}{\texttt{#1}}}
\newcommand{\citestacks}[1]{\cite[Tag \Tag{#1}]{stacks}}
\newcommand{\citetwostacks}[2]{\cite[Tags \Tag{#1} and \Tag{#2}]{stacks}}

\newcommand{\ppower}[2]{{#1}^{[p^{#2}]}}
\newcommand{\ehk}[1]{e_{\mathrm{HK}}({#1})}
\newcommand{\IFsig}[2]{I^{F\operatorname{-sig}}_{#1}({#2})}

\newtheorem{Thm}{Theorem}[subsection]

\newtheorem{Lem}[Thm]{Lemma}
\newtheorem{Cor}[Thm]{Corollary}
\newtheorem{Prop}[Thm]{Proposition}

\theoremstyle{definition}
\newtheorem{Def}[Thm]{Definition}

\newtheorem{Condit}[Thm]{Condition}
\newtheorem{Notns}[Thm]{Notation}
\newtheorem{Discu}[Thm]{Discussion}
\newtheorem{Situ}[Thm]{Situation}

\theoremstyle{remark}
\newtheorem{Fact}[Thm]{Fact}
\newtheorem{Rem}[Thm]{Remark}
\newtheorem{Ques}[Thm]{Question}
\newtheorem{StepLocalUnif}{Step}

\numberwithin{equation}{Thm}

\newtheorem*{Claim}{Claim}

\title[Uniform bound]{Uniform bounds in excellent $\bF_p$-algebras and applications to semi-continuity}

\author{Shiji Lyu}
\address{Institute of Mathematics\\Chinese Academy of Sciences\\Beijing, China}
\email{\href{mailto:slyu@alumni.princeton.edu}{slyu@alumni.princeton.edu}}
\urladdr{\url{https://slyu2026.github.io/}}

\author{Cheng Meng}
\address{Yau Mathematical Sciences Center\\ Tsinghua University\\ Beijing, China 100084}
\email{\href{mailto:cheng319000@mail.tsinghua.edu.cn}{cheng319000@mail.tsinghua.edu.cn}}
\urladdr{\url{https://cmengca23.github.io/}}

\begin{document}
\begin{abstract}
We study two important numerical invariants, Hilbert--Kunz multiplicity and $F$-signature, on the spectrum of a Noetherian $\bF_p$-algebra $R$ that is not necessarily $F$-finite.
When $R$ is excellent,
we show that the limits defining the invariants are uniform.
As a consequence, 
we show that the $F$-signature function (of the ring, pairs, and triples) is lower semi-continuous, and  the Hilbert--Kunz multiplicity function is upper semi-continuous provided $R$ is locally equidimensional.
Uniform convergence is achieved via a uniform version of the Cohen--Gabber normalization theorem.
We prove the results under weaker conditions than excellence.
\end{abstract}
\maketitle

We mostly work with $\bF_p$-algebras, where $p$ is an arbitrary prime number.
The Frobenius map of an $\bF_p$-algebra $R$ is $x\mapsto x^p$.
$R$ is \emph{$F$-finite} if the Frobenius map is finite.
The target of the $e$th iteration of the Frobenius map is denoted by $F^e_*R$;
in other words, $F^e_*R$ is $R$ as an abstract ring but equipped with the $R$-algebra structure $rF^e_*s=F^e_*(r^ps)$.
Similarly we have $F^e_*M$ for an $R$-module $M$.
For an ideal $I$ of $R$, $\ppower{I}{e}$ is the ideal generated by the $p^e$-th powers of elements of $I$.
In other words,
$IF^e_*R=F^e_*(\ppower{I}{e})$.
In particular, $\ppower{I}{0}=I$.

We say a function $\psi$ on $\Spec(R)$ is \emph{constructible} if it is locally constant in the constructible topology.
Concretely,
this says for all $\fp\in\Spec(R)$,
there exists a finitely generated ideal $I\subseteq \fp$ and an element $f\not\in \fp$
such that $\psi(\fq)=\psi(\fp)$ for all $\fq\in V(I)\cap D(f)$.
A constructible function $\psi$ is lower (resp. upper) semi-continuous if and only if $\psi(\fp_1)\leq \psi(\fp_2)$ (resp. $\psi(\fp_1)\geq \psi(\fp_2)$)
for all $\fp_1\supseteq \fp_2$.
See \citestacks{0542}.
The constructible topology is compact \citestacks{0901}; in particular,
a constructible function only takes finitely many values.

For a local ring $A$, $A^\wedge$ denotes its completion,
$A^h$ denotes its Henselization,
and $A^{sh}$ denotes its strict Henselization.
For a ring $R$, $R_{\operatorname{red}}$ denotes its reduction,
and $R^\nu$ denotes the integral closure of $R_{\operatorname{red}}$ in its total fraction ring.
\section{Introduction} 

\subsection{The subject}
The study of \emph{$F$-singularities},
that is, singularities of $\bF_p$-algebras defined with the Frobenius map and its iterations,
has a rich history.
Two of the most important numerical invariants defined this way for a Noetherian local $\bF_p$-algebra $(A,\fm)$,
are the \emph{Hilbert--Kunz multiplicity}
\[
\ehk{A}=\lim_e \frac{1}{p^{e\dim A}} l(A/\ppower{\fm}{e})
\]
and the \emph{$F$-signature}
\[
s(A)=\lim_e s_e(A)
\]
where $s_e(A)$ is the \emph{$F$-splitting number},
which is the maximal rank of a free $A$-submodule of $F^e_*A$ suitably normalized,
in the case $A$ is $F$-finite.
These limits exist \cite{Monsky-HKmult,Tucker-SRexists} -- this is already highly non-trivial.
We also know these invariants detect regularity \cite{Watanabe-Yoshida-eHK,HL-F-sig-regular},
so they can be viewed as measures of singularities.

We may consider two more general invariants than the $F$-signature. Let $R$ be a normal local domain, $\Delta$ be an effective $\mathbb{R}$-divisor, $\mathfrak{a}$ be a nonzero $R$-ideal and $t \in \mathbb{R}_{\geq 0}$. In \cite{BST-pairs-triples}, the authors prove the existence of the \textit{$F$-signature of the pair $(R,\mathfrak{a}^t)$ and the triple $(R,\Delta,\mathfrak{a}^t)$}
$$s(R,\mathfrak{a}^t),s(R,\Delta,\mathfrak{a}^t)$$
which are limits of suitably normalized colengths of certain sequences of ideals. When we fix $R,\Delta,\mathfrak{a}$ and let $t$ vary, they will become Lipshitz continuous functions. The $F$-signature is just the special case where $\Delta=0$ and $\mathfrak{a}=R$. These generalized notions appear in the study of $\mathfrak{a}^t$-tight closures.

In this paper,
we aim to study the semi-continuity of the functions $\fp\mapsto\ehk{R_\fp}$, $\fp\mapsto s(R_\fp)$, and $\fp\mapsto s(R_\fp, \mathfrak{a}^t_\fp)$, $\fp\mapsto s(R_\fp, \Delta_\fp, \mathfrak{a}^t_\fp)$ for fixed $\mathfrak{a}$ or fixed $(\mathfrak{a},\Delta)$ on the spectrum of a Noetherian ring $R$.
Concerning results on the Hilbert--Kunz multiplicity and the $F$-signature, the behavior with respect to generalization is known \cite{Kunz1976,YaoSe};
for $e_{\operatorname{HK}}$ you need the mild assumption that $\HT(\fp_1)=\HT(\fp_1/\fp_2)+\HT(\fp_2)$ for $\fp_1\supseteq\fp_2$.
However, without suitable assumptions on $R$ it cannot be expected that the functions are semi-continuous,
as illustrated in \cite{Hoc73-nonopen}.

For all $4$ notions $e_{\operatorname{HK}},s(R),s(R,\mathfrak{a}^t),s(R,\Delta,\mathfrak{a}^t)$, the common assumption to get things going is that $R$ is $F$-finite.
Indeed, it takes effort to even define $s_e(A)$ when $A$ is not $F$-finite.
In the $F$-finite case,
semi-continuity is established in \cite{Smirnov-eHKsemicont,PolstraUniform}.
There, the result is also shown for rings essentially of finite type over an excellent local ring, via reduction to the $F$-finite case.

We will show semi-continuity assuming only $R$ is excellent (up to the height condition for $e_{\operatorname{HK}}$),
and even less.
This requires new ideas, as \cite{Smirnov-eHKsemicont,PolstraUniform} rely on the analysis of the fixed, finite map $R\to F_*R$.
This also covers some interesting new cases, such as Tate algebras in rigid analytic geometry,
which are indeed quite interesting in terms of $F$-singularities,
as discussed in \cite{murayama-datta-tate}.

\subsection{Summary of results}
\label{sec:FIRST}
In this subsection,
let $R$ be an excellent 
(Noetherian) $\bF_p$-algebra.
The results hold with weaker assumptions; see the corresponding items in the main body.

\begin{Thm}[a uniform version of Cohen--Gabber theorem; see Theorem \ref{thm:UnifCohenGabber}]\label{thm:UnifCohenGabberIntro}
Assume that $R$ is $(R_0)$.
Then there exist constants $\delta,\mu,\Delta,\rho\in\bZ_{\geq 0}$ depending only on $R$,
and a quasi-finite, syntomic ring map $R\to S$,
such that for all $\fp\in\Spec(R)$,
there exist a $\fq\in\Spec(S)$ above $\fp$
and a ring map $P\to S^\wedge_\fq$
that satisfy the following.
\begin{enumerate}[label=$(\roman*)$]
    \item\label{UnifCGIntro_PisP} $(P,\fm_P)$ is a formal power series ring over a field. 
    \item\label{UnifCGIntro_samek}
    $P/\fm_P=\kappa(\fq)$.
    \item\label{UnifCGIntro_degree} $P\to S^\wedge_\fq$ is finite and generically \'etale of generic degree $\leq \delta$.
    \item\label{UnifCGIntro_embdim} $\fq S^\wedge_\fq/\fm_P S^\wedge_\fq$ is generated by at most $\mu$ elements.
    \item\label{UnifCGIntro_disc} There exist $e_1,\ldots,e_n\in S^\wedge_\fq$ that map to a basis of $S^\wedge_\fq\otimes_P\Frac(P)$
    (as an $\Frac(P)$-vector space),
    such that $\Disc_{S^\wedge_\fq/P}(e_1,\ldots,e_n)\not\in \fm_P^{\Delta+1}$.
    \item\label{UnifCGIntro_relembdim}
    There exists a regular system $T_1,\ldots,T_N$ of parameters of $P$ such that for all $j>\rho$ the image of $T_j$ in $S^\wedge_\fq$ lies in the image of $R$.
\end{enumerate}
\end{Thm}

\begin{Rem}
    The bound \ref{UnifCGIntro_relembdim} is only used for $F$-signature of pairs and triples for normal $R$;
    moreover, it is vacuous when $\dim R<\infty$, if we take $\rho=\dim R$.
    Also, see Theorem \ref{thm:UnifCohenGabberWithc} for another quantity we can control.
\end{Rem}

\begin{Thm}[see Theorem \ref{thm:unifbound}]\label{thm:unifboundIntro}
For every finite $R$-module $M$, there exists a constant $C(M)$ with the following property.
For all $\fp\in\Spec(R)$,
all ideals $I\subseteq J$ of $R_\fp$ with $l_{R_\fp}(J/I)<\infty$,
and all $e'\geq e\geq 1$,
the following holds.
\[
\left|\frac{1}{p^{e\dim M_\fp}}l_{R_\fp}\left(\frac{\ppower{J}{e}M_\fp}{\ppower{I}{e}M_\fp}\right)
-\frac{1}{p^{e'\dim M_\fp}}l_{R_\fp}\left(\frac{\ppower{J}{e'}M_\fp}{\ppower{I}{e'}M_\fp}\right)\right|
\leq C(M) p^{-e}l_{R_\fp}(J/I).
\]
\end{Thm}
Here by convention the left hand side is zero if $M_\fp=0$.
\begin{Thm}[see Theorem \ref{thm:eHKSemicont}]
Assume that $R$ is locally equidimensional.
Then the function $\fp\mapsto \ehk{R_\fp}$ is upper semi-continuous.
\end{Thm}

\begin{Thm}[see Theorem \ref{thm:sSemicont} and Corollary \ref{cor:sFRlocusopen}]
\label{Cor:Fsigcont}
The function $\fp\mapsto s(R_\fp)$ is lower semi-continuous,
and the strongly $F$-regular locus of $R$ is open.
\end{Thm}

\begin{Thm}[see Theorem \ref{thm:sTRIPLEsSemicont}]
Let $\fa$ be an ideal of $R$ not contained in any minimal prime of $R$ and let $t\in\bR_{>0}$. Then the function $\fp\mapsto s(R_\fp,\fa_\fp^t)$ is lower semi-continuous.

Assume $R$ is normal and let $\Delta$ be an effective $\bR$-Weil divisor on $R$. Then the function $\fp\mapsto s(R_\fp,\Delta_\fp,\fa_\fp^t)$ is lower semi-continuous.
\end{Thm}

Below, we record results on semicontinuity of individual functions (before taking the limit), including those in the literature and generalizations made by us. 

The following result is due to Shepherd-Barron \cite{SB-HKmult}.
\begin{Thm}[see Lemma \ref{lem:individualHKconstr} and Corollary \ref{cor:individualHKsemicont}]
For every $e$, the function $\fp\mapsto \frac{1}{p^{e\HT\fp}}l({R_\fp}/\ppower{\fp}{e}{R_\fp})$  is constructible,
and is lower semi-continuous if $R$ is locally equidimensional.

\end{Thm}
The following result is due to Enescu--Yao \cite{EY-splittingnumbers} in some key cases,
and Hochster--Yao \cite{HochsterYaoSe} in general, under very mild assumptions.
We follow and slightly generalize the same method to remove the further assumptions.
See a discussion in the appendix.
\begin{Thm}[see Lemma \ref{lem:SeSemicont}]
For every $e$, the function $\fp\mapsto s_e(R_\fp)$ is constructible,
and is upper semi-continuous if $R$ is locally equidimensional.

\end{Thm}

We also extend this result to the case of pais and triples.

\begin{Thm}[see Lemma \ref{lem:SeTRIPLEsSemicont}]
Let $\fa$ be an ideal of $R$ not contained in any minimal prime of $R$ and let $t\in\bR_{>0}$. Then the function $\fp\mapsto s_e(R_\fp,\fa_\fp^t)$ is constructible and lower semi-continuous.

Assume $R$ is normal and let $\Delta$ be an effective $\bR$-Weil divisor on $R$. Then the function $\fp\mapsto s_e(R_\fp,\Delta_\fp,\fa_\fp^t)$ is constructible and lower semi-continuous.
\end{Thm}

\subsection{The strategy}
Once we have the semi-continuity of the individual functions defining the limit (thanks to \cite{SB-HKmult,EY-splittingnumbers,HochsterYaoSe,PolstraUniform} and our generalizations, as stated above),
we only need to show that the convergence is uniform.
The plan for that is set up in \cite{PolstraUniform};
the uniform inequality Theorem \ref{thm:unifboundIntro} gives uniform convergence.
In particular, it recovers the existence of the limit.
See also \cite{PolstraTucker-semicont}.

Note that the left hand side of Theorem \ref{thm:unifboundIntro} is zero when the Frobenius map $F$ is flat on $R_\fp$;
equivalently, when $R_\fp$ is regular, \cite{Kunz1969-reg-F-flat}.
In the $F$-finite case, \cite{PolstraUniform} then studies the fixed, finite map $F$,
which is flat over a nonempty open when $R$ is an integral domain.
Passing to a filtration of $M$ by modules of the form $R/\fp$ allows us 
to control the effect of the non-flatness of $F$ on $M$ by stuff on a quotient of $R$ of smaller dimension,
which is then controlled relatively easily via
Corollary \ref{cor:PTY}.
We can iterate $F$ and maintain our control,
essentially because $\sum_e 1/p^e<\infty$.

In the general (excellent) case,
this is no fixed finite map that mimics the Frobenius,
at least not one I can find.
Fortunately,
the Cohen--Gabber structure/normalization theorem \cite[Expos{\'e} IV, Th\'eor\`eme 2.1.1]{Gabber-ILO}
gives us a workaround for a complete local ring $A$:
we can take a generically \'etale ring map $P\to A$ where $P$ is regular,
so we know $A$ is generically regular.
The effective/numerical control for the non-flatness of Frobenius, and all its iterations, are already considered in \cite{HH90-TightClos};
see Lemma \ref{lem:discKillsCoker}.
What we will need to do is finding such a $P\to R_\fp^\wedge$ for every $\fp\in\Spec(R)$
such that the multiplicity of the discriminant and several simpler invariants are bounded.
We are almost able to do this, see Theorem
\ref{thm:UnifCohenGabberIntro};
this is enough for us to
 derive the uniform inequality  in \S\ref{sec:unifBD} partially following  \cite{PolstraUniform};
then in \S\ref{sec:semicont} we deduce semi-continuity.
We believe that Theorem
\ref{thm:UnifCohenGabberIntro} may be of independent interest.

Among all $4$ notions, to prove the convergence or uniformity of the $F$-signature of pairs or triples, an extra input is needed, called the $p$- and $p^{-1}$-sequences, see Definition \ref{def:SequencesofIdeals}. Basically, these are sequences of ideals $\{I_e\}_{e \in \mathbb{N}}$ satisfying $cI_e^{[p]} \subset I_{e+1}$ and maps $F_*cI_{e+1}$ into $I_e$, plus some other mild assumptions. The notion of these sequences relies on a conductor-like element $c \in A$. The idea of defining and studying sequences of ideals in this generality comes from \cite{PolstraTucker-semicont}, where the authors prove that in $F$-finite rings, the normalized lengths of $p$- and $p^{-1}$-sequences converge. They are able to derive bounds on rate of convergence in terms of multiplicities of $A/cA$. 

Two difficulties appear immediately in the non-$F$-finite case if one tries to copy the proof of \cite{PolstraTucker-semicont}: first, mapping from $F_*I_{e+1}$ to $I_e$ directly would be meaningless if $F_*A \to A$ has no nonzero maps, which is possible even if $A$ is regular \cite{murayama-datta-tate}; second, we need a uniform bound of the multiplicity of $A/cA$ or some other rings on all of $\operatorname{Spec}(R)$, and here $c,A$ may both vary with $\mathfrak{p} \in \operatorname{Spec}(R)$. We are able to overcome the difficulties using the pointwise map $P\to R_\fp^\wedge$, and the uniformity comes from the bounds on the order of the discriminant as an element of $P$. We replace the maps $F_*A \to A$ with maps $F_*P\otimes_P A \to F_*A$. There are enough nonzero maps since $F_*P\otimes_P A \subset F_*A$ is a module-finite extension. One more technical detail is that we may not work with $R_\fp^\wedge$ directly, but with some extension $R_\fp^\wedge \to S$ where there is a map $P \to S$ for some regular ring $P$; for this we need to show that extensions of $p$- and $p^{-1}$-sequences of $R_\fp^\wedge$-ideals are $p$- and $p^{-1}$-sequences of $S$-ideals. Still, the $p^{-1}$ part is technically complicated. All these complexities are solved in Section \ref{sec:unifBD}. Once the uniform convergence for $p$- and $p^{-1}$-sequences is established, we can derive the same result for $F$-signature of pairs and triples since they come from certain $p$- and $p^{-1}$-sequences.

The purpose of \S\ref{sec:tame} is to establish
Theorem
\ref{thm:UnifCohenGabberIntro}.
To bound the multiplicity of the discriminant,
we borrow ideas from ramification theory of discrete valuations,
see Lemma \ref{lem:TameDiscCompute}.
We can then use local Bertini (Lemma \ref{lem:R0Bertini}) to find
$1$-dimensional and $(R_0)$ quotients of every $R_\fp$.
For ramification theory to be applicable,
we need Condition \ref{condit:tame} for these quotients.
This is why we need to take an extension $R\to S$ in Theorem
\ref{thm:UnifCohenGabberIntro}.
Curiously, the way we show that such an extension is good enough pointwise in \S\ref{subsec:findTame},
relies on a version of Cohen--Gabber (Theorem \ref{thm:nonComplCG}).
We remark that these discussions somewhat overlap with \cite{Skalit-CG}.
In \S\ref{subsec:uniftame} we spread out from a point to a constructible neighborhood,
eventually arriving at Theorem
\ref{thm:UnifCohenGabberIntro}
in \S\ref{subsec:UnifCG}.

Sections \ref{sec:Prelim1} and \ref{sec:Prelim2} are collections of preliminary facts.
The main theorems are obtained in Section \ref{sec:semicont}.

\subsection{Further directions}
The next natural thing to consider are the (relative) $F$-rational signature $s_{\operatorname{rat}}$ and  $s_{\operatorname{rel}}$ \cite{Hochster-Yao-srat,Smirnov-Tucker-srat}.
Semi-continuity of $s_{\operatorname{rel}}$ is established for $F$-finite rings in \cite{Smirnov-Tucker-srat} and rings essentially of finite type over a quasi-excellent local ring in \cite{Lyu-srat}.
These invariants can similarly be defined as limits and uniform convergence follows from Theorem \ref{thm:unifboundIntro}.
However, it seems to be difficult to show the semi-continuity of the individual functions in the limit,
even with the help of \cite{HochsterYaoSe}.

A more ambitious goal is to forget about the workaround and directly find some fixed, finite map that mimics the Frobenius,
after potentially replacing $R$ by a syntomic algebra.
For example, if the following question admits an affirmative answer,
then a new perspective appears in the picture.
\begin{Ques}\label{ques:decomplete-Frobenius}
    Let $A$ be a Noetherian complete local $\bF_p$-algebra.
    Does there exist a coefficient field $k$ of $A$ such that,
    the base change $A'= k^{1/p^\infty}\otimes_k A$
    admits a finite algebra $A''$, such that the completion $\widehat{A'}\to \widehat{A''}$ is isomorphic to the Frobenius map of $\widehat{A'}$?
    Picture 
    \[
    \begin{CD}
        A'@>{\exists ?}>> A''\\
        @VVV @VVV\\
        \widehat{A'}@>>>\widehat{A''}\\
        @V{=}VV @V{\wr}V{\exists}V\\
        \widehat{A'}@>{F}>>\widehat{A'}
    \end{CD}
    \]
\end{Ques}
We note that for any coefficient field $k$ of $A$,
the ring $A'$ is a Noetherian local ring whose completion is $F$-finite.
However, the ring $A'$ is quite exotic,
as it is not (and quite far from) excellent.

\subsection{Acknowledgment} We thank Melvin Hochster, Linquan Ma, Thomas Polstra, Karl Schwede, Kevin Tucker, and Wenliang Zhang for helpful discussions.
SL was supported by an AMS-Simons Travel Grant.

\section{Preliminaries}\label{sec:Prelim1}

\subsection{Excellence conditions}
For the definition of catenary and universally catenary rings, see \citetwostacks{00NI}{00NL}.
An algebra essentially of finite type over a universally catenary ring is universally catenary \citestacks{0ECE}.
If $R$ is an equidimensional (Noetherian) local ring, 
and $R$ is universally catenary, 
then $R^\wedge$ is equidimensional, cf. \citestacks{0AW3}.\footnote{Conversely, if $R$ is a Noetherian local ring such that $R^\wedge$ is equidimensional, then $R$ is equidimensional
and universally catenary, \citetwostacks{0AW4}{0AW5}.}

A Noetherian ring $R$ is
\emph{J-0} if the regular locus of $R$ contains a non-empty open,
\emph{J-1} if the regular locus of $R$ is open,
and 
\emph{J-2} if every finite type $R$-algebra is J-1.
See \citestacks{07P6}.
An algebra essentially of finite type over a J-2 ring is J-2 by definition.
If $R$ is a Noetherian ring such that $R/\fp$ is J-0 for all minimal primes $\fp$ of $R$ (resp. all $\fp\in\Spec(R)$),
then the normal (resp. Cohen--Macaulay) locus of $R$ is open, see
\cite[Proposition 6.13.4]{EGA4_2} (resp.  \cite[Proposition 6.11.8]{EGA4_2}).

A ring $R$ is called \emph{Nagata} if $R$ is Noetherian and, for all primes $\fp$ of $R$ and all finite extensions $L/\Frac(R/\fp)$,
the integral closure of $R/\fp$ in $L$ is finite.
See \citestacks{032R}.
If $R$ is Nagata, then any $R$-algebra essentially of finite type is Nagata \citetwostacks{0334}{032U}.
If $R$ is a reduced local ring, and $R$ is Nagata, then $R^\wedge$ is reduced, cf. \citestacks{0331}. 

A Noetherian ring $R$ is a \emph{G-ring} if for all $\fp\in\Spec(R)$,
the fibers of $R_\fp\to R_\fp^\wedge$ are geometrically regular.
A ring $R$ is \emph{quasi-excellent} if it is Noetherian, a G-ring, and J-2;
\emph{excellent} if it is quasi-excellent and universally catenary.
A Noetherian local G-ring is quasi-excellent.
A quasi-excellent ring is Nagata.
See \cite[(34.A)]{Mat-CA}.
A finitely generated algebra over a Noetherian complete local ring is excellent \citestacks{07QW}.

\subsection{Notions of $F$-singularities}
Let $R$ be a Noetherian $\bF_p$-algebra.
We say $R$ is \emph{$F$-pure} if the canonical map $R\to F_*^eR$ is pure for some (or equivalently, all) $e\geq 1$.

We say $R$ is \emph{weakly $F$-regular} if all ideals of $R$ are tightly closed.
We say $R$ is \emph{strongly $F$-regular} if all submodules of all $R$-modules are tightly closed, cf. \cite[\S 3]{Has10}.
A weakly $F$-regular ring is normal,
and is Cohen--Macaulay provided it is a quotient of a Cohen--Macaulay ring.
See \cite[Theorem 3.4]{HH94-FSingPerm}.

\subsection{Local Bertini}

We recall the following classical theorem. 
See also \cite{Trivedi-local-Bertini} and \cite{modern-local-Bertini}.
\begin{Thm}[{\cite[Satz 2.1]{Flenner-local-Bertini}}]
\label{thm:FlennerLocalBertini}
Let $(A,\fm)$ be a Noetherian local ring containing a field.\footnote{Flenner stated the theorem for mixed characteristic rings as well, but his proof in that case was wrong.
This is explained and fixed in \cite{Trivedi-local-Bertini}.}
Let $I$ be a proper ideal of $A$.
Let $D(I)$ be the open subset $\Spec(A)\setminus V(I)$ of $\Spec(A)$.
Let $\Sigma$ be a finite subset of $D(I)$.

Then there exists an element $a\in I$ that
is not contained in any prime in $\Sigma$,
and is not contained in $\fp^{(2)}$ for any $\fp\in D(I)$.
\end{Thm}

We shall use the following consequence.
\begin{Lem}
\label{lem:R0Bertini}
Let $(A,\fm)$ be a Noetherian local ring that is $(R_0)$.
Assume every quotient of $A$ is J-1.

Assume $d:=\dim A\geq 1$.
Then there exist elements $a_1,\ldots,a_{d-1}\in \fm$
such that 
\begin{enumerate}[label=$(\roman*)$]
    \item\label{R0Bertini_para} $a_{j+1}$ is not in any minimal prime of $(a_1,\ldots,a_j)$; and that
    \item\label{R0Bertini_R0} $A/(a_1,\ldots,a_{d-1})$ is $(R_0)$.
\end{enumerate}
\end{Lem}
\begin{proof}
This follows from the argument in \cite[\S 3]{Flenner-local-Bertini}.
We reproduce the proof for the reader's convenience.

We can assume $d>1$.
By induction, it suffices to find an element $a_1=a\in \fm$ not in any minimal prime of $A$ such that $A/aA$ is $(R_0)$.
Since $A$ is J-1, the singular locus $\Sing(A)$ is closed in $\Spec(A)$.
Since $A$ is $(R_0)$,
$\Sigma_1=\{\fp\in\Sing(A)\mid \HT{\fp}\leq 1\}$ is finite.
Let $\Sigma_2$ be the set of minimal primes of $A$.
Then $\fm\not\in\Sigma_1\cup\Sigma_2$ since $d>1$. 

By Theorem \ref{thm:FlennerLocalBertini}, 
we can find 
$a\in \fm$ such that
$a$ is not in any prime in $\Sigma_1\cup\Sigma_2$ and 
that $a\not\in \fp^{(2)}$ for all $\fp\in\Spec(A)\setminus\{\fm\}$.
It is then straightforward to verify that $A/aA$ is $(R_0)$.
\end{proof}

\subsection{Discriminant}
We review some basic facts about the discriminant we will use,
which is related to the Dedekind different, \citestacks{0BW0}, cf. \cite[Chapter III]{Lang-ANT}.

We let $A$ be a normal domain, $K$ its fraction field, $B$ a finite $A$-algebra, 
and we assume $B\otimes_A K$ finite \'etale over $K$ of degree $n$.

\begin{Def}
Let $e_1,\ldots,e_n\in B$ be elements that map to a basis of $B\otimes_A K$.
The \emph{discriminant} of $e_1,\ldots,e_n$ is
\[
\Disc_{B/A}(e_1,\ldots,e_n)=\det \left(\Tr_{B/A}(e_ie_j)_{i,j}\right).
\]
where $\Tr_{B/A}$ denotes the Galois-theoretic trace map $B\otimes_A K\to K$.
\end{Def}

Since $B$ is integral over $A$ and $A$ is normal, we have $\Tr_{B/A}(B)\subseteq A$,
and thus $\Disc_{B/A}(e_1,\ldots,e_n)\in A$.
Moreover, it is clear that the discriminant is unchanged along a flat base change $A\to A'$ of normal domains.

This notion is useful to us later because of the following result.

\begin{Lem}\label{lem:discKillsCoker}
Let $A,B,e_1,\ldots,e_n$ be as above. If $B$ is a torsion-free $A$-module and $A$ contains $\bF_p$,
then for any $m\in\bZ_{\geq 0}$, as subsets of $(B\otimes_A K)^{1/p^m}$
\[
\Disc_{B/A}(e_1,\ldots,e_n).B^{1/p^m}\subseteq A^{1/p^m}[B].
\]
\end{Lem}
\begin{proof}
This is \cite[Lemma 6.5]{HH90-TightClos} for $(-)^{1/p^{\infty}}$, but the same proof works in the case of $(-)^{1/p^m}$.
\end{proof}

We need a compatibility result.
\begin{Lem}\label{lem:DiscModfp}

Assume that $(A,\fm)$ is local and that $A\to B$ is finite \'etale.
Let $e_1,\ldots,e_n\in B$ be a basis of $B$ as an $A$-module.

Then

\[
\overline{\Disc_{B/A}(e_1,\ldots,e_n)}=\Disc_{(B/\fm B)/(A/\fm)}(\overline{e_1},\ldots,\overline{e_n})
\]
where $\overline{(-)}$ means $\mathrm{mod}\ \fm$ or $\mathrm{mod}\ \fm B$.
\end{Lem}
\begin{proof}
Let $z_{ijkl}$ be elements of $B$ such that $e_{i}e_je_k=\sum_l z_{ijkl}e_l$.
Then $\Tr_{B/A}(e_ie_j)=\sum_k z_{ijkk}$,
so $\Disc_{B/A}(e_1,\ldots,e_n)=\det\left( (\sum_k z_{ijkk})_{i,j}\right)$.
The same formulas compute the right hand side, showing the desired identity.
\end{proof}

We need an explicit computation.

\begin{Lem}\label{lem:TameDiscCompute}
Let $A\to B$ be a finite map of DVRs.
Let $K=\Frac(A)$, $L=\Frac(B)$, and $s=[L:K]$.
Assume that $s\in A^\times$,
and that the residue fields of $A$ and $B$ are the same 
(i.e. $L/K$ is totally tamely ramified).
Let $v_A$ and $v_B$ be the discrete valuation of $A$ and $B$ respectively.

Assume that there exists $y\in B$ such that $v_B(y)$ and $s$ are relatively prime, and that $x:=y^s\in A$.
Then $v_A(\Disc_{B'/A}(y,\ldots,y^{s-1},y^s))=(s+1)v_A(x)$
for any sub-$A$-algebra $B'$ of $B$ that contains $y$.
\end{Lem}
\begin{proof}
By assumptions $L/K$ is separable and $v_B|_A=sv_A$. 
Since $v_B(y)$ and $s$ are relatively prime, 
it is clear that $y,\ldots,y^{s-1},y^s$ is a basis of $L/K$,
thus for any sub-$A$-algebra $B'$ that contains $y$, $L=\Frac(B')$, so we may assume $B'=B$.

Since $x=y^s\in A$, 
we can easily write down the matrix of a power of $y$ as a linear transformation with respect to the basis $y,\ldots,y^{s-1},y^s$,
and it follows that $\Tr(y^{bs})=sx^b$ and $\Tr(y^a)=0$ if $s$ does not divide $a$.
Thus the matrix $\Tr(y^iy^j)$ has exactly one nonzero entry in each row,
which is $sx$ in the first $s-1$ rows and $sx^2$ in the last one.
Since $s\in A^\times$,
$v_A(\Disc_{B'/A}(y,\ldots,y^{s-1},y^s))=(s+1)v_A(x)$
as desired.
\end{proof}
\begin{Rem}
    We are using $y,\ldots,y^{s-1},y^s$ instead of $1,y,\ldots,y^{s-1}$,
    as later on, we will have a $1$-dimensional local ring with several minimal primes; see the proof of Proposition \ref{prop:tame1dimnl}.
\end{Rem}

\subsection{A non-complete version of Cohen--Gabber}
We will need the following version of Cohen--Gabber \cite[Expos{\'e} IV, Th\'eor\`eme 2.1.1]{Gabber-ILO}.

\begin{Thm}\label{thm:nonComplCG}
Let $(A^{nc},\fm^{nc},k)$ be a Noetherian local $\bF_p$-algebra
and let $(A,\fm,k)$ be the reduction of the completion of $A^{nc}$.
Assume that $A$ is equidimensional,
and assume that for each minimal prime $\fp$ of $A^{nc}$,
there is exactly one minimal prime of $A$ above $\fp$.

Let $d=\dim A$.
Then there exists a set $\Lambda\subseteq A^{nc}$
and a system of parameters $t_1,\ldots,t_d\in \fm^{nc}$
with the following properties.
\begin{enumerate}[label=$(\roman*)$]
    \item $\Lambda$ maps to a $p$-basis of $k$.
    \item For the unique coefficient field $\kappa$ of $A$ containing $\Lambda$ (see
\cite[Chapitre IX, \S 2, no.2, Th\'eor\`eme 1]{Bourbaki-CA-8-9}),
$A$ is finite and generically \'etale over the subring $\kappa[[t_1,\ldots,t_d]]$.
\end{enumerate}
\end{Thm}
\begin{proof}
We run the argument in \cite[Expos{\'e} IV, \S 2.1]{Gabber-ILO} for the ring $A$,
while making sure that the elements of concern
belong in the ring $A^{nc}$.
Let $\fp_1^{nc},\ldots,\fp_c^{nc}$ be the minimal primes of $A^{nc}$,
so by our assumption,
$A$ has exactly $c$ minimal primes
$\fp_1,\ldots,\fp_c$
with $\fp_i\cap A^{nc}=\fp_i^{nc}$.

We start with the constructions in \cite[Expos{\'e} IV, (2.1.4)]{Gabber-ILO}.
Fix a set $\Lambda\subseteq A^{nc}$ that maps to a $p$-basis of $k$ 
and let $\kappa$ be the unique coefficient field of $A$ containing $\Lambda$.
For a finite set $e\subseteq \Lambda$,
let $\kappa_e=\kappa^p(\Lambda\setminus e)$.

By \cite[Expos{\'e} IV, (2.1.6)]{Gabber-ILO},
we can find an $e$ such that for each ring $B=A/\fp_i$,
we have
\[
\rank \widehat{\Omega}^1_{B/\kappa_e}=d+|e|.
\]
Now, we observe that for every ideal $I$ of $A^{nc}$,
the sets
$\{\mathrm{d}(i)\mid i\in I\}$
and
$\{\mathrm{d}(i)\mid i\in IB\}$
generates the same submodule of $\widehat{\Omega}^1_{B/\kappa_e}$.
This is because $\widehat{\Omega}^1_{B/\kappa_e}$
is a finite module over $B$,
so all its submodules are closed,
and because $\{\mathrm{d}(i)\mid i\in I\}$
is dense in
$\{\mathrm{d}(i)\mid i\in IB\}$,
since $\mathrm{d}(\fm^N B)\subseteq \fm^{N-1}\widehat{\Omega}^1_{B/\kappa_e}$.
Applying this to $I=\fp_1^{nc}\cap\ldots\cap\fp_j^{nc}$,
noting that $IA\not\subseteq \fp_{j+1}$
since $\fp_{j+1}\cap A^{nc}=\fp_{j+1}^{nc}$,
we see that the elements $m_{i,j}$ in \cite[Expos{\'e} IV, (2.1.7)]{Gabber-ILO}
can be chosen to be in $A^{nc}$.
Finally, applying the observation to $I=A^{nc}$,
we see that the elements $f_i$ in \cite[Expos{\'e} IV, (2.1.12)]{Gabber-ILO} can be chosen in $A^{nc}$. 
The (arbitrary) system of parameters $\underline{\tau}$ in \cite[Expos{\'e} IV, (2.1.12)]{Gabber-ILO} can also be taken in $A^{nc}$. 
This concludes the proof.
\end{proof}

\subsection{Relative embedded dimension}
Some simple observations.
\begin{Lem}\label{lem:RelativeCotangentSpace}
    Let $R$ be a ring and let $S$ be an $R$-algebra generated by $\rho\in\bZ_{\geq 0}$ elements.
    Then for all $\fq\in S$ we have
    $\dim_{\kappa(\fq)}\fq S_\fq/(\fq^2S_\fq+\fp S_\fq)\leq \rho$,
    where $\fp=\fq\cap R$.
\end{Lem}
\begin{proof}
    Replacing $R$ by $\kappa(\fp)$ we may assume $R$ is a field.
    Then $S$ is the quotient of a regular Noetherian ring of dimension $\rho$, giving the required estimate.
\end{proof}
\begin{Lem}\label{lem:Unifm_S^peInm_R}
    Let $R$ be an $\bF_p$-algebra and let $S$ be a quasi-finite $R$-algebra.
    Then there exists $e_0\in\bZ_{\geq 0}$
    such that for all $\fq\in S$ we have
    $\ppower{\fq}{e_0} S_\fq\subseteq \fp S_\fq$,
    where $\fp=\fq\cap R$.
\end{Lem}
\begin{proof}
    We may assume $S$ is finite (cf. \citestacks{00QB}) and is generated by $n$ elements as an $R$-module.
    Then $e_0=\lceil \log_pn\rceil$ works.
    To see this, replacing $R$ by $\kappa(\fp)$ we may assume $R$ is a field.
    Then the Artinian ring $S_\fq$ has length $\leq n$,
    so $\fq^bS_\fq=\fq^{b+1}S_\fq$ for some $b\leq n$.
    By Nakayama's Lemma, $\fq^b S_\fq=0$,
    so $\ppower{\fq}{e_0}\subseteq\fq^n\subseteq\fq^b=0$.
\end{proof}

\section{Tame ramification}\label{sec:tame}

\subsection{A condition of one-dimensional local rings}\label{subsec:tame1dim}

We consider the following condition of a Noetherian local ring $A$ of dimension $1$.
\begin{Condit}\label{condit:tame}\
\begin{enumerate}[label=$(\roman*)$]
    \item\label{tame_R0} $A^\wedge$ is $(R_0)$.
    \item\label{tame_Hensel} $(A/\fp)^\nu$ is local for all minimal primes $\fp$ of $A$.\footnote{In other words, $A/\fp$ is unibranch \citestacks{0BPZ}.}
    \item\label{tame_resSame} The map $A\to (A/\fp)^\nu$ induces an isomorphism of residue fields 
    for all minimal primes $\fp$ of $A$.\footnote{In particular, $A/\fp$ is geometrically unibranch \citestacks{0BPZ}.}
\end{enumerate}
\end{Condit}

Note that if $A$ is complete, or more generally Henselian, then \ref{tame_Hensel} is automatic;
see \citestacks{0BQ0}.

\begin{Lem}\label{lem:formalR0}
Let $A$ be a Noetherian local ring of dimension $1$.
The following are equivalent:
\begin{enumerate}[label=$(\roman*)$]
    \item\label{formalR0_formalR0}
    the completion $A^\wedge$ of $A$ is $(R_0)$; and
    \item\label{formalR0_nuFin}  $A$ is $(R_0)$, and the normalization of $A$ is finite.
\end{enumerate}
If the equivalent conditions hold, then $(A^\wedge)_\mathrm{red}=(A_\mathrm{red})^\wedge$
and $A^{\wedge\nu}=A^\nu\otimes_A A^\wedge$.
\end{Lem}
\begin{proof}
Assume first that $A^\wedge$ is $(R_0)$, so $A$ is also $(R_0)$.
Let $\fN$ be the nilradical of $A$.
Then
$\fN (A^\wedge)_P=0$ for all minimal primes $P$ of $A^\wedge$,
thus $A^\wedge/\fN A^\wedge=(A_\mathrm{red})^\wedge$ is $(R_0)$.
Since $A$ is one-dimensional,
$A_\mathrm{red}$ is Cohen--Macaulay,
thus $(A_\mathrm{red})^\wedge$ is Cohen--Macaulay,
thus reduced since it is $(R_0)$.
Finiteness of normalization is then classical, see for example \citestacks{032Y}.

Now assume \ref{formalR0_nuFin}.
We need to show \ref{formalR0_formalR0}
and $A^{\wedge\nu}=A^\nu\otimes_A A^\wedge$.
Since $A^\nu$ is finite over $A$,
we see $A^{\nu}\otimes_A A^{\wedge}$ is the completion of $A^\nu$ as a semi-local ring.
Since $A^\nu$ is normal of dimension $1$, it is regular, hence so is $A^{\nu}\otimes_A A^{\wedge}$.
Since $A$ is $(R_0)$,
for any $f\in \fm$, $A_f=(A^\nu)_f$,
thus $(A^\wedge)_f=(A^{\nu}\otimes_A A^{\wedge})_f$, so $A^\wedge$ is $(R_0)$ and $A^\nu\otimes_A A^\wedge=A^{\wedge\nu}$.
\end{proof}

Condition \ref{condit:tame} implies desired tame behavior,
Proposition \ref{prop:tame1dimnl} below.
Before that, some notations.

\begin{Notns}\label{notns:tameinvariants}
Let $A$ be a Noetherian local ring of dimension $1$ that satisfies Condition \ref{condit:tame}.

Let $\fp$ be a minimal prime of $A$.
$(A/\fp)^\nu$ is finite over $A$ by Lemma \ref{lem:formalR0},
thus a DVR.\footnote{In fact, the normalization of a one-dimensional Noetherian domain is always Noetherian by the theorem of Krull--Akizuki \citestacks{00PG}.}
Denote by $v_\fp:A\to \bZ_{\geq 0}\cup\{\infty\}$ the corresponding valuation composed with the map $A\to (A/\fp)^\nu$;
by Condition \ref{condit:tame}\ref{tame_resSame},
we see $v_\fp(a)=l_A((A/\fp)^\nu/a(A/\fp)^\nu)$.
Let $\beta(\fp)\in\bZ_{\geq 0}$ be the minimal $\beta$ such that 
there exists an element $s\in A$ in all minimal primes of $A$ other than $\fp$ and that $v_\fp(s)=\beta$.

Denote by $\fc_\fp$ the conductor of the extension $A/\fp\to (A/\fp)^\nu$,
\emph{i.e.}, $\fc_\fp=\{a\in (A/\fp)^\nu\mid a(A/\fp)^\nu\subseteq A/\fp\}$.
Note that $A/\fp\to (A/\fp)^\nu$ is finite by Lemma \ref{lem:formalR0}, so $\fc_\fp$ is nonzero.
Denote by $\gamma_0(\fp)$ the number $l_A((A/\fp)^\nu/\fc_\fp)$.

Assume now that $A$ contains $\bF_p$.
We denote by $\gamma(\fp)$ 
the minimal integer $\gamma$ such that $\gamma\geq\gamma_0(\fp)+\beta(\fp)$ and that $\gamma$ is not divisible by $p$.
Finally, let $\delta(A)=\sum_\fp \gamma(\fp)$ and 
$\Delta(A)=\sum_{\fp}(\gamma(\fp)+1)^2$, 
where the sum is over all minimal primes.
\end{Notns}

We now present the main result of this subsection.

\begin{Prop}\label{prop:tame1dimnl}
Let $(A,\fm)$ be a Noetherian local ring of dimension $1$ that satisfies Condition \ref{condit:tame}.
Then the following hold.

\begin{enumerate}[label=$(\roman*)$]
    \item\label{tame1dimnl_t_p} Let $\fp$ be a minimal prime of $A$,
and let $n_\fp\in\bZ$, $n_\fp\geq\gamma_0(\fp)$.
Then there exists an element $t_\fp\in A$ lying in all minimal primes other than $\fp$, such that $v_\fp(t_\fp)=n_\fp+\beta(\fp)$.

\item\label{tame1dimnl_t} Assume that $A$ contains $\bF_p$.
Then there exists $t\in \fm$ such that for all minimal primes $\fp$ of $A$,
$v_\fp(t)=\gamma(\fp)$.

\item\label{tame1dimnl_goodpara} Assume that $A$ is complete and contains $\bF_p$. 
For any $t$ as in \ref{tame1dimnl_t},
and any choice of a coefficient field $k\subseteq A$,
the map
$k[[T]]\to A$ mapping $T$ to $t$ is finite and generically \'etale of generic degree $n=\delta(A)$.

\item\label{tame1dimnl_disc}
For any $k[[T]]\to A$ as in \ref{tame1dimnl_goodpara},
there exist elements
$e_1,...,e_n$ of $A$
mapping to a basis of $A[\frac{1}{T}]$ over $k((T))$ such that the $T$-adic valuation of
the discriminant $\Disc_{A/k[[T]]}(e_1,\ldots,e_n)$ is  $\Delta(A)$. 
\end{enumerate}

\end{Prop}

\begin{proof}
By the definition of the conductor,
we see that there exists $r_\fp\in A$ such that $v_\fp(r_\fp)=n_\fp$.
Let $s_\fp$ be an element of $A$ contained in all other minimal primes of $A$ and satisfies $v_\fp(s_\fp)=\beta(\fp)$.
Then $t_\fp=s_\fp r_\fp$ satisfies $v_\fp(t)=n_\fp+\beta(\fp)$,
showing \ref{tame1dimnl_t_p}.

For \ref{tame1dimnl_t}, let $n_\fp=\gamma(\fp)-\beta(\fp)$ for each $\fp$,
and let $t_\fp$ be as in \ref{tame1dimnl_t_p}.
Then $t=\sum_\fp t_\fp$ works.
Note that $t$ must be in $\fm$ since $\gamma(\fp)>0.$

Now we prove \ref{tame1dimnl_goodpara}.
Let $t\in \fm$ be such that for all minimal primes $\fp$ of $A$,
$v_\fp(t)=\gamma(\fp)$.
In particular, $t$ is a parameter of $A$.
Let $k\subseteq A$ be an arbitrary coefficient field,
so the map $k[[T]]\to A$ mapping $T$ to $t$ is finite.
Since $v_\fp(t)=\gamma(\fp)$ is not divisible by $p$ and since the residue field of $(A/\fp)^\nu$ is $k$ (Condition \ref{condit:tame}\ref{tame_resSame}),
we see that $k[[T]]\to (A/\fp)^\nu$
is totally tamely ramified of index $\gamma(\fp)$.
In particular, $k[[T]]\to A/\fp$ is generically \'etale,
thus so is $k[[T]]\to A$ since $A$ is $(R_0)$.
That $k[[T]]\to A$ has generic degree $\delta(A)$ is clear.

It remains to show \ref{tame1dimnl_disc}.
Let $\fp$ be a minimal prime of $A$.
Let $s$ be an element of $A$ contained in all other minimal primes of $A$ and satisfies $v_\fp(s)=\beta(\fp)$.
In $(A/\fp)^{\nu}$ we can write
$s^{\gamma(\fp)}=t^{\beta(\fp)}u$,
where $u\in (A/\fp)^{\nu\times}$.
Then we can write $u=vw_1^{-1},$ with $v\in k^\times$ and $w_1$ 
has residue class $1$ in the residue field of $(A/\fp)^{\nu}$,
since the residue field of $(A/\fp)^{\nu}$ is $k$ (Condition \ref{condit:tame}\ref{tame_resSame}).
Since $p$ does not divide $\gamma(\fp)$,
by Hensel's Lemma $w_1=w^{\gamma(\fp)}$
for some $w\in (A/\fp)^{\nu}$ with residue class $1$.
Then $(ws)^{\gamma(\fp)}=t^{\beta(\fp)}v$.

Let $y\in A$ be such that the image of $y$ in $A/\fp$ is in $\fc_\fp$ and that $v_\fp(y)=n_\fp:=\gamma(\fp)-\beta(\fp)+1$.
This is possible because $n_\fp\geq \gamma_0(\fp)$.
Then similarly we can write $(w'y)^{\gamma(\fp)}=t^{n_\fp}v'\in (A/\fp)^{\nu}$,
where $v'\in k^\times$ and $w'$ has residue class $1$.
Now, since the image of $y$ in $A/\fp$ is in $\fc_\fp$,
there exists $z\in A$ such that the $z=yww'\in A/\fp$.
Finally, let $x_\fp$ be the element $sz$.
Then $x_\fp$ is in all minimal ideals other than $\fp$,
and $x_\fp^{\gamma(\fp)}=t^{\gamma(\fp)+1}vv'\in (A/\fp)^{\nu}$,
where $v,v'\in k^\times$.

We have that $v_\fp(x_\fp)=\gamma(\fp)+1$ and $\gamma(\fp)$ are relatively prime,
so we see that $x_\fp,\ldots,x_\fp^{\gamma(\fp)}$
is a basis of $\Frac(A/\fp)$ over $k((T))$.
Since $x_\fp$ is in all minimal primes other than $\fp$,
we see that $\cup_\fp\{x_\fp,\ldots,x_\fp^{\gamma(\fp)}\}$ is a basis of $A[\frac{1}{T}]$ over $k((T))$.
It suffices to show the discriminant of this basis has $T$-adic valuation $\Delta(A)$;
thus it suffices to show the discriminant of the basis $x_\fp,\ldots,x_\fp^{\gamma(\fp)}$ of $\Frac(A/\fp)$ over $k((T))$
has $T$-adic valuation  $(\gamma(\fp)+1)^2$.
Since $x_\fp^{\gamma(\fp)}$ is the image of $T^{\gamma(\fp)+1}vv'\in k[[T]]$,
this follows from Lemma \ref{lem:TameDiscCompute}.
\end{proof}

We will need to move between a local ring and its completion.
\begin{Lem}\label{lem:tameCompleteSameInv}
Let $A$ be a Noetherian local ring of dimension $1$.
Assume that $A$ satisfies Condition \ref{condit:tame}\ref{tame_R0}\ref{tame_Hensel}.
Then the following hold.
\begin{enumerate}[label=$(\roman*)$]
    \item\label{tamecomp_mini} The map $\fp\mapsto \fp A^\wedge$ is a bijection between the minimal primes of $A$ and those of $A^\wedge$.

    \item\label{tamecomp_3} $A$ satisfies Condition \ref{condit:tame}\ref{tame_resSame}
if and only if $A^\wedge$ does.

    \item\label{tamecomp_inv} If \ref{tamecomp_3} is the case,
    then the map in \ref{tamecomp_mini} identifies $\beta,\gamma_0$ and $\gamma$.
    In particular, $\delta(A)=\delta(A^\wedge)$
    and $\Delta(A)=\Delta(A^\wedge)$.
\end{enumerate}

\end{Lem}
\begin{proof}
Let $\fp$ be a minimal prime of $A$.
By Lemma \ref{lem:formalR0}, $A/\fp\to (A/\fp)^\nu$ is finite,
so by Lemma \ref{lem:formalR0} again
we see that $(A/\fp)^{\nu\wedge}= (A/\fp)^{\wedge\nu}$ is normal.
Condition \ref{condit:tame}\ref{tame_Hensel} says that $(A/\fp)^\nu$ is local,
thus $(A/\fp)^{\nu\wedge}$ is local,
hence a DVR, and its subring $(A/\fp)^{\wedge}$
is then an integral domain.
So $\fp A^\wedge$ is a minimal prime of $A^\wedge$, showing \ref{tamecomp_mini},
and $(A/\fp)^{\nu\wedge}= (A^\wedge/\fp A^\wedge)^{\nu}$, showing \ref{tamecomp_3}.

For \ref{tamecomp_inv},
by previous discussions $v_{\fp A^\wedge}|_{A}=v_{\fp}$.
Since taking conductor and finite intersection commute
with flat base change,
it is clear that $\gamma_0(\fp A^\wedge)=\gamma_0(\fp)$ 
and $\beta(\fp A^\wedge)= \beta(\fp)$. 
Therefore $\gamma(\fp A^\wedge)= \gamma(\fp)$.
\end{proof}

\subsection{Tame curves}

\begin{Def}\label{def:tamecurve}
Let $A$ be a Noetherian local ring, $d=\dim A$.
We say a proper ideal $\fa$ of $A$ \emph{defines a tame curve}
if
\begin{enumerate}[label=$(\roman*)$]
    \item\label{tamecurve_para} all minimal primes of $\fa$ have height $d-1$; and 
    \item\label{tamecurve_tame}
    $A/\fa$ satisfies Condition \ref{condit:tame}.
\end{enumerate}
\end{Def}

\begin{Lem}\label{lem:tamecurveComplete}
Let $A$ be a Noetherian local ring, $\fa$ a proper ideal of $A$.
If $\fa$ defines a tame curve, so does $\fa A^\wedge\subseteq A^\wedge$.

\end{Lem}
\begin{proof}
Since $A\to A^\wedge$ is flat, every minimal prime of $\fa A^\wedge$ has the same height as some minimal prime of $\fa$.
This takes care of \ref{tamecurve_para} in Definition \ref{def:tamecurve}. 
For \ref{tamecurve_tame},
Condition \ref{condit:tame}\ref{tame_R0} for $A/\fa$ and $A^\wedge/\fa A^\wedge$ are the same,
\ref{tame_Hensel} is automatic for the complete local ring $A^\wedge/\fa A^\wedge$,
and $A^\wedge/\fa A^\wedge$ satisfies \ref{tame_resSame} by Lemma \ref{lem:tameCompleteSameInv}.
\end{proof}

Tame curves give Cohen--Gabber type normalizations.
\begin{Thm}\label{thm:TameCurveNormalization}
Let $(A,\fm)$ be a Noetherian local $\bF_p$-algebra, $d=\dim A$. 
Assume that there exist elements $a_1,\ldots,a_{d-1}$
such that $\fa=(a_1,\ldots,a_{d-1})$ defines a tame curve.

Then there exists $t\in \fm$
such that for any coefficient field $k$ of $A^\wedge$,
the map $P:=k[[X_1,\ldots,X_{d-1},T]]\to A^\wedge$ mapping $X_i$ to $a_i$ and $T$ to $t$
is finite of generic degree $n=\delta(A/\fa)$, and is \'etale at the prime $\fP=(X_1,\ldots,X_{d-1})$,
and there exists a basis $e_1,\ldots,e_n$ of $A^\wedge\otimes_P\Frac(P)$ over $\Frac(P)$ such that
\[
\Disc_{A^\wedge/P}(e_1,\ldots,e_n)\not\in \fP+T^{\Delta(A/\fa)+1}P.
\]

\end{Thm}
See Notation \ref{notns:tameinvariants}
for $\delta(-)$ and $\Delta(-)$ in the statement.
\begin{proof}
The completion of $A$ satisfies the same assumptions by Lemmas \ref{lem:tamecurveComplete} and \ref{lem:tameCompleteSameInv}.
We will show that if $A$ is complete, 
and $t\in A$ is such that the image of $t$ in $A/\fa$ is as in Proposition \ref{prop:tame1dimnl}\ref{tame1dimnl_t},
then $t$ works.
This proves the theorem,
since the set of $t$ indicated in Proposition \ref{prop:tame1dimnl}\ref{tame1dimnl_t}
is open in the adic topology.

Assume $A$ and $t$ are as above.
Let $k$ be an arbitrary coefficient field and let $P\to A$ and $\fP$ be as in the statement of our theorem.
Note that $P\to A$ is finite and $\fa=\fP A$.
Since every minimal prime of $\fa$ has height $d-1$,
every maximal ideal of $A_\fP$ has height $d-1$.
Since $P_\fP/\fP P_\fP\to A_\fP/\fP A_\fP$ is finite \'etale of degree $n=\delta(A/\fa)$ (Proposition \ref{prop:tame1dimnl}\ref{tame1dimnl_goodpara}),
and since $P_\fP$ is normal of dimension $d-1$,
$P_\fP\to A_\fP$ is finite \'etale of degree $n$,
see for example 
\citestacks{0GSC}.
Find $n$ elements of $A/\fa=A/\fP A$ as in Proposition \ref{prop:tame1dimnl}\ref{tame1dimnl_disc}
and lift them to elements $e_1,\ldots,e_n\in A$.
Then $e_1,\ldots,e_n$ is a basis of $A_\fP$ over $P_\fP$,
and Lemma \ref{lem:DiscModfp} gives 
$\Disc_{A/P}(e_1,\ldots,e_n)\not\in \fP+T^{\Delta(A/\fa)+1}P$,
as desired.
\end{proof}

\subsection{Finding tame curves}\label{subsec:findTame}

The goal of this subsection is to show that tame curves in the spectrum of a local ring can be found after a reasonable extension (Proposition \ref{prop:locallyfindtamecurve}).

\begin{Lem}\label{lem:curveistameafterFieldExtn}
Let $R$ be a Noetherian  local $\bF_p$-algebra of dimension $1$.
Assume that $R^\wedge$ is $(R_0)$,
and assume that $R/\fp$ is geometrically unibranch for all minimal primes $\fp$ of $R$.
Then the following hold.
\begin{enumerate}[label=$(\roman*)$]
    \item\label{curvetame_goodcoeff} There exist a subset $\Lambda$ of $R$
    and a parameter $t\in R$
    that satisfies the conclusions of Theorem \ref{thm:nonComplCG}.

    \item For $\Lambda,t$ as in \ref{curvetame_goodcoeff},
    put $\kappa_0=\bF_p(\Lambda)$.
    Then there exists a finite purely inseparable extension $\kappa'/\kappa_0$ such that $R\otimes_{\kappa_0} \kappa'$ satisfies Condition \ref{condit:tame}.
\end{enumerate}
\end{Lem}
\begin{proof}
For \ref{curvetame_goodcoeff},
we need to verify the assumptions of Theorem \ref{thm:nonComplCG}.
Since $R$ is one-dimensional,
$A=(R^\wedge)_{\mathrm{red}}$ is equidimensional.
Since $R/\fp$ is (geometrically) unibranch for each minimal prime $\fp$ of $R$ and since
$R^\wedge$ is $(R_0)$,
$R$ satisfies Condition \ref{condit:tame}\ref{tame_R0}\ref{tame_Hensel},
so by Lemma \ref{lem:tameCompleteSameInv},
$\fp\mapsto \fp A$ is a bijection between the minimal primes of $R$ and $A$.
Thus the assuptions of Theorem \ref{thm:nonComplCG}
are satisfied.

Now fix $\Lambda$ and $t$ as in \ref{curvetame_goodcoeff} and let $\kappa_0=\bF_p(\Lambda)$.
Let $\kappa$ be the unique coefficient field of $R^\wedge$ containing $\kappa_0$,
so $A$ is finite and generically \'etale over $\kappa[[t]]$,
see Theorem \ref{thm:nonComplCG}.
Since $R^\wedge$ is $(R_0)$,
we see $R^\wedge$ is finite and generically \'etale over $\kappa[[t]]$
as well.

We fix a perfect closure $\kappa_0^\mathrm{perf}$
and denote by $\kappa_1,\kappa_2,\ldots$
the finite purely inseparable extensions of $\kappa_0$ inside $\kappa_0^\mathrm{perf}$.
For a $\kappa_1$, denote by $R_1$ the ring $R\otimes_{\kappa_0}\kappa_1$,
so $R_1$ is a Noetherian local ring with residue field $k\otimes_{\kappa_0}\kappa_1$
where $k$ is the residue field of $R$.
Let $\Tilde{R}$
be the local ring $R\otimes_{\kappa_0}\kappa_0^{\mathrm{perf}}$,
and let $R^*$ be the ring
$R^\wedge\otimes_{\kappa[[t]]} \kappa^{\mathrm{perf}}[[t]]$.
Since $\Lambda$ is a $p$-basis of $\kappa$,
we have $\kappa^{\mathrm{perf}}=\kappa\otimes_{\kappa_0}\kappa_0^{\mathrm{perf}}$,
so we have canonical maps
$R_1\to \Tilde{R}\to R^*$.
Note that $R^*$ is finite and generically \'etale over $\kappa^{\mathrm{perf}}[[t]]$,
hence is complete, Noetherian, and $(R_0)$.
The map $R_1\to R^*$ is faithfully flat,
and $\Tilde{R}$ is the union of all such rings $R_1$,
so $\fa=\fa R^*\cap \Tilde{R}$ for every ideal $\fa$ of $\Tilde{R}$.
Thus $\Tilde{R}$ is Noetherian,
and it is clear that $\Tilde{R}^\wedge=R^*$.
By Lemma \ref{lem:formalR0},
the normalization of $\Tilde{R}$ is finite,
so we can find a $\kappa_1$
such that $R_1^\nu\otimes_{\kappa_1}\kappa_0^{\mathrm{perf}}=\Tilde{R}^\nu$,
hence for all $\kappa_2/\kappa_1$,
we have $R_1^\nu\otimes_{\kappa_1}\kappa_2=R_2^\nu$.

For a $\kappa_2/\kappa_1$,
consider the quantity $\lambda(\kappa_2)=l_{R_2^\nu}(R_2^\nu/tR_2^\nu)$.
Then
$\lambda(\kappa_2)\leq l_{R_2}(R_2^\nu/tR_2^\nu)$,
and since $R_1^\nu\otimes_{\kappa_1}\kappa_2=R_2^\nu$,
this latter quantity is equal to
$l_{R_1}(R_1^\nu/tR_1^\nu)$.
Thus the quantities $\lambda(\kappa_2)$ are bounded,
so we may take a $\kappa_2/\kappa_1$ that achieves the maximal $\lambda(\kappa_2)$.
We claim that
$\kappa_2/\kappa_0$ is what we want.
Condition \ref{condit:tame}\ref{tame_R0} is clear since $R_2\to R^*$ is faithfully flat; we need the rest two items.

Note that $R\to R_2$ is finite and radicial,
so for each minimal prime $\fp_2$ of $R_2$,
$\fp_2\cap R$ is a minimal prime $R$
and we have
$(R/\fp_2\cap R)^{sh}\otimes_{R/\fp_2\cap R} R_2/\fp_2=(R_2/\fp_2)^{sh}$.
Since $R/\fp_2\cap R$ is geometrically unibranch,
$R_2/\fp_2$ is geometrically unibranch, 
see \citestacks{06DM}.
In particular, Condition \ref{condit:tame}\ref{tame_Hensel} holds for $R_2$.

Now we show Condition \ref{condit:tame}\ref{tame_resSame} holds for $R_2$.
Let $\fm_1,\ldots,\fm_c$ be the maximal ideals of $R_2^\nu$.
For every $\kappa_3/\kappa_2$,
since $R_2^\nu\otimes_{\kappa_2}\kappa_3=R_3^\nu$,
$\fn_i=\sqrt{\fm_iR_3^\nu}$
are exactly the maximal ideals of $R_3^\nu$,
and
$(R_2^\nu)_{\fm_i}\otimes_{\kappa_2} \kappa_3= (R_3^\nu)_{\fn_i}$.
Thus we see
\begin{align*}
    l_{R_3^\nu}(R_3^\nu/tR_3^\nu)&=\sum_i l_{R_3^\nu}((R_3^\nu)_{\fn_i}/t(R_3^\nu)_{\fn_i})\\
    &=\sum_i\frac{1}{[\kappa(\fn_i):\kappa(\fm_i)]} l_{R_2^\nu}((R_3^\nu)_{\fn_i}/t(R_3^\nu)_{\fn_i})\\
    &=\sum_i\frac{[\kappa_3:\kappa_2]}{[\kappa(\fn_i):\kappa(\fm_i)]} l_{R_2^\nu}((R_2^\nu)_{\fm_i}/t(R_2^\nu)_{\fm_i})
\end{align*}
and we always have
\[
l_{R_2^\nu}(R_2^\nu/tR_2^\nu)=\sum_i l_{R_2^\nu}((R_2^\nu)_{\fm_i}/t(R_2^\nu)_{\fm_i}).
\]
For each $i$ we have ${[\kappa_3:\kappa_2]}\geq {[\kappa(\fn_i):\kappa(\fm_i)]}$
since $\kappa(\fn_i)$ is a quotient of $\kappa(\fm_i)\otimes_{\kappa_2}\kappa_3$.
Thus the maximality of $\lambda(\kappa_2)$ gives ${[\kappa_3:\kappa_2]}= {[\kappa(\fn_i):\kappa(\fm_i)]}$
and thus $\kappa(\fn_i)=\kappa(\fm_i)\otimes_{\kappa_2}\kappa_3$.
Since $\kappa_3/\kappa_2$ was arbitrary,
we must have $\kappa(\fm_i)$ separable over $\kappa_2$ for all $i$.
Since $R_2/\fp_2$ is geometrically unibranch for every minimal prime $\fp_2$ of $R_2$,
we see $\kappa(\fm_i)$ is purely inseparable over $\kappa_2$,
so $\kappa(\fm_i)=\kappa_2$,
which is Condition \ref{condit:tame}\ref{tame_resSame},
as desired.
\end{proof}

\begin{Prop}\label{prop:locallyfindtamecurve}
Let $(A,\fm,k)$ be a Noetherian local $\bF_p$-algebra of dimension $d$.
Let $\fa$ be a proper ideal of $A$.
Assume that all minimal primes of $\fa$ are of height $d-1$
and that $(A/\fa)^\wedge$ is $(R_0)$.

Then there exists a syntomic-local\footnote{This, by convention, means that $A\to B$ is a local map of local rings such that $B$ is the localization of a syntomic
$A$-algebra at a prime ideal.} ring map $A\to B$
such that $B/\fm B$ is finite over $k$ and that $\fa B$ defines a tame curve.

\end{Prop}
\begin{proof}
Any \'etale-local ring map $A/\fa\to E$ is syntomic-local
by \citestacks{00UE}, and
$E^\wedge$ is $(R_0)$
since it is \'etale over $(A/\fa)^\wedge$.
Take an $E$ such that
$E/\fp$ is geometrically unibranch for all minimal primes $\fp$ of $E$, cf. \citestacks{0CB4}.
By Lemma \ref{lem:curveistameafterFieldExtn},
there exists a finite syntomic $E$-algebra
$C$ that is local and 
satisfies Condition \ref{condit:tame}.
Note that $A/\fa\to C$ is also syntomic-local,
and $C/\fm C$ is finite over $k$.

By \citestacks{07M8},
we can lift $C$ to a syntomic-local $A$-algebra $B$.
By our choice, $B/\fa B=C$ satisfies Definition \ref{def:tamecurve}\ref{tamecurve_tame},
and $B/\fm B=C/\fm C$ is finite over $k$.
By flatness,
$\dim B=d$ and all minimal primes of $\fa B$ have height $d-1$,
giving Definition \ref{def:tamecurve}\ref{tamecurve_para}.
\end{proof}

\subsection{Local uniformity}\label{subsec:uniftame}
The goal of this subsection is to prove the following statement.

\begin{Thm}\label{thm:uniftame}
Let $R$ be a Noetherian ring, $\fp\in\Spec(R)$, $d=\HT\fp$.
Let $\fa\subseteq \fp$ be an ideal of $R$ such that $\fa R_\fp$ that defines a tame curve (Definition \ref{def:tamecurve}).
Then, upon replacing $R$ by $R_g$ for some $g\not\in\fp$,
the following hold.

\begin{enumerate}[label=$(\roman*)$]
    \item\label{uniftame_height} For all $\fP\in V(\fp)$, $\HT(\fP)=d+\HT(\fP/\fp)$.
    \item\label{uniftame_tameeverywhere} For all $\fP\in V(\fp)$ such that $R_\fP/\fp R_\fP$ is regular,
    and all sequence of elements $\pi_1,\ldots,\pi_h\in \fP R_\fP$
    that maps to a regular system of parameters of $R_\fP/\fp R_\fP$,
    $\fA:=\fa R_\fP +(\underline{\pi})$ defines a tame curve.
    \item\label{uniftame_numerical} Notations as in \ref{uniftame_tameeverywhere} and Notation \ref{notns:tameinvariants}.
    If $R/\fa$ contains $\bF_p$,
    then for $\fP,\underline{\pi}$
as in \ref{uniftame_tameeverywhere},
    we have $\delta(R_\fP/\fA)=\delta(R_\fp/\fa)$
    and $\Delta(R_\fP/\fA)= \Delta(R_\fp/\fa)$.
\end{enumerate}
\end{Thm}

Item \ref{uniftame_height} follows from \cite[Proposition 6.10.6]{EGA4_2}
and does not rely on the existence of $\fa.$
Before going into the proof of \ref{uniftame_tameeverywhere} and \ref{uniftame_numerical}, we note the following.

\begin{Discu}[cf. {\cite[Lemmas 3.2 and 3.3]{EY-splittingnumbers}}]\label{discu:localize}
Let $R$ be a Noetherian ring, $\fp\in\Spec(R)$.
Let $M$ be a finite $R$-module.
Then, upon replacing $R$ by $R_g$ for some $g\not\in\fp$,
there is a filtration $M=M_n\supsetneq M_{n-1}\supsetneq\ldots\supsetneq M_1\supsetneq M_0=0$
such that $M_j/M_{j-1}\cong R/\fp_j$ where $\fp_j\subseteq\fp$.
In particular, if $M_\fp$ is of finite length,
then $M$ is a successive extension of $R/\fp$.
Thus if $\fP\in V(\fp)$, $\pi_1,\ldots,\pi_h$ elements of $R_\fP$ that are a regular sequence in $R_\fP/\fp R_\fP$,
then $\pi_1,\ldots,\pi_h$ is a regular sequence in $M_\fP$.
Consequently, if 
\[
\begin{CD}
0@>>> N_1@>>> N_2@>>> M@>>> 0
\end{CD}
\]
is a short exact sequence of $R$-modules,
then
\[
\begin{CD}
0@>>> (N_1)_\fP/(\underline{\pi})@>>> (N_2)_\fP/(\underline{\pi})@>>> M_\fP/(\underline{\pi})@>>> 0
\end{CD}
\]
is exact. 
Moreover, if $h=\dim R_\fP/\fp R_\fP$
(so in particular $R_\fP/\fp R_\fP$ is Cohen--Macaulay),
then looking at the prime filtration we see
\[
l(M_\fP/(\underline{\pi}))=l(M_\fp)l(R_\fP/(\fp R_\fP +(\underline{\pi}))).
\]
\end{Discu}

Now we continue the proof of Theorem \ref{thm:uniftame}.

\begin{StepLocalUnif}
Let $\fq_1,\ldots,\fq_m$ be the minimal primes of $\fa$.
Localize $R$,
we may assume $\fq_i\subseteq\fp$ for all $i$.
\end{StepLocalUnif}

\begin{StepLocalUnif}
For each $i$, the normalization of $R_\fp/\fq_i R_\fp$ is finite (Lemma \ref{lem:formalR0}).
Thus there exists a finite extension $R'_i$ of $R/\fq_i$ in its fraction field such that $(R'_i)_\fp=(R_\fp/\fq_i R_\fp)^\nu$.
\end{StepLocalUnif}

\begin{StepLocalUnif}
By Condition \ref{condit:tame}\ref{tame_Hensel},
$(R'_i)_\fp$ is local,
so $R'_i$ has exactly one prime $\fp'_i$ above $\fp$.
Localizing $R$ we may assume $\fp'_i=\sqrt{\fp R'_i}$. 
By Condition \ref{condit:tame}\ref{tame_resSame},
$(R'_i)_\fp/\fp'_i(R'_i)_\fp=\kappa(\fp)$,
so after localizing $R$
we may assume $R/\fp=R'_i/\fp'_i$.
In particular, for each $\fP\in V(\fp)$,
there is a unique prime $\fP'_i$ of $R'_i$ above $\fP$,
and $R_\fP/\fp R_\fP=(R'_i)_{\fP'_i}/\fp'_i (R'_i)_{\fP'_i}$,
in particular $\kappa(\fP'_i)=\kappa(\fP)$.

\end{StepLocalUnif}

\begin{StepLocalUnif}\label{step:Normalizationisregular}
Since $(R'_i)_\fp$ is a DVR, after localizing $R$
we may assume that $\fp'_i$ is a principal ideal.
Let $\tau_i$ be a generator,
so $R/\fp=R'_i/\tau_iR'_i$.
For $\fP,\underline{\pi}$
as in \ref{uniftame_tameeverywhere},
$\tau_i,\underline{\pi}$ is then a regular sequence in $(R'_i)_{\fP'_i}=(R'_i)_{\fP}$
that generates the maximal ideal.
Thus $(R'_i)_{\fP}$ is regular and $\tau_i,\underline{\pi}$ is a regular system of parameters.
In particular, $\underline{\pi}$ is a regular sequence in $(R'_i)_{\fP}$ and
$(R'_i)_{\fP}/(\underline{\pi})$ is a DVR.
\end{StepLocalUnif}

\begin{StepLocalUnif}\label{step:Normalizationmodeachqi}
Apply Discussion \ref{discu:localize}
to $M=\frac{R'_i}{R/\fq_i}$,
we see that after localizing $R$,
we may assume that for all $\fP,\underline{\pi}$, $R_\fP/(\fq_iR_\fP +(\underline{\pi}))\to (R'_i)_{\fP}/(\underline{\pi})$
is injective with finite length cokernel.
Since $(R'_i)_{\fP}/(\underline{\pi})$ is a DVR (Step \ref{step:Normalizationisregular}),
it is the normalization of the integral domain
$R_\fP/(\fq_iR_\fP +(\underline{\pi}))$.
In particular,
$(\fq_iR_\fP +(\underline{\pi}))$ is a prime ideal,
and $\dim R_\fP/(\fq_iR_\fP +(\underline{\pi}))=1.$
\end{StepLocalUnif}

\begin{StepLocalUnif}\label{step:qiqjnotthesame}
Apply Discussion \ref{discu:localize}
to $M={R/(\fq_i+\fq_j)}\ (i\neq j)$,
we see that after localizing $R$,
we may assume that for all $\fP,\underline{\pi}$,
$R_\fP/(\fq_i R_\fP+\fq_j R_\fP+(\underline{\pi}))$
has finite length.
Thus $\fq_iR_\fP +(\underline{\pi})\neq \fq_jR_\fP +(\underline{\pi})$.
\end{StepLocalUnif}

\begin{StepLocalUnif}\label{step:modaminprimes}
Apply Discussion \ref{discu:localize}
to $M=\frac{\oplus_i R/\fq_i}{R/\sqrt{\fa}}$,
we see that after localizing $R$,
we may assume that for all $\fP,\underline{\pi}$, $\cap_i(\fq_iR_\fP +(\underline{\pi}))=\sqrt{\fa} R_\fP+(\underline{\pi})$.
Thus 
$\fq_iR_\fP +(\underline{\pi})$ are precisely all the minimal primes of $\fa R_\fP+(\underline{\pi})$.
\end{StepLocalUnif}

\begin{StepLocalUnif}
Apply Discussion \ref{discu:localize}
to $M=\sqrt{\fa}/\fa$,
we see that after localizing $R$,
we may assume that for all $\fP,\underline{\pi}$,
$\frac{\sqrt{\fa} R_\fP+(\underline{\pi})}{\fa R_\fP+(\underline{\pi})}$
has finite length.
Thus $R_\fP/(\fa R_\fP+(\underline{\pi}))$ is $(R_0)$.

\end{StepLocalUnif}

At this point,
with the characterization of minimal primes and normalizations
in the previous steps,
and with Lemma \ref{lem:formalR0},
we conclude that for all $\fP,\underline{\pi}$,
the ring
$R_\fP/(\fa R_\fP+(\underline{\pi}))$ has dimension $1$ and satisfies
Condition \ref{condit:tame}.
To see $\fa R_\fP+(\underline{\pi})$ defines a tame curve,
we must show $\HT(\fq_iR_\fP +(\underline{\pi}))=\HT(\fP)-1$ for all $i$.

By what we have done in Steps \ref{step:Normalizationisregular} and \ref{step:Normalizationmodeachqi},
$\underline{\pi}$ is a regular sequence in both $(R'_i)_{\fP}$ and $\left(\frac{R'_i}{R/\fq_i}\right)_\fP$.
Thus $\underline{\pi}$ is a regular sequence in $R_\fP/\fq_i R_\fP$.
Therefore $\HT(\fq_iR_\fP +(\underline{\pi}))\geq \HT(\fq_i R_\fP)+h$,
where $h=\HT(\fP/\fp)$.
Since $\HT(\fq_i R_\fP)=\HT(\fq_i R_\fp)=d-1$ (Definition \ref{def:tamecurve}\ref{tamecurve_para}), 
we see $\HT(\fq_iR_\fP +(\underline{\pi}))\geq d+h-1$.
Since $\HT(\fP)=d+h$ (by \ref{uniftame_height}),
we see that $\HT(\fq_iR_\fP +(\underline{\pi}))=d+h-1$,
thus $\fa R_\fP +(\underline{\pi})$ defines a tame curve.

It remains to show the agreement of $\delta$ and $\Delta$, assuming $R/\fa$ contains $\bF_p$.
By definition (Notation \ref{notns:tameinvariants}),
it suffices to show, for each $i$,
that $\beta(\overline{\fq}_i)=\beta(\fq_i R_\fp/\fa R_\fp)$ and same for $\gamma_0$.
Here $\overline{\fq}_i$ denotes 
$\frac{\fq_iR_\fP +(\underline{\pi})}{\fa R_\fP +(\underline{\pi})}$.

\begin{StepLocalUnif}
Fix an index $i$.
Let $\fr_i=\cap_{j\neq i}\fq_j$,
so $\beta(\fq_i R_\fp/\fa R_\fp)=l_{R_\fp}(M_\fp)$,
where $M$ is the finite $R$-module $R'_i/\fr_i R'_i$.
($R_\fp/\fa R_\fp$ satisfies Condition \ref{condit:tame}\ref{tame_resSame}
so we can calculate the length over $R_\fp$.)
As in Step \ref{step:modaminprimes},
after localizing $R$ we may assume
for all $\fP,\underline{\pi}$,
$\fr_iR_\fP+(\underline{\pi})=\cap_{j\neq i} \left(\fq_j R_\fP+(\underline{\pi})\right).$
Thus $\beta(\overline{\fq}_i)=l(M_\fP/(\underline{\pi})M_\fP)$.
Apply Discussion \ref{discu:localize},
we see $\beta(\overline{\fq}_i)=\beta(\fq_i R_\fp/\fa R_\fp)$ after localizing $R$ again.
\end{StepLocalUnif}

\begin{StepLocalUnif}\label{step:Conductor}
Again, fix an index $i$.
Let $\fc_i=\{a\in R\mid aR'_i\subseteq R/\fq_i\}$, the conductor of $R'_i$ over $R/\fq_i$ computed in $R$,
and let $M=\frac{R'_i}{R/\fq_i}$.
If $x_1,\ldots,x_l$ generate $M$ as an $R$-module,
then we have an injection
\begin{align*}
R/\fc_i&\to M^{\oplus l}\\
a&\mapsto (ax_1,\ldots,ax_l)
\end{align*}
The cokernel of this map has finite length at $\fp$ since $M$ does.
Apply Discussion \ref{discu:localize},
we see that after localizing $R$,
$\fc_iR_\fP+(\underline{\pi})$ is the conductor of the normalization over $R_\fP/(\fq_iR_\fP +(\underline{\pi}))$ computed in $R_\fP$.
We have $\gamma_0(\fq_i R_\fp/\fa R_\fp)=l_{R_\fp}((R'_i/\fc_i R'_i)_\fp)$ and similar for $\gamma_0(\overline{\fq}_i)$.
Thus applying Discussion \ref{discu:localize} again,
we see that after localizing $R$ again,
$\gamma_0(\fq_i R_\fp/\fa R_\fp)=\gamma_0(\overline{\fq}_i)$.
\end{StepLocalUnif}
The proof of Theorem \ref{thm:uniftame} is now finished.

\subsection{Uniform Cohen--Gabber}\label{subsec:UnifCG}
We arrive at the main results of the section.

\begin{Prop}\label{prop:UnifCohenGabberConstrNbhd}
Let $R$ be a Noetherian $\bF_p$-algebra,
$\fp\in\Spec(R)$, $d=\HT(\fp)\geq 1$.
Assume the following.
\begin{enumerate}
    \item $R/\fp$ is J-2.
    \item\label{unifCG_geomR0formalfib} There exist elements $a_1,\ldots,a_{d-1}\in \fp R_\fp$ such that
all minimal primes of $(\underline{a})$ are of height $d-1$
and that $(R_\fp/(\underline{a}))^\wedge$ is $(R_0)$.
    \end{enumerate}
Then there exist constants $\delta,\mu,\Delta,\rho\in\bZ_{\geq 0}$, 
a quasi-finite syntomic ring map $R\to S$,
and an $f\not\in \fp$,
such that for all $\fP\in V(\fp)\cap D(f)$, 
there exists $\fQ\in \Spec(S)$ above $\fP$
and a ring map $P\to S^\wedge_\fQ$
that satisfy the following.
\begin{enumerate}[label=$(\roman*)$]
    \item\label{UnifCGP_PisP} $(P,\fm_P)$ is a formal power series ring over a field. 
    \item\label{UnifCGP_samek}
    $P/\fm_P=\kappa(\fQ)$.
    \item\label{UnifCGP_degree} $P\to S^\wedge_\fQ$ is finite and generically \'etale of generic degree $\leq \delta$.
    \item\label{UnifCGP_embdim} $\fQ S^\wedge_\fQ/\fm_P S^\wedge_\fQ$ is generated by at most $\mu$ elements.
    \item\label{UnifCGP_disc} There exist $e_1,\ldots,e_n\in S^\wedge_\fQ$ that map to a basis of $S^\wedge_\fQ\otimes_P\Frac(P)$,
    such that $\Disc_{S^\wedge_\fQ/P}(e_1,\ldots,e_n)\not\in \fm_P^{\Delta+1}$.
    \item\label{UnifCGP_relembdim}
    There exists a regular system  of parameters $T_1,\ldots,T_N$ of $P$ such that for all $j>\rho$ the image of $T_j$ in $S^\wedge_\fQ$ lies in the image of $R$.
\end{enumerate}
\end{Prop}
Note that if $\dim R<\infty$ then \ref{UnifCGP_relembdim} is vacuous when we take $\rho=\dim R$.
\begin{proof}
Let $R_\fp\to B$ be as in Proposition \ref{prop:locallyfindtamecurve} for the ideal $\fa=(\underline{a})$.
$B$ is a syntomic-local $R_\fp$-algebra,
and $B/\fp B$ is finite over $\kappa(\fp)$, thus $B=S_\fq$ for some quasi-finite syntomic $R$-algebra $S$ and some $\fq\in\Spec(S)$.
We also have $\fa B=\fb_0 B$
for some ideal $\fb_0$ of $S$ generated by $d-1$ elements.

Since $R/\fp$ is J-2, we can localize $S$ near $\fq$ to assume $S/\fq$ regular.
Let $\delta=\delta(B/\fb_0B),\Delta=\Delta(B/\fb_0B)$ (Notation \ref{notns:tameinvariants}),
and let $\mu$ be the number of generators of $\fq/\fb_0$.
Find $g\not\in\fq$ as in Theorem \ref{thm:uniftame} (for $R=S$, $\fa=\fb_0$, and $\fp=\fq$), and replace $S$ by $S_g$.
Then for all $\fQ\in V(\fq),$
there exists an ideal $\fB$ of $S_\fQ$ generated by $\HT(\fQ)-1$ elements 
defining a tame curve (Theorem \ref{thm:uniftame}\ref{uniftame_height}\ref{uniftame_tameeverywhere})
and satisfying $\delta(S_\fQ/\fB)=\delta$
and $\Delta(S_\fQ/\fB)=\Delta$ (Theorem \ref{thm:uniftame}\ref{uniftame_numerical}).
The form of $\fB$ as in Theorem \ref{thm:uniftame}\ref{uniftame_tameeverywhere} tells us that $\fQ/\fB$ is generated by at most $\mu$ elements.

Let $P\to S^\wedge_\fQ$ be a 
map as in Theorem \ref{thm:TameCurveNormalization},
so \ref{UnifCGP_samek} is true by construction
and \ref{UnifCGP_PisP}\ref{UnifCGP_degree}\ref{UnifCGP_disc} follow from the theorem.
By construction, $\fB S^\wedge_\fQ\subseteq \fm_P S^\wedge_\fQ$,
so we get \ref{UnifCGP_embdim}.
To get \ref{UnifCGP_relembdim},
notice that the system of parameters $\underline{\pi}$ in Theorem \ref{thm:uniftame}\ref{uniftame_tameeverywhere},
which becomes part of the generators of the ideal defining the tame curve, and thus part of the image of a regular system of parameters of $P$ in Theorem \ref{thm:TameCurveNormalization},
is arbitrary.
Therefore, applying Lemma \ref{lem:RelativeCotangentSpace} to the map $R\to S/\fq$,
we see that if $S/\fq$ is generated by $\rho_0$ elements as an $R$-algebra,
then $\rho=\rho_0+\HT(\fq)$ works.

Finally, we let $\cC$ be the image of $V(\fq)$ in $\Spec(R)$, 
which is constructible since $R\to S$ is of finite type,
and contains $\fp$ as $\fq\cap S=\fp$.
So we can take $f$ such that $V(\fp)\cap D(f)\subseteq \cC$.
This finishes the proof.
\end{proof}

\begin{Rem} \label{rem:Sreduced-CONSTRnbhd}
    In Proposition \ref{prop:UnifCohenGabberConstrNbhd},
    if $R$ is locally equidimensional, universally catenary, $(S_1)$, and has $(S_1)$ formal fibers (e.g. if $R$ is Cohen--Macaulay), 
    then $S^\wedge_\fQ$ is torsion-free over $P$ (hence reduced by \ref{UnifCGP_degree}).
    Similarly for $S^\wedge_\fq$ in Theorem \ref{thm:UnifCohenGabber} as the construction there is from Proposition \ref{prop:UnifCohenGabberConstrNbhd}.
    Indeed,
 $S_\fQ$ is equidimensional
(Lemma \ref{lem:locequidAscent} 
since $R\to S$ is flat quasi-finite)
and universally catenary,
hence $S_\fQ^\wedge$ is equidimensional.
Thus all minimal primes of $S_\fQ^\wedge$ are above $(0)\subseteq P$.
Moreover the formal fibers of every finite type $R$-algebra is $(S_1)$ \citetwostacks{0BIA}{0BIY},
and the fibers of $R\to S$ are $(S_1)$,
so
$S_\fQ^\wedge$ is $(S_1)$
by \citetwostacks{0BIY}{0339}.
Therefore all associated primes of $S_\fQ^\wedge$ are above $(0)\subseteq P$, as desired.
\end{Rem}

Consider the following condition on a Noetherian ring $R$.
\begin{Condit}\label{condit:excellence}\

\begin{enumerate}[label=$(\roman*)$]
    \item\label{excel_J2} $R$ is J-2.
    
    \item\label{excel_R0} For all primes $\fp'\subset \fp$ of $R$ with $\HT(\fp/\fp')=1$,
    $(R_\fp/\fp' R_\fp)^\wedge$ is $(R_0)$.

\end{enumerate}
\end{Condit}

\begin{Rem}\label{rem:excelRedIsOkay}
A quasi-excellent $R$, or more generally a J-2 and Nagata $R$,
satisfies Condition \ref{condit:excellence}.
See \citestacks{0BI2}.
\end{Rem}
The global version of Proposition \ref{prop:UnifCohenGabberConstrNbhd} is as follows.
\begin{Thm}\label{thm:UnifCohenGabber}
Let $R$ be a Noetherian $\bF_p$-algebra
that satisies Condition \ref{condit:excellence}.

Assume $R$ is $(R_0)$.
Then there exist constants $\delta,\mu,\Delta,\rho\in\bZ_{\geq 0}$ depending only on $R$,
and a quasi-finite, syntomic ring map $R\to S$,
such that for all $\fp\in\Spec(R)$,
there exist a $\fq\in\Spec(S)$ above $\fp$
and a ring map $P\to S^\wedge_\fq$
that satisfy the following.
\begin{enumerate}[label=$(\roman*)$]
    \item $(P,\fm_P)$ is a formal power series ring over a field. 
    \item
    $P/\fm_P=\kappa(\fq)$.
    \item $P\to S^\wedge_\fq$ is finite and generically \'etale of generic degree $\leq \delta$.
    \item $\fq S^\wedge_\fq/\fm_P S^\wedge_\fq$ is generated by at most $\mu$ elements.
    \item\label{UnifCG:disc} There exist $e_1,\ldots,e_n\in S^\wedge_\fq$ that map to a basis of $S^\wedge_\fq\otimes_P\Frac(P)$
    (as an $\Frac(P)$-vector space),
    such that $\Disc_{S^\wedge_\fq/P}(e_1,\ldots,e_n)\not\in \fm_P^{\Delta+1}$.
    \item 
    There exists a regular system $T_1,\ldots,T_N$ of parameters of $P$ such that for all $j>\rho$ the image of $T_j$ in $S^\wedge_\fq$ lies in the image of $R$.
\end{enumerate}
\end{Thm}
\begin{proof}
    As the constructible topology is compact, 
we can take a finite product of algebras $S$ as in Proposition \ref{prop:UnifCohenGabberConstrNbhd} with the corresponding sets $V(\fp)\cap D(f)$ covering $\Spec(R)$,
and take the maximum of the constants.
We just need to verify assumption (\ref{unifCG_geomR0formalfib}) in Proposition \ref{prop:UnifCohenGabberConstrNbhd} for every non-minimal $\fp$,
as, for a minimal $\fp$,
as $R$ is J-2 and $(R_0)$,
there is some $f\not\in \fp$ such that $R_f$ is regular,
so we can just take $S=R_f,\delta=1$, $\mu=\Delta=0$.

Assume $d=\HT(\fp)\geq 1.$
As $R_\fp$ is J-2,
by Lemma \ref{lem:R0Bertini},
there exist elements $a_1,\ldots,a_{d-1}\in \fp R_\fp$ such that
all minimal primes of $(\underline{a})$ are of height $d-1$
and that $R_\fp/(\underline{a})$ is $(R_0)$.
Then $\left( R_\fp/(\underline{a})\right)^\wedge$ is $(R_0)$ by Condition \ref{condit:excellence}\ref{excel_R0}.
\end{proof}

There is a modified version useful for $F$-signature of pairs and triples when $\dim R=\infty$.
\begin{Thm}\label{thm:UnifCohenGabberWithc}
Notations as in Theorem  \ref{thm:UnifCohenGabber}.
Assume that $R$ is a catenary normal domain and let $0\neq c\in R$.

Then there exists $\nu\in \bZ_{\geq 0}$ such that, for a suitable choice of $S$ (depending on $c$),
we may choose $\fq$ and $P\to S_\fq^\wedge$ such that $cS_\fq^\wedge\cap P\not\subseteq\fm_P^{\nu+1}$.
\end{Thm}
\begin{proof}
    Let us fix $\nu\in \bZ_{\geq 1}$ such that $\sqrt{cR}^\nu\subseteq cR$.
    We will show that this $\nu$ works.

    For any $\fp\in D(c)$ and any choice of $S,\fq,P\to S_\fq^\wedge$ we have $cS_\fq^\wedge\cap P\not\subseteq\fm_P$.
    Consider $\fp\in V(c)$.
    Following the construction as in the proof of Theorem \ref{thm:UnifCohenGabber},
    we can choose any elements $a_1,\ldots,a_{d-1}$ as in Lemma \ref{lem:R0Bertini}.
    Since $R_\fp$ is a catenary normal local ring, by \citestacks{0BWZ}
    we may choose $a_1\in (\sqrt{cR})_\fp$,
    so after localization near $\fp$ we can choose $a_1\in \sqrt{cR}$ and thus $a_1^\nu\in cR$.
    Tracing back the constructions in Proposition \ref{prop:UnifCohenGabberConstrNbhd},
    which goes back to Theorems \ref{thm:uniftame} and \ref{thm:TameCurveNormalization},
    we see $a_1\in S_{\fQ}^\wedge$ is the image of an element $X_1$ in $\fm_P\setminus\fm_P^2$.
    We then have $X_1^\nu\in cS_{\fQ}^\wedge\cap P$ and $X_1^\nu\not\in\fm_P^{\nu+1}$.
    Since our $S$ is obtained from taking a finite product of the $S$'s as in Proposition \ref{prop:UnifCohenGabberConstrNbhd},
    we get the desired result.
\end{proof}

\section{More preliminaries}\label{sec:Prelim2}
\subsection{Local equidimensionality}

\begin{Lem}\label{lem:locequidAscent}
Let $R$ be a Noetherian ring that is locally equidimensional and universally catenary.
Let $R\to S$ be a flat ring map of finite type.
If all nonempty generic fibers of $R\to S$ are equidimensional and have the same demension,
then $S$ is locally equidimensional.
\end{Lem}
\begin{proof}
Let $d$ be the generic fiber dimension.
Let $\fq\in\Spec(S)$ be above some $\fp\in\Spec(R)$.
Let $\fq_0$ be an arbitrary minimal prime of $S$ contained in $\fq$,
lying over $\fp_0\in\Spec(R)$.
Then $\fp_0$ is a minimal prime of $R$ by flatness.

By \citestacks{02IJ}, 
$\HT(\fq/\fq_0)=\HT(\fp/\fp_0)+\operatorname{trdeg}_{\kappa(\fp_0)}{\kappa(\fq_0)}-\operatorname{trdeg}_{\kappa(\fp)}{\kappa(\fq)}$.
By our assumptions,
$\operatorname{trdeg}_{\kappa(\fp_0)}{\kappa(\fq_0)}=d$ is independent of the choice of $\fq_0$.
Also, $\HT(\fp/\fp_0)$ does not depend on the choice of $\fq_0$
since $R$ is locally equidimensional.
Thus $\HT(\fq/\fq_0)$ does not depend on the choice of $\fq_0$,
so $S$ is locally equidimensional. 
\end{proof}

\subsection{Number of generators}
\begin{Lem}\label{lem:generators}
Let $P$ be a normal domain, $A$ a finite torsion-free $P$-algebra of generic degree $\delta\geq 1$.
Let $\mu\in\bZ_{\geq 0}$ be such that $A$ is generated by $\mu$ elements as a $P$-algebra.

Assume that $A\otimes_P \Frac(P)$ is a product of fields.
Then $A$ is generated by at most $\delta^\mu$ elements as a $P$-module.
\end{Lem}
\begin{proof}
Let $a\in A$.
In each factor of $A\otimes_P \Frac(P)$, $a$ has a 
monic minimal polynomial whose coefficients are in $P$
since $P$ is normal.
Since $A\to A\otimes_P \Frac(P)$ is injective,
the product of these minimal polynomials is a 
monic polynomial of degree $\leq \delta$ with coefficients in $P$ that has $a$ as a root.
The rest is clear.
\end{proof}

\subsection{Length estimates for hypersurfaces}

\begin{Lem}\label{lem:HKmulthypersurf}
Let $(P,\fm,k)$ be a regular local ring containing $\bF_p$. 
Let $d=\dim P$.
Let $n\in\bZ_{\geq 0}$, $F\in P$, $F\not\in\fm^{n+1}$.
Then for all $e\in\bZ_{\geq 0}$,
$l(P/((F)+\ppower{\fm}{e}))\leq np^{e(d-1)}.$
\end{Lem}
\begin{proof}
We may assume $F\in\fm^n$ and $k$ infinite.
Arguing as in \cite[(40.2)]{Nagata-local},
we can find a regular system of parameters $x_1,\ldots,x_d$ of $P$
such that 
$l(P/(F,x_2,\ldots,x_d))=n$
and that $F,x_2,\ldots,x_d$ is a regular sequence in $P$.
Then $l(P/((F)+\ppower{\fm}{e}))\leq l(P/((F)+\ppower{(x_2,\ldots,x_d)}{e})) =np^{e(d-1)}.$
\end{proof}

\begin{Lem}\label{lem:HKmulthypersurf-imbalanced}
Let $(P,\fm,k)$ be a regular local ring containing $\bF_p$. 
Let $x_1,\ldots,x_d$ be a regular system of parameters of $P$ and let $e_1,\ldots,e_d\in\bZ_{\geq 0}$.
Let $n\in\bZ_{\geq 0}$, $F\in P$, $F\not\in\fm^{n+1}$.
Then
$l(P/(F,x_1^{p^{e_1}},\ldots,x_d^{p^{e_d}}))\leq np^{E},$
where $E=\sum_i e_i + (n-1)\max_i{e_i}-n\min_i e_i$.
\end{Lem}
\begin{proof}
We may assume $F\in\fm^n$.
If $d\leq n$ we have $E\geq (d-1)\max_i e_i$ so the result follows from Lemma \ref{lem:HKmulthypersurf}.
If $d>n$, after reordering $\underline{x}$,
the image of $F$ in $\fm^n/\fm^{n+1}$ has nonzero coefficient for a monomial involving only $x_1,\ldots,x_n$ (or less),
so the image of $F$ in $P/(x_{n+1},\ldots,x_d)$ is still not in the $(n+1)$st power of the maximal ideal.
We therefore have, by Lemma \ref{lem:HKmulthypersurf},
$l(P/(F,x_1^{p^{e_1}},\ldots,x_n^{p^{e_n}},x_{n+1},\ldots,x_d))\leq np^{(n-1)\max_{i\leq n}e_i}$.
Thus $$l(P/(F,x_1^{p^{e_1}},\ldots,x_d^{p^{e_d}}))\leq np^{(n-1)\max_{i\leq n}e_i+\sum_{i>n}e_i}.$$
The result follows as $(n-1)\max_{i\leq n}e_i+\sum_{i>n}e_i\leq(n-1)\max_{i\leq n}e_i+\sum_{i}e_i -n\min_{i\leq n}e_i\leq E$.
\end{proof}

\subsection{The localization-completion trick}
We make use of \cite{Lyu-dual-complex-lift} to relax the excellence hypotheses of our results.
The ``excellent'' reader can skip this subsection,
except for its use in the appendix.

In \cite{Lyu-dual-complex-lift} we introduced
\begin{Def}
    Let $R$ be a Noetherian ring.
    An object $\omega\in D(R)$ is a \emph{pseudo-dualizing complex}
    for $R$ if $\omega\in D^b_{Coh}(R)$
    and $\omega_\fp$ is a dualizing complex for $R_\fp$ for all $\fp\in\Spec(R)$.
\end{Def}
In particular, $R$ is Gorenstein if and only if $R$ is a pseudo-dualizing complex
    for $R$.

The results in \cite{Lyu-dual-complex-lift} we use are as follows.
See \cite[Corollary 4.9, Corollary 4.7, and Theorem 1.3]{Lyu-dual-complex-lift}.
The item \ref{duCX:preserveBDD} is only used in the appendix.
\begin{Thm}\label{thm:dualCPLX}
Let $R$ be a Noetherian ring. The following hold.
    \begin{enumerate}[label=$(\roman*)$]
    \item\label{duCX:preserveBDD} For $\omega\in D^b_{Coh}(R)$ a pseudo-dualizing complex and  $M\in D^b_{Coh}(R)$,
    we have $R\Hom_R(M,\omega)\in D^b_{Coh}(R).$
    \item\label{duCX:UC}
     If $R$ admits a pseudo-dualizing complex,
     then $R$ is universally catenary.
        \item\label{duCX:formallift} If there exists an ideal $I$ of $R$ such that $R$ is $I$-adically complete and $R/I$ admits a pseudo-dualizing complex (for example Gorenstein),
        then $R$ admits a pseudo-dualizing complex.
    \end{enumerate}
\end{Thm}

\begin{Discu}\label{discu:localize-complete}
    Let $R$ be a Noetherian ring and $I$ an ideal of $R$.
    The $I$-adic completion $R'=\lim_t R/I^t$
    is a flat $R$-algebra.
    We have $R/I^t=R'/I^tR'$,
    therefore the closed subsets $V(I)$
    and $V(IR')$ are canonically homeomorphic;
    moreover,
    for $\fp\in V(I)$,
    $R_\fp^\wedge=R'^\wedge_{\fp R'}$.

    Assume that $R/I$ admits a pseudo-dualizing complex.
    Then $R'$ also does, by Theorem \ref{thm:dualCPLX}\ref{duCX:formallift}.
    In particular,
    $R'$ is universally catenary by
    Theorem \ref{thm:dualCPLX}\ref{duCX:UC},
    and $R'$ has Gorenstein formal fibers \citestacks{0AWY}.

    In particular, if $\fp\in\Spec(R)$ is such that $(R/\fp)_f$ is Gorenstein for some $f\not\in\fp$,
    then we can apply this discussion to $\lim_tR_f/\fp^tR_f$.
\end{Discu}

\section{Uniform bound}\label{sec:unifBD}

\subsection{Sequences of ideals}\label{subsec:idealSequence}
We introduce a framework of good sequences of ideals for which uniform convergence can be expected.
This eventually recovers the uniform convergence considered in \cite{PolstraUniform,PolstraTucker-semicont} and generalizes to the case of pairs and triples.

First, let us fix some notations on what we previously found (Theorem \ref{thm:UnifCohenGabber}).
Note that in addition to the conditions there we added torsion-freeness of $A$, which we will achieve later.
\begin{Def}\label{def:CohenGabberDisc}
    Let $(A,\fm,k)$ be a Noetherian local $\bF_p$-algebra.
    We say a pair $(P,D)$ is a \emph{Cohen--Gabber normalization} of $A$ if the following hold.
\begin{enumerate}
    \item $P=(P,\fm_P,k_P)$ is a regular local subring of $A$.
    \item\label{CG:samek} $k_P$ maps isomorphically onto $k$.
    \item $A$ is finite, generically \'etale, and torsion-free over $P$ (in particular $A$ is reduced).
    \item There exist $e_1,\ldots,e_n\in A$ that map to a basis of $A\otimes_P\Frac(P)$
    (as an $\Frac(P)$-vector space),
    such that $\Disc_{A/P}(e_1,\ldots,e_n)=D$.
\end{enumerate}
\end{Def}

\begin{Def}\label{def:SequencesofIdeals}
    Let $(A,\fm,k)$ be a Noetherian local $\bF_p$-algebra, $(I_e)_{e\in\bZ_{\geq 0}}$ a sequence of ideals of $A$.
    Let $c\in A,e_0,\rho\in\bZ_{\geq 0}$, and let $(P,D)$ be a Cohen--Gabber normalization of $A$.
    We say $(I_e)_e$ is a \emph{$p$- and $p^{-1}$-sequence of ideals with respect to $(P,D)$, $c$, $e_0$, and $\rho$} if the following hold.
    \begin{enumerate}
        \item\label{seq:cofinal} $\ppower{\fm}{e+e_0}\subseteq I_{e}$ for all $e\in\bZ_{\geq 0}$.
        \item\label{seq:AlmostMpe}
        There exists a regular system $x_1,\ldots,x_d$ of parameters of $P$ such that $x_i^{p^e}\in I_e$ for all $i>\rho$.
        \item\label{seq:pSeq} ($p$-sequence) $c\ppower{I}{}_e\subseteq I_{e+1}$ for all $e\in\bZ_{\geq 0}$.
        \item\label{seq:p-1Seq} ($p^{-1}$-sequence) The image of $F_*I_{e+1}$ under the composition map
        \[
        \begin{CD}
           F_*A@>{\times F_*c}>>  F_*A@>{\times D}>> F_* P\otimes_P A
        \end{CD}
        \]
        is contained in $F_*P\otimes_P I_e$ for all $e\in\bZ_{\geq 0}$.
        In other words, $DF_*(cI_{e+1})=F_*(cD^pI_{e+1})\subseteq F_*P\otimes_P I_e$ as subsets of $F_*A$.
    \end{enumerate}
\end{Def}
Here, the map $\times D$ is well-defined by Lemma \ref{lem:discKillsCoker}, whereas $F_*P\otimes_P I_e$ is indeed a subset of $F_* P\otimes_P A$ as $F_*P$ is flat over $P$.
Note that $F_* P\otimes_P A$ is a subring of $F_*A$ by the fact $P\to A$ is torsion-free and generically \'etale.\\

\begin{Thm}\label{thm:pandp-1SeqEstimate}
   Let $(A,\fm,k)$ be a Noetherian local $\bF_p$-algebra of dimension $d$, $(I_e)_{e\in\bZ_{\geq 0}}$ a sequence of ideals of $A$.   
    Let $c\in A,e_0,\rho\in\bZ_{\geq 0}$, and let $(P,D)$ be a Cohen--Gabber normalization of $A$.
Assume that $(I_e)_e$ is a $p$- and $p^{-1}$-sequence of ideals with respect to $(P,D)$, $c$, $e_0$, and $\rho$.
Assume that $cA\cap P\neq 0$.
Then for any nonzero $c'\in cA\cap P$ and all $e\leq e'\in\bZ_{\geq 0}$ we have
\[
\left|\frac{l_A(A/I_e)}{p^{ed}}-\frac{l_A(A/I_{e'})}{p^{e'd}}\right|\leq \frac{{m} np^{(\rho+n-1)e_0}}{p^{e}}.
\]
Here ${m}=l_P(A/\fm_PA)$ is the minimal number required to generate $A$ as a $P$-module, and 
$n$ is such that $c'D^p\not\in\fm_P^n$.

\end{Thm}
\begin{proof}
    From the definition of $p$- and $p^{-1}$-sequences we may assume $c'=c$.
    Since $p^{-1}+p^{-2}+\ldots\leq 1$ it suffices to show 
    \[
\left|\frac{l_A(A/I_e)}{p^{ed}}-\frac{l_A(A/I_{e+1})}{p^{(e+1)d}}\right|\leq \frac{{m} np^{(\rho+n-1)e_0}}{p^{e+1}}.
\]

First, examine the exact sequence which defines $N_1$
\[
\begin{CD}
    0@>>> F_*P\otimes_P A@>{\times F_*c}>> F_*A@>>> N_1@>>> 0,
\end{CD}
\]
where $N_1$ is annihilated by $F_*(cD^p)$ by Lemma \ref{lem:discKillsCoker},
and is generated by at most ${m} $ elements as an $F_*P$-module.
By Definition \ref{def:SequencesofIdeals}\eqref{seq:pSeq} we have $(F_*c)I_e\subseteq F_*I_{e+1}$, therefore we have an induced sequence where $\overline{N_1}$ is a quotient of $N_1$
\[
\begin{CD}
    F_*P\otimes_P (A/I_e)@>{\times F_*c}>> F_*(A/I_{e+1})@>>> \overline{N_1}@>>> 0.
\end{CD}
\]
By Definition \ref{def:SequencesofIdeals}\eqref{seq:cofinal}\eqref{seq:AlmostMpe}, we can find a regular system of parameters $x_1,\ldots,x_d$ of $P$ such that
$F_*(A/I_{e+1})$, hence $\overline{N_1}$, is annihilated by $F_*x_j^{p^{e+1+e_0}}$ for all $j$ and $F_*x_j^{p^{e+1}}$ for all $j>\rho$.
It now follows from Lemma \ref{lem:HKmulthypersurf} that 
\[
l_{F_*P}(\overline{N_1})\leq {m} np^{E},
\]
where $E\leq \rho (e+1+e_0) + (d-\rho)(e+1)+(n-1)(e+1+e_0)-n(e+1)=(d-1)(e+1)+(\rho+n-1)e_0$.
Therefore the exact sequence above gives
\[
l_{F_*P}(F_*(A/I_{e+1}))-l_{F_*P}\left(F_*P\otimes_P (A/I_e)\right)\leq {m} np^{(d-1)(e+1)+(\rho+n-1)e_0}.
\]
Note that $l_{F_*P}(F_*P\otimes_P -)=p^dl_P(-)$ as $F_*P$ is flat over $P$ and as this is true for $M=k_P$.
Also note that $k_P=k$ (Definition \ref{def:CohenGabberDisc}\eqref{CG:samek}) so $l_P(-)=l_A(-)$.
The above inequality can be simplified to
\[
l_{A}(A/I_{e+1})-p^dl_{A}(A/I_e)\leq {m} np^{(d-1)(e+1)+(\rho+n-1)e_0}.
\]
Dividing by $p^{(e+1)d}$, we get
\[-\frac{l_A(A/I_e)}{p^{ed}}+\frac{l_A(A/I_{e+1})}{p^{(e+1)d}}\leq \frac{{m} np^{(\rho+n-1)e_0}}{p^{e+1}}.\]

It remains to show the same inequality with the signs on the left hand side reversed.
For this purpose, consider the exact sequence which defines $N_2$
\[
\begin{CD}
0@>>>     F_*A@>{\times F_*(cD^p)}>>F_*P\otimes_P A@>>> N_2@>>> 0,
\end{CD}
\]
which is well-defined by Lemma \ref{lem:discKillsCoker}.
It is clear that $N_2$ is annihilated by $F_*(cD^p)$
and is generated by at most ${m} $ elements as an $F_*P$-module.
By Definition \ref{def:SequencesofIdeals}\eqref{seq:p-1Seq} we have $(F_*(cD^p)) F_*I_{e+1}\subseteq F_*P\otimes_P I_e$, therefore we have an induced sequence where $\overline{N_2}$ is a quotient of $N_2$
\[
\begin{CD}
    F_*(A/I_{e+1})@>{\times F_*(cD^p)}>> F_*P\otimes_P (A/I_e)@>>> \overline{N_2}@>>> 0.
\end{CD}
\]
By Definition \ref{def:SequencesofIdeals}\eqref{seq:cofinal}\eqref{seq:AlmostMpe},
we can find a regular system of parameters $x_1,\ldots,x_d$ of $P$ such that
$A/I_{e}$, hence $\overline{N_2}$, is annihilated by $x_j^{p^{e+e_0}}=F_*x_j^{p^{e+1+e_0}}$ for all $j$ and $x_j^{p^{e}}=F_*x_j^{p^{e+1}}$ for all $j>\rho$.
Therefore, the same calculation as above tells us
\[\frac{l_A(A/I_e)}{p^{ed}}-\frac{l_A(A/I_{e+1})}{p^{(e+1)d}}\leq \frac{{m} np^{(\rho+n-1)e_0}}{p^{e+1}},\]
as desired.
\end{proof}

\subsection{Ubiquity of $p$- and $p^{-1}$-sequences}\label{subsec:examplesofIdealSeq}
We examine three situations where $p$- and $p^{-1}$-sequences naturally appear, and the results in the previous section will therefore yield the desired bounds. 

\subsubsection{The case of colons}
\begin{Lem}\label{lem:Ie:Je}
Let $(A,\fm,k)$ be a Noetherian local $\bF_p$-algebra.
Let $I,J$ be ideals of $A$ such that $\ppower{\fm}{e_0}J\subseteq I$.
Then for every Cohen--Gabber normalization $(P,D)$ of $A$,
the sequence $K_e:=(\ppower{I}{e}:\ppower{J}{e})$ is a $p$- and $p^{-1}$-sequence with respect to $(P,D)$, $c=1$, $e_0$, and $\rho=\dim A$.
\end{Lem}
\begin{proof}
    As $\ppower{\fm}{e_0}J\subseteq I$, we have $\ppower{\fm}{e+e_0}\ppower{J}{e}\subseteq \ppower{I}{e}$, \emph{i.e.} $\ppower{\fm}{e+e_0}\subseteq K_e$.    
    Similarly we see $\ppower{K}{}_e\subseteq K_{e+1}$.
    We have now verified \eqref{seq:cofinal}\eqref{seq:pSeq} in Definition \ref{def:SequencesofIdeals}.
        \eqref{seq:AlmostMpe} is vacuous
    and only \eqref{seq:p-1Seq} remains.

    Let $x\in K_{e+1}$ so $x\ppower{J}{e+1}\subseteq \ppower{I}{e+1}$, or 
    $\ppower{J}{e}F_*(xA)\subseteq \ppower{I}{e}F_*A$.
    Applying the map $\times D$ in Definition \ref{def:SequencesofIdeals}\eqref{seq:p-1Seq},
    we see $DF_*x$ multiplies $\ppower{J}{e}F_*A$, hence $\ppower{J}{e}(F_*P\otimes_P A)$, into $\ppower{I}{e}(F_*P\otimes_P A)$.
    As $F_*P$ is faithfully flat over $P$, we see  $DF_*x\in F_*P\otimes_PK_e$, as desired.
\end{proof}

\begin{Rem}
    In applications (to Hilbert--Kunz multiplicity and $F$-signature) we can reduce to the case $\fm J\subseteq I$, then $e_0=0$, so the value of $\rho$ does not affect the bound in Theorem \ref{thm:pandp-1SeqEstimate}.
\end{Rem}

\subsubsection{The case of $F$-signature of pairs}
We need some setup.

\begin{Def}[{cf. \cite[Definition 2.8]{CRST-Bertini-F-signature}}]\label{def:SplittingIdealPairs}
    Let $(T,\fm_T,k_T)$ be a Noetherian local $\bF_p$-algebra.
    Assume that $T$ is normal and let $D$ be an effective Weil divisor on $T$.
    Consider the ideal \[ I_e^D = \{x\in T\mid T\xrightarrow{F^e_*x} F^e_*T(D)\text{ is not pure}\};\]
    equivalently, for  the injective hull $E_T$  of $k_T$,
\cite[Proposition 1.3(5)]{Fed83-complete}
tells us $F^e_*I_e^D$ is the $F^e_*T$-annihilator of the image of the socle element of $E_T$ in $F^e_*T(D)\otimes_T E_T$.
    The number $s_e(T,\Delta):=l_T(T/I_e^{\Delta_e})/p^{e\dim T}$ is the $e$th normalized $F$-splitting ratio of $(T,\Delta)$, where $\Delta$ is an effective $\bR$-Weil divisor and $\Delta_e=\lceil (p^e-1)\Delta\rceil$.
\end{Def}

\begin{Condit}\label{condit:Divisors}
    Let $T$ be a Noetherian normal domain. 
    Let $D_e$ be a sequence of effective Weil divisors on $T$.
    Let $E$ be an effective Weil divisor on $T$.
    Consider the condition
    \[
    pD_e-E\leq D_{e+1}\leq pD_e+E\text{ for all }e.
    \]
\end{Condit}

\begin{Lem}\label{lem:DivisorSeqBounded}
    Let $T$ be a Noetherian normal domain. 
    Let $\Delta=\sum a_i\Delta_i$ be an effective $\bR$-Weil divisor on $T$ where $\Delta_i$ are prime divisors.
    Then $\Delta_e=\lceil (p^e-1)\Delta\rceil$ satisfies Condition \ref{condit:Divisors} for  $E=\sum \max\{p,\lceil(p-1)a_i\rceil\}\Delta_i$.
\end{Lem}
\begin{proof}
    We must show, for all $a\in\bR_{>0}$,
    \[
    p\lceil(p^e-1)a\rceil-p\leq \lceil(p^{e+1}-1)a\rceil\leq p\lceil(p^e-1)a\rceil+\lceil(p-1)a\rceil.
    \]
The first inequality is because (noting $p\lceil b\rceil-p\leq \lceil pb\rceil$ for all $b\in\bR_{\geq 0}$)
\[
p\lceil(p^e-1)a\rceil-p\leq \lceil p(p^e-1)a\rceil\leq \lceil (p^{e+1}-1)a\rceil,
\]
and the second inequality is because
\[
\lceil(p^{e+1}-1)a\rceil\leq\lceil(p^{e+1}-p)a\rceil+\lceil(p-1)a\rceil\leq p\lceil(p^e-1)a\rceil+\lceil(p-1)a\rceil,
\]
as desired.
\end{proof}

An extra complication must be involved here as there is no reason to expect that the ring $S$ in Theorem \ref{thm:UnifCohenGabber} is normal assuming $R$ is. 
Therefore, we must consider $R$ and $S$ together.
\begin{Lem}\label{lem:SplittingIdealPairs}
    Let $(T,\fm_T,k_T)$ be a Noetherian local $\bF_p$-algebra.
    Assume that $T$ is normal and let $D_e$ be a sequence of effective Weil divisors on $T$
    satisfying Condition \ref{condit:Divisors} for some effective Weil divisor $E$.

    Let $(A,\fm,k)$ be a Noetherian local $T$-algebra such that the map $T\to A$ is flat and  local and has $(S_2)$ fibers.
Let $(P,D)$ be a Cohen--Gabber normalization of $A$.
    Let $c\in T(-E)$ and let $e_0,\rho\in\bZ_{\geq 0}$.
Assume 
\begin{enumerate}
    \item\label{pairs:0dimfiber} $\ppower{\fm}{e_0}\subseteq \fm_T A$.
    \item\label{pairs:relembdim} There exists a regular system of parameters $x_1,\ldots,x_d$ of $P$ such that the image of $x_j$ in $A$ lies in the image of $T$ for all $j>\rho$.
\end{enumerate}    
    Then the sequence $I_e^{D_e}A$ is a $p$- and $p^{-1}$-sequence  with respect to $(P,D)$, $c$, $e_0$, and $\rho$.
\end{Lem}
\begin{proof}
    First, by definition, we have $\ppower{\fm}{e}_T\subseteq I^{D_e}_e$.
    The assumption \eqref{pairs:0dimfiber} tells us $\ppower{\fm}{e+e_0}\subseteq I^{D_e}_eA$
    which is Definition \ref{def:SequencesofIdeals}\eqref{seq:cofinal}.
    Similarly the assumption \eqref{pairs:relembdim} gives  Definition \ref{def:SequencesofIdeals}\eqref{seq:AlmostMpe}.

    To get \eqref{seq:pSeq} in Definition \ref{def:SequencesofIdeals},
    it suffices to show $c\ppower{(I_e^{D_e})}{}\subseteq I_{e+1}^{D_{e+1}}$.
    Take $x\in I_e^{D_e}$ so $T\xrightarrow{\times F^e_*x}F^e_*T(D_e)$ is not pure.
    Note that there is a canonical map $F^e_*T(D_e)\to F^{e+1}_*T(pD_e)$ as for all $y\in\Frac(T)^\times$ we have $y\in T(D_e)$ if and only if $y^p\in T(pD_e)$.
    Therefore, we get several maps in which the leftmost is not pure
    \[
    T\xrightarrow{\times F^e_*x}F^e_*T(D_e)\to F^{e+1}_*T(pD_e)\xrightarrow{\times F^{e+1}_*c} F^{e+1}_*T(pD_e-E)\to F^{e+1}_*T(D_{e+1}),
    \]
    where the rightmost map is well-defined as $pD_e-E\leq D_{e+1}$.
    We conclude that the composition map is also not pure, in other words, $cx^p\in I_{e+1}^{D_{e+1}}$, as desired.

    To get \eqref{seq:p-1Seq} in Definition \ref{def:SequencesofIdeals}, first notice that the previously mentioned canonical map $F^e_*T(D_e)\to F^{e+1}_*T(pD_e)$
    induces a canonical map
    \begin{align*}
        F^{e+1}_*T\otimes_{F^e_*T}F^e_*T(D_e)\longrightarrow F^{e+1}_*T(pD_e)
    \end{align*}
    which is an isomorphism in codimension $1$.
    Applying the flat base change $F^{e+1}_*T\to F^{e+1}_*A$, we get a canonical map 
    \begin{align}\label{pairs:can}
        F^{e+1}_*A\otimes_{F^e_*T}F^e_*T(D_e)\longrightarrow F^{e+1}_*A\otimes_{F^{e+1}_*T}F^{e+1}_*T(pD_e)
    \end{align}
    which is an isomorphism in codimension $1$.
    On the other hand, Lemma \ref{lem:discKillsCoker} gives us a map
    \begin{align*}
        F^{e+1}_*A\xrightarrow{\times F^e_*D} F^{e+1}_*P\otimes_{F^{e}_*P}F^{e}_*A,
    \end{align*}
    which induces a map
    \begin{align}\label{pairs:timesDisc}
        F^{e+1}_*A\otimes_{F^{e}_*T}F^{e}_*T(D_e)\xrightarrow{\times F^e_*D} F^{e+1}_*P\otimes_{F^{e}_*P}F^{e}_*A\otimes_{F^e_*T}F^e_*T(D_e).
    \end{align}
\begin{Claim}
    The target of both \eqref{pairs:can} and \eqref{pairs:timesDisc} are $(S_2)$ as $F^{e+1}_*P$-modules.
\end{Claim} 

The proof of the claim is postponed to the end.
Granting this claim, the fact that  \eqref{pairs:can}  is an isomorphism in codimension $1$ tells us that we have a canonical map
\begin{align}\label{pairs:timesDiscUSEFUL}
        F^{e+1}_*A\otimes_{F^{e+1}_*T}F^{e+1}_*T(pD_e)\xrightarrow{\times F^e_*D} F^{e+1}_*P\otimes_{F^{e}_*P}F^{e}_*A\otimes_{F^e_*T}F^e_*T(D_e).
    \end{align}
Now take $x\in I_{e+1}^{D_{e+1}}$. We want to show $DF_*cx=F_*(cD^px)\in F_*P\otimes_P I_e^{D_e}A$.
 We have several maps in which the leftmost is not pure
    \[
    T\xrightarrow{\times F^{e+1}_*x}F^{e+1}_*T(D_{e+1})\xrightarrow{\times F^{e+1}_*c} F^{e+1}_*T(D_{e+1}-E)\to F^{e+1}_*T(pD_{e}),
    \]
where the rightmost map is well-defined as $D_{e+1}-E\leq pD_e$.
Hence the composition is also not pure.
Composing further with \eqref{pairs:timesDiscUSEFUL},
we see that the composition
 \[
    T\xrightarrow{\times F^{e+1}_*cx} F^{e+1}_*T(pD_{e})\to F^{e+1}_*A\otimes_{F^{e+1}_*T}F^{e+1}_*T(pD_e)\xrightarrow{\times F^e_*D} F^{e+1}_*P\otimes_{F^{e}_*P}F^{e}_*A\otimes_{F^e_*T}F^e_*T(D_e)
    \]   
    is not pure. This map can be rewritten as
\begin{align}\label{pairs:FinalMap}
    T\xrightarrow{\times F^{e+1}_*cD^px} F^e_*B\otimes_{F^e_*T}F^e_*T(D_e),
\end{align}
where $B=F_*P\otimes_{P}A$ is flat over $A$,
hence $F^e_*B$ is flat over $F^e_*T$.
Let $E_T$ be the injective hull of $k_T$. \eqref{pairs:FinalMap} not being pure is the same as the image of the socle element $u\in E_T$ in $F^e_*B\otimes_{F^e_*T}F^e_*T(D_e)\otimes_T E_T$
is annihilated by $F^{e+1}_*cD^px\in F^e_*B$ (see Lemma \ref{lem:discKillsCoker} for this ``$\in$'').
Since flat base change preserves annihilators,
we see $F^{e+1}_*cD^px\in F^e_*B\otimes_{F^e_*T}F^e_*I_e^{D_e}$;
in other words, $F_*(cD^px)\in B\otimes_{T}I_e^{D_e}=F_*P\otimes_{P}I_e^{D_e}A$,
as desired.

It remains to prove the Claim. First, we note that being $(S_2)$ is preserved by flat base change if the ring map has $(S_2)$ fibers \cite[Proposition 6.4.1]{EGA4_2}.
Therefore, as both $T\to A$ (by assumption) and $P\to F_*P$ (by regularity) are flat and have $(S_2)$ fibers, to show the Claim it suffices to show that a finite $A$-module $M$ of full support is $(S_2)$ as an $A$-module if and only if it is $(S_2)$ as a $P$-module.
Since $P$ is universally catenary \citestacks{00NM} and $A$ is $(S_2)$ \citestacks{0339},
$A$ is equidimensional \cite[Corollaire 5.10.9]{EGA4_2},
so $\dim A=\dim (A/\fp)+\HT(\fp)$ for all $\fp\in\Spec A$. 
Since the same is true for $P$ we see that $\HT(\fp)=\HT(\fp\cap P)$.
Hence it suffices to show $\depth_{A_\fp} M_\fp=\depth_{P_{\fp\cap P}} M_{\fp\cap P}$.
After passing to the completion it remains to show for a finite inclusion of Noetherian local rings $(R,\fm_R)\subseteq (S,\fm_S)$ and a finite $S$-module $M$ we have $\depth_R M=\depth_S M$.
This is clear as $\fm_S=\sqrt{\fm_R S}$.
\end{proof}

We can run the same argument in the case $D_e=0$,
in which case \eqref{pairs:can} is trivially an isomorphism so the $(S_2)$ property as well as the normality of $T$ can be dropped.
Therefore we have
\begin{Lem}\label{lem:IeF-sigIspandp-1seq}
    Let $(T,\fm_T,k_T)$ be a Noetherian local $\bF_p$-algebra.    Let $(A,\fm,k)$ be a Noetherian local $T$-algebra such that the map $T\to A$ is flat and  local.

Let $(P,D)$ be a Cohen--Gabber normalization of $A$.
    Let $e_0,\rho\in\bZ_{\geq 0}$.
Assume 
\begin{enumerate}
    \item\label{pairs:0dimfiber} $\ppower{\fm}{e_0}\subseteq \fm_T A$.
    \item\label{pairs:relembdim} There exists a regular system of parameters $x_1,\ldots,x_d$ of $P$ such that the image of $x_j$ in $A$ lies in the image of $T$ for all $j>\rho$.
\end{enumerate}    
    Then the sequence $I_e^{F\text{-sig}}(T)A$ (see Definition \ref{def:SplittingIdeals}) is a $p$- and $p^{-1}$-sequence  with respect to $(P,D)$, $c=1$, $e_0$, and $\rho$.
\end{Lem}

\subsubsection{The case of $F$-signature of triples}
\begin{Lem}\label{lem:SplitingPairsToTriples}
    Let $(A,\fm,k)$ be a Noetherian local $\bF_p$-algebra.
        Let $c\in A,e_0,\rho\in\bZ_{\geq 0}$, and let $(P,D)$ be a Cohen--Gabber normalization of $A$.
    Let $(I_e)_e$ be a $p$- and $p^{-1}$-sequence of ideals with respect to $(P,D)$, $c$, $e_0$, and $\rho$.
    Let $\fa\in A$ and let $t\in\bR_{>0}$.
    Let ${m}=l(\fa/\fm\fa)$ be the minimal number of generators of $\fa$ and let $s_1 = \lceil t - tp + p\rceil, s_2 = \max\{s_1,0\} + p{m} - p, s_3 = \lceil tp - t\rceil$.
    
    Then for all $c'\in\fa^{s_2}\cap \fa^{s_3}=\fa^{\max\{s_2,s_3\}}$,
        the sequence $K_e=(I_e:\fa^{\lceil t(p^e-1) \rceil})$
    is a $p$- and $p^{-1}$-sequence of ideals with respect to $(P,D)$, $cc'$, $e_0$, and $\rho$.
\end{Lem}
\begin{proof}
As $K_e\supseteq I_e$ Definition \ref{def:SequencesofIdeals}\eqref{seq:cofinal}\eqref{seq:AlmostMpe} for $K_e$ come from that for $I_e$.
We need to verify Definition \ref{def:SequencesofIdeals}\eqref{seq:pSeq}\eqref{seq:p-1Seq}.

    First, note that a simple calculation \cite[Lemma 3.6]{MM-hilbert-kunz-lengths} shows that for all $b\in\bZ_{\geq 0}$
    \begin{align}\label{triple:MM3.6}
        \fa^{(b+{m}-1)p}\subseteq\ppower{(\fa^b)}{}=(\ppower{\fa}{})^b\subseteq\fa^{bp}.
    \end{align}
We will deduce two more containments we need.
First, note that 
\[
\lceil t(p^{e+1}-1)\rceil + s_1\geq \lceil tp^{e+1}-t+ t - tp + p \rceil=\lceil pt(p^{e}-1)\rceil+ p\geq p\lceil t(p^{e}-1)\rceil.
\]
Therefore, by the first containment in  \eqref{triple:MM3.6}
\begin{align}\label{triple:pfamily}
        \fa^{\lceil t(p^{e+1}-1) \rceil}\fa^{s_2}\subseteq\fa^{p\lceil t(p^{e}-1) \rceil+p{m}-p}\subseteq \ppower{(\fa^{\lceil t(p^{e}-1) \rceil})}{}.
    \end{align}
On the other hand,
\[
p\lceil t(p^{e}-1)\rceil +s_3\geq \lceil pt(p^{e}-1)+tp-t\rceil = \lceil t(p^{e+1}-1) \rceil,
\]
so by by the second containment in \eqref{triple:MM3.6}
\begin{align}\label{triple:p-1family}
        \ppower{(\fa^{\lceil t(p^{e}-1)\rceil})}{}\fa^{s_3}\subseteq\fa^{p\lceil t(p^{e}-1)\rceil}\fa^{s_3}\subseteq\fa^{\lceil t(p^{e+1}-1) \rceil}.
    \end{align}

To get Definition \ref{def:SequencesofIdeals}\eqref{seq:pSeq} for $K_e$, take $x\in K_e$,
so \[
x^p\in \ppower{K}{}_e\subseteq\left(\ppower{I}{}_e:\ppower{(\fa^{\lceil t(p^e-1) \rceil})}{}\right)\subseteq\left(\ppower{I}{}_e:c'\fa^{\lceil t(p^{e+1}-1) \rceil}\right),
\]
where the last containment follows from \eqref{triple:pfamily} because $c'\in\fa^{s_2}$.
As $c\ppower{I}{}_e\subseteq I_{e+1}$, we therefore get $cc'x^p\in \left(I_{e+1}:\fa^{\lceil t(p^{e+1}-1) \rceil}\right)=K_{e+1}$, as desired.

Finally, to get Definition \ref{def:SequencesofIdeals}\eqref{seq:p-1Seq} for $K_e$, let $x\in K_{e+1}$.
As $c'\in\fa^{s_3}$, by \eqref{triple:p-1family} we have $c'x\in \left(I_{e+1}:\ppower{(\fa^{\lceil t(p^e-1) \rceil})}{}\right)$,
in other words $\fa^{\lceil t(p^e-1) \rceil}F_*(c'xA)\subseteq F_*I_{e+1}$.
Therefore, using that $I_e$ is a $p^{-1}$-sequence, we get
$\fa^{\lceil t(p^e-1) \rceil}F_*(c'cD^pxA)\subseteq F_*P\otimes_P I_{e}$.
As $F_*P$ is flat over $P$, taking colon commutes with base change, so we get
$F_*(c'cD^px)\in F_*P\otimes_P K_{e}$,
as desired.
\end{proof}

\subsection{From the ring to modules}\label{subsec:RingToM}
The situation considered in \S\ref{subsec:idealSequence} requires the ring $A$ to be reduced.
To deduce the case of general rings,
we compare the bound between the ring and its modules.

\begin{Lem}[cf.\ {\cite[proof of Corollary 3.4]{PolstraUniform}}]
\label{lem:basicLengths}
Let $A$ be an $\bF_p$-algebra, 
$\fm$ a maximal ideal of $A$,
$I$ an ideal of $A$,
$u$ an element of $A$ such that $(I:u)=\fm$,
$e$ a positive integer.
Let $M$ be an $A$-module.

Write $J=I+(u)$.
Then the following hold.
\begin{enumerate}[label=$(\roman*)$]
    \item\label{basicl_basiciso} $M/(\ppower{I}{e}M:_M u^{p^e})\cong \ppower{J}{e}M/\ppower{I}{e}M$.
    
    \item\label{basicl_1/p} If 
    $A/\fm$ is perfect and $M$ is finitely generated,
    then for all $t\in\bZ_{\geq 0}$,
    \[
    l_A\left(\frac{F^t_*M}{(\ppower{I}{e}F^t_*M:_{F^t_*M}u^{p^e})}\right)=l_A\left(\frac{\ppower{J}{e+t}M}{\ppower{I}{e+t}M}\right)<\infty.
    \]
\end{enumerate}
\end{Lem}
\begin{proof}
There is a canonical surjection $M\to \ppower{J}{e}M/\ppower{I}{e}M$
sending $m$ to $u^{p^e}m$,
showing \ref{basicl_basiciso}.
For \ref{basicl_1/p},
finiteness follows from the fact that $\frac{\ppower{J}{e+t}M}{\ppower{I}{e+t}M}$ is a finitely generated $(A/\ppower{\fm}{e+t})$-module.
To see the identity, notice that $\frac{F^t_*M}{(\ppower{I}{e}F^t_*M:_{F^t_*M}u^{p^e})}=F^t_*\left(\frac{M}{(\ppower{I}{e+t}M:_M u^{p^{e+t}})}\right)$,
and that calculating the length of an $(F^t_*\fm)$-primary $(F^t_*A)$-module over $F^t_*A$ and $A$ are the same since $A/\fm$ is perfect.
\end{proof}

\begin{Prop}\label{prop:singlebound}
Let $(A,\fm)$ be a Noetherian local $\bF_p$-algebra and let $(P,D)$ be a Cohen--Gabber normalization of $A$.

Let $\Delta\in\bZ_{\geq 0},{m}\in\bZ_{\geq 1}$.
Assume that $A$ is generated by ${m}$ elements
as a $P$-module and that
$D\not\in\fm_P^{\Delta+1}$.
Then for all $e\leq e'\in\bZ_{\geq 1}$,
and all ideals $I\subseteq J$ of $A$ with $l_A(J/I)<\infty$,
we have 
\[
\left|\frac{1}{p^{ed}}l_A\left(\frac{\ppower{J}{e}}{\ppower{I}{e}}\right)
-\frac{1}{p^{e'd}}l_A\left(\frac{\ppower{J}{e'}}{\ppower{I}{e'}}\right)\right|
\leq {m} p\Delta p^{-e}l_A(J/I).
\]
\end{Prop}
\begin{proof}
If $I\subseteq J_1\subseteq J$ and the statement is true for both inclusions,
then it is true for $I\subseteq J$ by additivity.
Thus we may assume $l_A(J/I)=1$.
In particular $\fm J\subseteq I$.
By Lemma \ref{lem:basicLengths} we need to show
\[
\left|\frac{1}{p^{ed}}l_A\left(\frac{A}{(\ppower{I}{e}:\ppower{J}{e})}\right)
-\frac{1}{p^{e'd}}l_A\left(\frac{A}{(\ppower{I}{e'}:\ppower{J}{e'})}\right)\right|
\leq {m} p\Delta p^{-e}.
\]
This follows from Lemma \ref{lem:Ie:Je} (note that $e_0=0$) and Theorem \ref{thm:pandp-1SeqEstimate}.
\end{proof}

\begin{Prop}\label{prop:compareMandA}
Let $(P,\fm_P,k)$ be a regular local ring of dimension $d$ containing $\bF_p$, $K=\Frac(P)$, $D\in P$.
Let $A$ be a $P$-algebra, $\fN$ an ideal of $A$, and $M$ a finite $A$-module.
Let $\Delta,\epsilon_0\in\bZ_{\geq 0},{m},b\in\bZ_{\geq 1}$.
Write $\overline{A}=A/\fN$. Assume the following hold.
\begin{enumerate}[label=$(\roman*)$]
    \item $(P,D)$ is a Cohen--Gabber normalization of $\overline{A}$.
    \item $\overline{A}$ is generated by ${m}$ elements as a $P$-module. 
    \item 
$D\not\in\fm_P^{\Delta+1}$.
    \item\label{MandA_Nnilp} $\ppower{\fN}{\epsilon_0}=0$.
    \item\label{MandA_Mfiltr} $M$ has a filtration $M=M_b\supsetneq M_{b-1}\supsetneq\ldots\supsetneq M_0=0$ such that $M_{j}/M_{j-1}\cong \overline{A}$ as $A$-modules.
\end{enumerate}
Then for all $e> \epsilon_0$,
and all ideals $I\subseteq J$ of $A$ with $l_A(J/I)<\infty$,
we have
\[
\left|\frac{b}{p^{(e-\epsilon_0)d}}l_A\left(\frac{\ppower{J}{e-\epsilon_0}\overline{A}}{\ppower{I}{e-\epsilon_0}\overline{A}}\right)
-\frac{1}{p^{ed}}l_A\left(\frac{\ppower{J}{e}M}{\ppower{I}{e}M}\right)\right|
\leq p^{\epsilon_0}b^2 {m}\Delta p^{-e}l_A(J/I).
\]
\end{Prop}
\begin{proof}
Let $(P,\fm_P,k)\to (P',\fm_{P'},k')$ be a flat map of 
regular local rings with $\fm_P P'=\fm_{P'}$,
and let $A'=A\otimes_P P'$.
Then it is clear that 
$(P',D)$ is a Cohen--Gabber normalization of the local ring $A'/\fN A'$ (in particular $A'/\fN A'$ is reduced), and
$A'/\fN A'$ is
generated by ${m}\in\bZ_{>0}$ elements
as a $P'$-module.
We know $\fm_{P'}^{\Delta+1}\cap P=\fm_P^{\Delta+1}$ by flatness,
so all assumptions hold for $P'\to A'$.
For any finite length $A$-module $X$,
$l_A(X)=l_{A'}(X\otimes_A A')$.
Thus we may assume $P$ complete and $k$ algebraically closed.
Calculation of lengths therefore does not depend on the base ring chosen.

As before, we may assume $l_A(J/I)=1$, $J=I+(u)$.
Write $H=F^{\epsilon_0}_*P\otimes_P A$ and $\overline{H}=F^{\epsilon_0}_*P\otimes_P \overline{A}$.
As seen in the proof of Theorem \ref{thm:pandp-1SeqEstimate},
there exists an exact sequence
\begin{align}\label{MandA_HandAExSeq}
    \begin{CD}
0@>>> \overline{H}@>>> F^{\epsilon_0}_*\overline{A}@>>> N@>>> 0
\end{CD}
\end{align}
where $N$ is annihilated by $D$ and is generated by ${m}$ elements as a $F^{\epsilon_0}_*P$-module.
By \ref{MandA_Mfiltr},
as an $F^{\epsilon_0}_*A$-module, 
$F^{\epsilon_0}_*M$ is a successive extension of $b$ isomorphic copies of $F^{\epsilon_0}_*\overline{A}$,
thus the same is true for $F^{\epsilon_0}_*M$ as an $H$-module.
By \ref{MandA_Nnilp}, $F^{\epsilon_0}_*M$ is an $\overline{H}$-module.
Thus the exact sequence \eqref{MandA_HandAExSeq} implies the existence of an exact sequence of $\overline{H}$-modules
\begin{align}\label{MandA_HandMExSeq}
\begin{CD}
0@>>> \overline{H}^{\oplus b}@>>> F^{\epsilon_0}_*M@>>> N'@>>> 0
\end{CD}
\end{align}
where $N'$ is a successive extension of $b$ isomorphic copies of $N$.
In particular, $N'$ is annihilated by $D^b\not\in \fm_P^{b\Delta+1}$ and is generated by $bm$ elements as an $F^{\epsilon_0}_*P$-module.

We have an exact sequence
\[
\begin{CD}
\left(\frac{\overline{H}}{(\ppower{I}{e-\epsilon_0}\overline{H}:_{\overline{H}} u^{p^{e-\epsilon_0}})}\right)^{\oplus b}@>>> \frac{F^{\epsilon_0}_*M}{(\ppower{I}{e-\epsilon_0}F^{\epsilon_0}_*M:_{F^{\epsilon_0}_*M} u^{p^{e-\epsilon_0}})}
@>>> N''@>>> 0
\end{CD}
\]
of $H$-modules with $N''$ a quotient of $N'/\ppower{\fm}{e}_P N'$,
see \cite[proof of Corollary 3.4]{PolstraUniform} for more details.
Lemma \ref{lem:basicLengths} gives
\[
-bp^{\epsilon_0d}l_P\left(\frac{\ppower{J}{e-\epsilon_0}\overline{A}}{\ppower{I}{e-\epsilon_0}\overline{A}}\right)
+l_P\left(\frac{\ppower{J}{e}M}{\ppower{I}{e}M}\right)
\leq l_P\left(\frac{N'}{\ppower{\fm}{e-\epsilon_0}_P N'}\right).
\]
By Lemma \ref{lem:HKmulthypersurf}, $l_{P}(N'/\ppower{\fm}{e-\epsilon_0}_P N')=l_{F^{\epsilon_0}_*P}(N'/\ppower{\fm}{e}_{F^{\epsilon_0}_*P}N')\leq b{m} p^{\epsilon_0}b\Delta p^{e(d-1)}$,
as $D^b\not\in \fm_P^{b\Delta+1}$ implies
$D^b\not\in \fm_{F^{\epsilon_0}_*P}^{p^{\epsilon_0}b\Delta+1}$.
Thus
\begin{align}
-\frac{b}{p^{(e-\epsilon_0)d}}l_A\left(\frac{\ppower{J}{e-\epsilon_0}\overline{A}}{\ppower{I}{e-\epsilon_0}\overline{A}}\right)
+\frac{1}{p^{ed}}l_A\left(\frac{\ppower{J}{e}M}{\ppower{I}{e}M}\right)
\leq p^{\epsilon_0}b^2 {m}\Delta p^{-e}.
\label{MandA_formula}
\end{align}

The exact sequence \eqref{MandA_HandMExSeq} gives
\[
\begin{CD}
0@>>> D^b.F^{\epsilon_0}_*M@>>> \overline{H}^{\oplus b}@>>>  L@>>> 0
\end{CD}
\]
where $L$ is annihilated by $D^b$ by construction, and is generated by $bm$ elements as a $F^{\epsilon_0}_*P$-module since $\overline{A}$ is generated by ${m}$ elements as a $P$-module.
By \ref{MandA_Mfiltr}, $M$ is a torsion-free $P$-module.
Thus $D^b$ is a nonzerodivisor on $F^{\epsilon_0}_*M$ and $D^b.F^{\epsilon_0}_*M\cong F^{\epsilon_0}_*M$.
This gives the inequality (\ref{MandA_formula}) with signs on the left hand side reversed, showing the proposition.
\end{proof}

\begin{Cor}\label{cor:singleboundALLM}
Notations and assumptions as in Proposition \ref{prop:compareMandA}. 
Then for all $e'\geq e> \epsilon_0$,
and all ideals $I\subseteq J$ of $A$ with $l_A(J/I)<\infty$,
we have
\[
\left|\frac{1}{p^{ed}}l_A\left(\frac{\ppower{J}{e}M}{\ppower{I}{e}M}\right)
-\frac{1}{p^{e'd}}l_A\left(\frac{\ppower{J}{e'}M}{\ppower{I}{e'}M}\right)\right|
\leq (p+2b)p^{\epsilon_0}b{m}\Delta p^{-e}l_A(J/I).
\]
\end{Cor}
\begin{proof}
Immediate from Propositions \ref{prop:singlebound} and \ref{prop:compareMandA}.
\end{proof}

\subsection{Uniform bound in excellent and less-excellent rings}

We shall use the following fact.

\begin{Prop}
\label{prop:PTY}
Let $R$ be a Noetherian $\bF_p$-algebra,
$\fp\in\Spec(R)$.
Assume that $R/\fp$ is J-0.
Then for every $R$-module $M$,
there exists a constant $C=C(M,\fp)$
and $f=f(M,\fp)\not\in \fp$,
such that for all $\fP\in V(\fp)\cap D(f)$
and all $e\in\bZ_{\geq 0}$,
$l(M_\fP/\ppower{\fP}{e}M_\fP)\leq Cp^{e\dim M_\fP}$.
\end{Prop}
\begin{proof}
This is \cite[Lemma 15]{Smirnov-eHKsemicont},
where $R$ is assumed to be excellent.
The proof works with just the assumption $R/\fp$ is J-0.
\end{proof}

\begin{Cor}
\label{cor:PTY}
Let $R$ be a Noetherian $\bF_p$-algebra.
Assume that $R/\fp$ is J-0 for all $\fp\in\Spec(R)$.
Then for every $R$-module $M$,
there exists a constant $C=C(M)$
such that for all $\fp\in\Spec(R)$
and all $e\in\bZ_{\geq 0}$,
$l(M_\fp/\ppower{\fp}{e}M_\fp)\leq Cp^{e\dim M_\fp}$.
\end{Cor}
\begin{proof}
This follows from the compactness of the constructible topology. 
\end{proof}
\begin{Prop}
\label{prop:unifbound}
Let $R$ be a Noetherian $\bF_p$-algebra,
$\fp\in\Spec(R)$,
$M$ a finite $R$-module
with $\operatorname{Ann}_RM\subseteq\fp$,
$d=\dim M_\fp$. 
Assume the following.
\begin{enumerate}
    \item $R/\fp$ is J-2.
    \item\label{unifbound:localBertini} 
    There exists an ideal $\fN$ of $R_\fp$ such that $\operatorname{Ann}_{R_\fp}M_\fp\subseteq{\fN}\subseteq\sqrt{\operatorname{Ann}_{R_\fp}M_\fp}$,
    and elements $a_1,\ldots,a_{d-1}\in \fp R_\fp$ such that
all minimal primes of $((\underline{a})+\fN)/\fN$
are of height $d-1$
and that $(R_\fp/((\underline{a})+\fN))^\wedge$ is $(R_0)$.
\end{enumerate}
Then 
there exists a constant $C$ 
and an $f\not\in \fp$
with the following property.
For all $\fP\in V(\fp)\cap D(f)$, 

all ideals $I\subseteq J$ of $R_\fP$ with $l_{R_\fP}(J/I)<\infty$,
and all $e\leq e'\in\bZ_{\geq 0}$,
the following holds.
\[
\left|\frac{1}{p^{e\dim M_\fP}}l_{R_\fP}\left(\frac{\ppower{J}{e}M_\fP}{\ppower{I}{e}M_\fP}\right)
-\frac{1}{p^{e'\dim M_\fP}}l_{R_\fP}\left(\frac{\ppower{J}{e'}M_\fP}{\ppower{I}{e'}M_\fP}\right)\right|
\leq C p^{-e}l_{R_\fP}(J/I).
\]
\end{Prop}

\begin{proof}

Note that annihilators, dimensions, and lengths do not change along a base change that identifies completions.
Applying Discussion \ref{discu:localize-complete},
we may assume $R$ is universally catenary
and has $(S_1)$ formal fibers.

We may replace $R$ by $R/\operatorname{Ann}_R(M)$, so $\dim M_\fP=\HT \fP$ for all $\fP$.
We may localize near $\fp$ so that all minimal primes of $R$ are contained in $\fp$,
so $\HT(\fP)=\HT(\fp)+\HT(\fP/\fp)$
for all $\fP\in V(\fp)$
as $R$ is catenary.

Let $\fp_0$ be a minimal prime of $R$.
Then there exists a submodule $N=N(\fp_0)$ of $M$ that is a successive extension of isomorphic copies of $R/\fp_0$, such that $N_{\fp_0}= M_{\fp_0}$, 
by the theory of associated primes.
Let $N'=\oplus_{\fp_0} N(\fp_0)$
(not necessarily a submodule of $M$),
so $M$ and $N'$ are isomorphic at all minimal primes of $R$,
in particular $\operatorname{Ann}_R(N')$ is nilpotent.
Apply the argument in \cite[proof of Corollary 3.4]{PolstraUniform}
(which is somewhat analogous to the proof of Proposition \ref{prop:singlebound}),
using Proposition \ref{prop:PTY}
instead of \cite[Proposition 3.3]{PolstraUniform},
we see that it suffices to prove the inequality for $N'$.

In fact, it suffices to prove the inequality for each $N(\fp_0)$
when $\HT(\fp/\fp_0)=d$.
Indeed, assume the result is true for each $N(\fp_0)$ when $\HT(\fp/\fp_0)=d$ and let $C(\fp_0),f(\fp_0)$ be the corresponding constant and element.
For other $\fp_0$,
Let $C'(\fp_0),f(\fp_0)$ be as in Proposition \ref{prop:PTY} for $N(\fp_0)$
and $C(\fp_0)=2C'(\fp_0)$.
We claim that $C=\sum_{\fp_0} C(\fp_0)$ and $f=\prod_{\fp_0} f(\fp_0)$ and 
works for $N'$.
To see this, let $\fP\in V(\fp)\cap D(f)$. 
Since $N'$ is the direct sum of all $N(\fp_0)$, it suffices to show
\[
\left|\frac{1}{p^{e\HT\fP}}l_{R_\fP}\left(\frac{\ppower{J}{e}N(\fp_0)_\fP}{\ppower{I}{e}N(\fp_0)_\fP}\right)
-\frac{1}{p^{e'\HT\fP}}l_{R_\fP}\left(\frac{\ppower{J}{e'}N(\fp_0)_\fP}{\ppower{I}{e'}N(\fp_0)_\fP}\right)\right|
\leq C(\fp_0) p^{-e}l_{R_\fP}(J/I).
\]
for every minimal prime $\fp_0$.
If $\HT(\fp/\fp_0)<d$,
then $\HT(\fP/\fp_0)<\HT(\fP)$ as $R$ is catenary,
so this follows from the proof of \cite[Lemma 3.2]{PolstraUniform} and the choice of $C'(\fp_0)$.
Otherwise, $\HT(\fp/\fp_0)=\HT(\fp)$,
so $\HT(\fP/\fp_0)=\HT(\fP)$, and
the desired inequality follows from the choice of $C(\fp_0)$.

Next, we verify that each $N(\fp_0)$ satisfies (\ref{unifbound:localBertini}),
so we can replace $M$ by $N=N(\fp_0)$.
Let $\fN'=\fN+\operatorname{Ann}_{R_\fp}N_\fp$.
Then 
$\operatorname{Ann}_{R_\fp}N_\fp\subseteq{\fN'}\subseteq\sqrt{\operatorname{Ann}_{R_\fp}N_\fp}$,
and $\sqrt{\fN'}=\fp_0R_\fp$,
as $\fN$ is nilpotent.
Let $Q/\fN'$ be a minimal prime of $((\underline{a})+\fN')/\fN'$ where $Q\in\Spec(R_\fp)$.
Then $\HT(Q/\fN')\leq d-1$ as $(\underline{a})$ is generated by $d-1$ elements.
Therefore $\dim(R_\fp/Q)\geq 1$.
On the other hand $Q/\fN$ contains a minimal prime of $((\underline{a})+\fN)/\fN$,
so $\HT(Q/\fN)\geq d-1$.
This forces $\dim(R_\fp/Q)= 1$.
In particular,
$(R_\fp/((\underline{a})+\fN'))^\wedge$
is a $1$-dimensional quotient 
of the $1$-dimensional $(R_0)$ local ring $(R_\fp/((\underline{a})+\fN))^\wedge$,
hence is $(R_0)$.
Since we also have $\sqrt{\fN'}=\fp_0R_\fp$ and $R_\fp$ catenary,
we see $\HT(Q/\fN')=\HT(\fp/\fp_0)-\dim(R_\fp/Q)=d-1$.
as desired.

Thus
we may replace $M$ by $N(\fp_0)$ and assume $M$
is a successive extension of isomorphic copies of $R/\fp_0$ where $\fp_0$ is a fixed minimal prime of $R$.
Replace $R$ by $R/\operatorname{Ann}_R(M)$ once again,
we may assume $\fp_0$ is the nilradical of $R$.
Write $\overline{R}=R/\fp_0$.
Let $b=l_{R_{\fp_0}}(M_{\fp_0})$ and let $\epsilon_0\in\bZ_{\geq 0}$ be such that
$\ppower{(\fp_0)}{\epsilon_0}=0.$
By Proposition \ref{prop:PTY} and \cite[Lemma 3.2]{PolstraUniform}, it suffices to find a constant $C$ such that the desired inequality holds for all $e'\geq e> \epsilon_0.$

Note that (\ref{unifbound:localBertini}) holds for $\fN=\fp_0$ as
$(R_\fp/((\underline{a})+\fp_0))^\wedge$
is a $1$-dimensional quotient 
of the $1$-dimensional $(R_0)$ local ring $(R_\fp/((\underline{a})+\fN))^\wedge$,
hence is $(R_0)$.
Let $\overline{R}\to \overline{S}$, $f$, and $\delta,\mu$ and $\Delta$ be as in Proposition \ref{prop:UnifCohenGabberConstrNbhd}.
We will show that $C=(p+2b)p^{\epsilon_0}b\delta^\mu\Delta$ and $f$ work. 
Let $R\to S$ be a syntomic 
ring map that lifts $\overline{R}\to \overline{S}$, see \citestacks{07M8}.
Then $\fp_0 S$ is a nilpotent ideal of $S$, so we can identify $\Spec(S)$ and $\Spec(\overline{S})$.
Fix $\fP\in V(\fp)\cap D(f)$
and let $\fQ\in\Spec(\overline{S})$, $P\to \overline{S}_\fQ^\wedge$ be as in Proposition \ref{prop:UnifCohenGabberConstrNbhd}.
Lift the map $P\to \overline{S}_\fQ^\wedge$ to a ring map $P\to S_\fQ^\wedge$,
possible as $P$ is formally smooth over $\bF_p$ \citestacks{07NL}.

Note that $R\to S$ is flat quasi-finite,
so $R_\fP\to S_\fQ^\wedge$ is flat local with zero-dimensional closed fiber.
Thus, for all finite length $R_\fP$-modules $X$,
$l_{R_\fP}(X)l_{S_\fQ^\wedge}(S_\fQ^\wedge/\fP S_\fQ^\wedge)=l_{S_\fQ^\wedge}(X\otimes_{R_\fP} S_\fQ^\wedge)$.
Hence it suffices to prove an estimate as in the statement of Corollary \ref{cor:singleboundALLM} for the $S_\fQ^\wedge$-module $M\otimes_R S_\fQ^\wedge$
with the correct constants $b,m=\delta^\mu,$ and $\Delta$.
It thus suffices to verify the assumptions of Corollary \ref{cor:singleboundALLM}
for $P\to S_\fQ^\wedge$.

As the domain $\overline{R}$ is locally equidimensional and $(S_1)$, by Remark \ref{rem:Sreduced-CONSTRnbhd}
$\overline{S}_\fQ^\wedge$ is reduced and is a torsion-free $P$-module.
Note that $\overline{S}_\fQ^\wedge$ is a finite and generically \'etale $P$-algebra (Proposition \ref{prop:UnifCohenGabberConstrNbhd}\ref{UnifCGP_degree}).

By Proposition \ref{prop:UnifCohenGabberConstrNbhd}\ref{UnifCGP_embdim},
we can find $y_1,\ldots,y_\mu\in \fQ$
such that $\fQ \overline{S}_\fQ^\wedge=\fm_P \overline{S}_\fQ^\wedge +(\underline{y})$.
Let $S'=P[y_1,\ldots,y_\mu]\subseteq \overline{S}_\fQ^\wedge$, and $\fm'=\fm_P S'+(\underline{y})$.
Then we see that $(S',\fm')$ is a local ring and that
$\overline{S}_\fQ^\wedge=S'+\fm' \overline{S}_\fQ^\wedge$ by Proposition \ref{prop:UnifCohenGabberConstrNbhd}\ref{UnifCGP_samek}.
Therefore $S'=\overline{S}_\fQ^\wedge$.
By Lemma \ref{lem:generators} (and Proposition \ref{prop:UnifCohenGabberConstrNbhd}\ref{UnifCGP_degree}),
we see that 
\begin{align}\label{numberofGeneratorsofSoverP}
   \overline{S}_\fQ^\wedge\text{ is generated by at most }\delta^\mu\text{ elements as a }P\text{-module.} 
\end{align}

Let $e_1,\ldots,e_n\in \overline{S}_\fQ^\wedge$ be as in 
Proposition \ref{prop:UnifCohenGabberConstrNbhd}\ref{UnifCGP_disc}, and let
$D=\Disc_{\overline{S}_\fQ^\wedge/P}(e_1,\ldots,e_n)$.
Then $D\not\in\fm_P^{\Delta+1}$.

We have $\ppower{(\fp_0 S_\fQ^\wedge)}{\epsilon_0}=0$ since $\ppower{\fp}{\epsilon_0}_0=0$.
Since $M$ is a successive extension of $b$ isomorphic copies of $\overline{R}$,
$M\otimes_R {S}_\fQ^\wedge$
is a successive extension of $b$ isomorphic copies of $\overline{S}_\fQ^\wedge$.
We have verified all assumptions of and checked all constants in Corollary \ref{cor:singleboundALLM},
showing what we want.
\end{proof}

We get a global version as we did in \S\ref{subsec:UnifCG}.

\begin{Thm}[cf. {\cite[Theorem 4.4]{PolstraUniform}}]
\label{thm:unifbound}
Let $R$ be a Noetherian $\bF_p$-algebra that satisfies Condition \ref{condit:excellence}.
Then for every finite $R$-module $M$, there exists a constant $C(M)$ with the following property.
For all $\fp\in\Spec(R)$, 
all ideals $I\subseteq J$ of $R_\fp$ with $l_{R_\fp}(J/I)<\infty$,
and all $e\leq e'\in\bZ_{\geq 0}$,
the following holds.
\[
\left|\frac{1}{p^{e\dim M_\fp}}l_{R_\fp}\left(\frac{\ppower{J}{e}M_\fp}{\ppower{I}{e}M_\fp}\right)
-\frac{1}{p^{e'\dim M_\fp}}l_{R_\fp}\left(\frac{\ppower{J}{e'}M_\fp}{\ppower{I}{e'}M_\fp}\right)\right|
\leq C(M) p^{-e}l_{R_\fp}(J/I).
\]
\end{Thm}
Here by convention the left hand side is zero if $M_\fp=0$.
\begin{proof}
    As the constructible topology is compact, 
    it suffices to verify assumption
    (\ref{unifbound:localBertini})
    in Proposition \ref{prop:unifbound}
    for all $\fp\in\Spec(R)$.
    We can just take $\fN=\sqrt{\operatorname{Ann}_{R_\fp}M_\fp}$
    in view of Lemma \ref{lem:R0Bertini} and Condition \ref{condit:excellence}\ref{excel_R0},
    as $R_\fp$ is J-2.
\end{proof}
\section{Applications: semi-continuity}\label{sec:semicont}

\subsection{Hilbert--Kunz multiplicity}

For a Noetherian local $\bF_p$-algebra $(A,\fm)$,
denote by $\lambda_e(A)$ the number
$\frac{l(A/\ppower{\fm}{e})}{p^{e\dim A}}$.
We have, by definition, $\ehk{A}=\lim_e \lambda_e(A)$,
and the limit exists \cite{Monsky-HKmult}.
We recover the existence in our Theorem \ref{thm:eHKSemicont}
as $A^\wedge$ satisfies Condition \ref{condit:excellence}.

The following slightly strengthens \cite{SB-HKmult}.

\begin{Lem}\label{lem:individualHKconstr}
Let $R$ be a Noetherian $\bF_p$-algebra,
$\fp\in\Spec(R)$.
Assume that $R/\fp$ is J-0.
Let $e$ be a positive integer.
Then for some $g\not\in\fp$ and all $\fP\in D(g)\cap V(\fp)$,
$\lambda_e(R_{\fp})=\lambda_e(R_{\fP})$.
\end{Lem}
\begin{proof}
We may assume $R/\fp$ regular.
By Theorem \ref{thm:uniftame}\ref{uniftame_height},
we may assume for all $\fP\in V(\fp)$,
$\HT(\fP)=\HT(\fp)+\HT(\fP/\fp)$.
It remains to apply Discussion \ref{discu:localize}
to the module $M=R/\ppower{\fp}{e}$
and the regular sequence $\pi_1=t_1^{p^e},\ldots, \pi_h=t_h^{p^e}$,
where $t_1,\ldots,t_h\in R_\fP$ map to a regular sequence of parameters of $R_\fP/\fp R_\fP$.
\end{proof}

\begin{Cor}\label{cor:individualHKsemicont}
Let $R$ be a Noetherian $\bF_p$-algebra.
Assume that $R/\fp$ is J-0 for all $\fp\in \Spec(R)$,
and that $R$ is catenary and locally equidimensional.
Let $e$ be a positive integer.
Then the function
$\fp\mapsto\lambda_e(R_{\fp})$
is constructible and upper semi-continuous.
\end{Cor}
\begin{proof}
By Lemma \ref{lem:individualHKconstr} our function is constructible.
We have $\HT(\fp_1)=\HT(\fp_2)+\HT(\fp_1/\fp_2)$ for all $\fp_1\supseteq\fp_2$,
since $R$ is catenary and locally equidimensional.
By \cite[Corollary 3.8]{Kunz1976},
our function is non-decreasing along specialization.
Thus our function is upper semi-continuous.
\end{proof}

\begin{Thm}[cf. {\cite[Theorem 23]{Smirnov-eHKsemicont}}]\label{thm:eHKSemicont}
Let $R$ be a Noetherian $\bF_p$-algebra.
Assume that $R$ satisfies Condition \ref{condit:excellence}.
If $R$ is catenary and locally equidimensional,
then the function $\fp\mapsto\ehk{R_\fp}$
is upper semi-continuous.
\end{Thm}
\begin{proof}
Applying Theorem \ref{thm:unifbound} to $M=R, I=\fp R_\fp, J=R_\fp$,
we see that our function is the uniform limit of the functions $\fp\mapsto\lambda_e(R_{\fp})$.
These functions are upper semi-continuous
by Corollary \ref{cor:individualHKsemicont}.
Thus our function is upper semi-continuous as well.
\end{proof}

\subsection{$F$-signature of rings}
For a Noetherian local $\bF_p$-algebra $(A,\fm)$,
denote by $s_e(A)$ the $e$th normalized $F$-splitting number as in \cite[Definition 1.1]{EY-splittingnumbers}.
The limit $s(A)=\lim s_e(A)$ is called the $F$-signature of $A$.
The limit was first shown to exist in \cite{Tucker-SRexists}.
(We also recover the existence in Proposition \ref{prop:SeUnifConv} as $A^\wedge$ satisfies Condition \ref{condit:excellence}.)

We use the following facts.

\begin{Fact}\label{fact:completionSameSe}
Let $(A,\fm)\to (A',\fm')$ be a flat map of Noetherian local $\bF_p$-algebras with $\fm A'=\fm'$.
Then $s_e(A)=s_e(A')$ for all $e$,
see \cite[Remark 2.3(3)]{YaoSe}.
\end{Fact}

\begin{Fact}\label{fact:FpureSePositive}
For a Noetherian local $\bF_p$-algebra $(A,\fm)$,
$s_e(A)>0$ for some $e$ if and only if $s_e(A)>0$ for all $e$,
if and only if $A$ is $F$-pure.
Indeed, under the notations preceding \cite[Definition 1.1]{EY-splittingnumbers},
$s_e(A)>0$ if and only if $A^{(e)}\otimes_A k\to A^{(e)}\otimes_A E$ is nonzero,
if and only if $k$ is not killed in $A^{(e)}\otimes_A E$,
if and only if $A\to A^{(e)}$ is pure, see \cite[Proposition 1.3(5)]{Fed83-complete}.
\end{Fact}

\begin{Fact}\label{fact:SeLength}
Let $(A,\fm)$ be a Noetherian local $\bF_p$-algebra.
For two positive integers $e,e'$,
there exists an $\fm$-primary ideal $I$
and an element $u\in (I:\fm)$
such that
$s_e(A)=p^{-e\dim A}l(\ppower{{(I,u)}}{e}/\ppower{I}{e})$
and
$s_{e'}(A)=p^{-e'\dim A}l(\ppower{{(I,u)}}{e'}/\ppower{I}{e'})$.
Indeed, by Fact \ref{fact:completionSameSe}
we may assume $A$ complete,
and by Fact \ref{fact:FpureSePositive}
we may assume $A$ $F$-pure (otherwise take $I=\fm$ and $u=0$), 
in particular reduced,
so \cite[Lemma 5.4]{PolstraUniform} applies.
\end{Fact}

\begin{Prop}\label{prop:SeUnifConv}
Let $R$ be a Noetherian ring that satisfies 
Condition \ref{condit:excellence}.
Let $C=C(R)$ be as in 
Theorem \ref{thm:unifbound}.
Then for all $\fp\in\Spec(R)$
and all $e'\geq e\geq 1$,
$|s_e(R_\fp)-s_{e'}(R_{\fp})|\leq Cp^{-e}.$
\end{Prop}
\begin{proof}
Let $\fp\in\Spec(R)$.
By Fact \ref{fact:SeLength},
we need to show
\[
\left|\frac{1}{p^{e\HT{\fp}}}l_{R_\fp}\left({\ppower{J}{e}}/{\ppower{I}{e}}\right)
-\frac{1}{p^{e'\HT{\fp} }}l_{R_\fp}\left({\ppower{J}{e'}}/{\ppower{I}{e'}}\right)\right|
\leq C p^{-e},
\]
where $I$ is a $\fp R_\fp$-primary ideal of $R_\fp$, and $J=(I,u)$ for some $u\in (I:\fp R_\fp)$.
In particular $l_{R_\fp}(J/I)\leq 1$.
The inequality now follows from Theorem \ref{thm:unifbound}.
\end{proof}

The following result appeared in \cite{EY-splittingnumbers} assuming $R$ is excellent Gorenstein, or $R$ is a quotient of an excellent regular ring.
The material in \cite{HochsterYaoSe}
gives the case when $R$ is an $(S_2)$ quotient of an excellent Cohen--Macaulay ring.
With the help of \cite{Lyu-dual-complex-lift},
we can push the generality a step further.

\begin{Lem}[=Theorem \ref{thm:seSemicont}]\label{lem:SeSemicont}
Let $R$ be a Noetherian $\bF_p$-algebra
such that $R/\fp$ is J-0 for all $\fp\in\Spec(R)$.
Then for all $e$, the function 
$\fp\mapsto s_e(R_\fp)$
is constructible.
If $R$ is catenary and locally equidimensional,
then the function 
$\fp\mapsto s_e(R_\fp)$ is also lower semi-continuous.
\end{Lem}

The following result is likely well-known; we include it for completeness.

\begin{Lem}\label{lem:AL02}
    Let $(A,\fm,k)$ be a Noetherian local $\bF_p$-algebra.
    Then the following hold.
    \begin{enumerate}[label=$(\roman*)$]
        \item\label{AL02:posFsigImpliesSFR} If $s(A)>0$, then $A$ is strongly $F$-regular.
        \item\label{AL02:SFRimpliesposFsig}  If $A$ is a G-ring then the converse to \ref{AL02:posFsigImpliesSFR} holds.
    \end{enumerate}
\end{Lem}
\begin{proof}
    Assume $s(A)>0$.
    Let $(A',\fm',k')$ be a Noetherian local flat $A$-algebra with $A'$ complete, $\fm A'=\fm'$, and $k'$ algebraically closed.
    Fact \ref{fact:completionSameSe} shows $s(A')>0$, and \cite{ALsFRsignature} shows $A'$ strongly $F$-regular.
    Thus $A$ is strongly $F$-regular by \cite[Lemma 3.17]{Has10}.

    Conversely, assume $A$ is a G-ring and strongly $F$-regular.
    Then $A$ is normal, thus excellent, cf. \citetwostacks{0C23}{0AW6}.
    By \cite[Lemma 3.28]{Has10} the completion $A^\wedge$ is strongly $F$-regular.
    By \cite[Corollary 3.31]{Has10}
    there exists a flat local ring map $A^\wedge\to A'$ such that $A'$ is $F$-finite and strongly $F$-regular, and that $\fm A'$ is the maximal ideal of $A'$.
    By \cite{ALsFRsignature}, $s(A')>0$,
    and $s(A)=s(A')$ by Fact \ref{fact:completionSameSe}.
\end{proof}
\begin{Rem}
    The G-ring assumption can be weakened to that the generic fiber of $A\to A^\wedge$ is regular, see Lemma \ref{lem:SFRcriteria}.
\end{Rem}

\begin{Thm}[cf.\ {\cite[Theorem 5.6]{PolstraUniform}}]\label{thm:sSemicont}
Let $R$ be a Noetherian $\bF_p$-algebra that satisfies 
Condition \ref{condit:excellence}.
Then the function $\fp\mapsto s(R_\fp)$
is lower semi-continuous.
\end{Thm}
\begin{proof}
Note that $s(R_\fp)\geq 0$ for all $\fp$.
If $s(R_\fp)>0$ for some $\fp$,
then $R_\fp^\wedge$ is strongly $F$-regular by Fact \ref{fact:completionSameSe} and 
Lemma \ref{lem:AL02},
so $R_\fp$ is Cohen--Macaulay.
Since $R$ is J-2,
the Cohen--Macaulay locus of $R$ is open.
Thus we may assume $R$  Cohen--Macaulay,
in particular catenary and locally equidimensional \citetwostacks{00NM}{0BUS}.

By Proposition \ref{prop:SeUnifConv} the function
$\fp\mapsto s(R_\fp)$
is the uniform limit of the functions
$\fp\mapsto s_e(R_\fp)$
which are lower semi-continuous by
Lemma \ref{lem:SeSemicont},
thus $\fp\mapsto s(R_\fp)$
is lower semi-continuous.
\end{proof}

\begin{Cor}\label{cor:sFRlocusopen}
Let $R$ be a Noetherian quasi-excellent $\bF_p$-algebra.
Then the locus
\[
\{\fp\in\Spec(R)\mid R_\fp\text{\ is\ strongly\ }F\text{-regular}\}
\]
is open.
\end{Cor}
\begin{proof}

For $\fp\in\Spec(R)$,
$R_\fp$ is strongly $F$-regular
if and only if $s(R_\fp)>0$,
see Lemma \ref{lem:AL02}.
\end{proof}

\begin{Rem}
    The materials in \cite{HochsterYaoSe} imply Corollary \ref{cor:sFRlocusopen},
    and
    we can push the generality a step further, see Theorem \ref{thm:SFRopenGENERAL}.
\end{Rem}

\subsection{$F$-signature of pairs and triples}
\begin{Def}[{cf. \cite[\S 2]{BST-pairs-triples}}]\label{def:SplittingIdealTriples}
    Let $(T,\fm_T,k_T)$ be a Noetherian local $\bF_p$-algebra and let $\fa$ be an ideal of $T$ not contained in any minimal prime of $T$. Let $t\in \bR_{>0}$.
    Consider the ideal $I_e^{\fa^t}=(I_e^{F\text{-sig}}(T):\fa^{\lceil t(p^e-1)\rceil})$,
    where $I_e^{F\text{-sig}}(T)$ is the $e$th $F$-splitting ideal of $T$, see Definition \ref{def:SplittingIdeals}.
    We write $s_e(T,\fa^t)=p^{-e\dim T}l(T/I_e^{\fa^t})$.
    The limit $s(T,\fa^t)=\lim_e s_e(T,\fa^t)$ is the \emph{$F$-signature of $(T,\fa^t)$}.

    Assume $T$ is normal and let $\Delta$ be an effective $\bR$-Weil divisor.
    Consider the ideal $I_e^{\fa^t}=(I_e^{\Delta_e}:\fa^{\lceil t(p^e-1)\rceil})$,
    where $\Delta_e=\lceil(p^e-1)\Delta\rceil$ and $I_e^{\Delta_e}$ is as in Definition \ref{def:SplittingIdealPairs}.
    We write $s_e(T,\Delta,\fa^t)=p^{-e\dim T}l(T/I_e^{\Delta,\fa^t})$.
    The limit $s(T,\Delta,\fa^t)=\lim_e s_e(T,\Delta,\fa^t)$ is the \emph{$F$-signature of $(T,\fa^t)$}.
\end{Def}

\begin{Rem}
    The limits in Definition \ref{def:SplittingIdealTriples} always exist.
    Indeed, in the $F$-finite case this follows from the consideration on Cartier algebras \cite{BST-pairs-Cartier-algebra},
    and the general case can be reduced to the $F$-finite case via completion and the $\Gamma$-construction.
    We also recover the existence immediately below
as $T^\wedge$ satisfies Condition \ref{condit:excellence}.

\end{Rem}

We apply our framework developed in \S\ref{subsec:idealSequence} and \S\ref{subsec:examplesofIdealSeq} to deduce the uniform convergence result. 

\begin{Prop}\label{prop:SeTRIPLEsUnifConv}
Let $R$ be a Noetherian $\bF_p$-algebra  that satisfies 
Condition \ref{condit:excellence}.
Let $\fa$ be an ideal of $R$ not contained in any minimal prime of $R$ and let $t\in\bR_{>0}$.
Then there exists a constant $C=C(R,\fa,t)$
such that for all $\fp\in\Spec(R)$
and all $e'\geq e\geq 1$,
$|s_{e'}(R_{\fp},\fa^t_\fp)-s_e(R_\fp,\fa^t_\fp)|\leq Cp^{-e}.$

Assume $R$ is normal and let $\Delta$ be an effective $\bR$-Weil divisor on $R$.
Then there exists a constant $C=C(R,\Delta,\fa,t)$
such that for all $\fp\in\Spec(R)$
and all $e'\geq e\geq 1$,
$|s_{e'}(R_{\fp},\Delta_\fp,\fa^t_\fp)-s_e(R_\fp,\Delta_\fp,\fa^t_\fp)|\leq Cp^{-e}.$
\end{Prop}
\begin{proof}
    We know  $ s_e(R_\fp)\geq s_e(R_\fp,\fa^t_\fp)\geq 0$ and $ s_e(R_\fp)\geq s_e(R_\fp,\Delta_\fp,\fa^t_\fp)\geq 0$, so if $s(R_\fp)=0$, the desired convergence follows from Proposition \ref{prop:SeUnifConv}.
    We may therefore assume, by Theorem \ref{thm:sSemicont} (and the Noetherianness of $\Spec(R)$), that $s(R_\fp)>0$ for all $\fp\in\Spec(R)$,
    in particular $R^\wedge_\fp$ is normal and Cohen--Macaulay, hence so is $R$.
    We may also assume $R$ is an integral domain.

Let us first prove the result in the case of $(R,\fa^t)$.
Fix $m\in\bZ_{\geq 1}$ so that $\fa$ is generated by $m$ elements.
    Let $s_2,s_3\in\bZ_{\geq 0}$ be as in Lemma \ref{lem:SplitingPairsToTriples}
    and let $c'\in \fa^{s_2}\cap \fa^{s_3}$.
    Let $R\to S$ and $\delta,\mu,\Delta,\rho,\nu\in\bZ_{\geq 0}$ be as in Theorem \ref{thm:UnifCohenGabberWithc} for $c=c'$.
    For a given $\fp\in\Spec(R)$ take $\fq\in\Spec(S)$ be as in Theorem \ref{thm:UnifCohenGabberWithc}.
    For $D=\Disc_{S^\wedge_\fq/P}(e_1,\ldots,e_n)$ as in  Theorem \ref{thm:UnifCohenGabber}\ref{UnifCG:disc}
    we have $(P,D)$ is a Cohen--Gabber normalization of $S^\wedge_\fq$ (torsion-freeness follows from Remark \ref{rem:Sreduced-CONSTRnbhd}).
    Let $e_0$ be as in Lemma \ref{lem:Unifm_S^peInm_R}.
    Lemma \ref{lem:IeF-sigIspandp-1seq}, applied to $R_\fp\to S_{\fq}^\wedge$, and Lemma \ref{lem:SplitingPairsToTriples}, then tell us $I_e^{\fa_\fp^t}S^\wedge_\fq$
    is a $p$- and $p^{-1}$-sequence  with respect to $(P,D)$, $c'$, $e_0$, and $\rho$.
    Since $S_\fq^\wedge$ is generated by at most $\delta^\mu$ elements as a $P$-module (cf. \eqref{numberofGeneratorsofSoverP}), 
    and since $c'S_{\fq}^\wedge\cap P\not\subseteq\fm_P^{\nu+1}$,
    we conclude from Theorem \ref{thm:pandp-1SeqEstimate} that
    $|s_{e'}(R_{\fp},\fa^t_\fp)-s_e(R_\fp,\fa^t_\fp)|\leq \delta^\mu(\nu+p\Delta)p^{(\rho+\nu+p\Delta-1)e_0}p^{-e},$
    as desired.

Finally, the case involving the divisor $\Delta$ follows from the same argument, using Lemma \ref{lem:SplittingIdealPairs} instead of Lemma \ref{lem:IeF-sigIspandp-1seq}, and replacing $c'$ with $cc'$ where $c\in R(-E)$ and $E$ is as in Lemma \ref{lem:DivisorSeqBounded}.
\end{proof}

\begin{Rem}
    In the second paragraph of the proof above we used Theorem \ref{thm:UnifCohenGabberWithc}, hence the algebra $S$ depends on $c$ (or at least the ideal $\sqrt{cR}$),
    thus the bound $C$ depends on $\Delta,\fa,t$ in a highly non-uniform way.

    If we assume $\dim R<\infty$, then we can actually use the same $S$ for all $0\neq c\in R$,
    assuming $s(R_\fp)>0$ for all $\fp\in\Spec(R)$ and $R$ is an integral domain.
    We briefly illustrate this.
    
    The construction in Proposition \ref{prop:locallyfindtamecurve}, which relies on Lemma \ref{lem:curveistameafterFieldExtn},
    actually ensures $\fm B$ is the maximal ideal of $B$.
    In particular $s(A)=s(B)$ (Fact \ref{fact:completionSameSe}).
    Therefore, in the proof of Proposition \ref{prop:UnifCohenGabberConstrNbhd},
    we have $s(S_\fq)>0$.
    Applying Theorem \ref{thm:sSemicont}\footnote{Note that technically $S$ may not satisfy Condition \ref{condit:excellence}, but that condition was used only to find the uniform Cohen--Gabber normalization, which $S$ already has.},
    we may therefore find an $S$ as in Theorem \ref{thm:UnifCohenGabber} such that we may choose $\fq$ so that additionally $s(S_\fq)>0$ for every $\fp\in\Spec R$,
    in particular $S_\fq^\wedge$ is normal and Cohen--Macaulay.
    We will show that such an $S$ works.
   
   Now $S_\fq^\wedge$ is a free $P$-module of rank at most $\delta$, we have $\fq^\delta S_\fq^\wedge\subseteq \fm_PS_\fq^\wedge$ 
    (see the proof of Lemma \ref{lem:Unifm_S^peInm_R}).
    Note that ${\fq}S_\fq^\wedge/\fm_PS_\fq^\wedge$ is generated by at most $\mu$ elements.
    Using that $l(X/y^\alpha X)\leq \alpha l(X/yX)$,
    it follows easily that for all finite $S_\fq^\wedge$-modules $M$ and all $e\in\bZ_{\geq 0}$ we have
    $$l_P(M/\ppower{\fm}{e}_P M)\leq \delta^\mu l_P(M/\ppower{\fq}{e}M)\leq \delta^{\mu+1} l_{S_\fq^\wedge}(M/\ppower{\fq}{e}M).$$

    Given $0\neq c\in R$, let 
$C$ be as in Corollary \ref{cor:PTY} for the $S$-module $S/cS$.
This yields $l_{S_\fq^\wedge}(S_\fq^\wedge/(c,\ppower{\fq}{e})S_\fq^\wedge)\leq Cp^{e(d-1)}$ where $d=\dim P=\HT(\fq)$.
Taking the Galois-theoretic norm $N(c)$ of $c$ into $P$,
we have $l_{S_\fq^\wedge}(S_\fq^\wedge/(N(c),\ppower{\fq}{e})S_\fq^\wedge)\leq \delta Cp^{e(d-1)}$.
Note that we also have $l_P(S_\fq^\wedge/\ppower{\fm}{e}_P S_\fq^\wedge)\geq l_P(P/\ppower{\fm}{e}_P)=p^{ed}$ as $S_\fq^\wedge$ is a nonzero free $P$-module.
It now follows that if $N(c)\in \ppower{\fm}{e}_P S_\fq^\wedge$, 
then $p^{ed}\leq \delta^{\mu+2}Cp^{e(d-1)}$.
Hence for $e=\lfloor\log_p \delta^{\mu+2}C\rfloor+1$ we have $N(c)\not\in \ppower{\fm}{e}_P$,
so $N(c)\not\in \fm_P^{(p^e-1)d+1}$,
and we may take $\nu=(p^e-1)\dim R$,
so $N(c)\in cS_\fq^\wedge\cap P$ and $N(c)\not\in \fm_P^{\nu+1}$, as desired.
\end{Rem}

Similar to the $s_e(R_\fp)$ case, we have 
\begin{Lem}[=Theorem \ref{thm:seTRIPLEsSemicont}]\label{lem:SeTRIPLEsSemicont}   Let $R$ be a Noetherian $\bF_p$-algebra so that $R/\fp$ is J-0 for all $\fp\in\Spec(R)$.
    Let $\fa$ be an ideal of $R$ not contained in any minimal prime of $R$ and let $t\in\bR_{>0}$.
    Then for every $e$, the function $\fp\mapsto s_e(R_\fp,\fa^t_\fp)$ is constructible.
    If $R$ is catenary and locally equidimensional,
    then the function $\fp\mapsto s_e(R_\fp,\fa^t_\fp)$ is lower semi-continuous.

    Assume $R$ is normal and let $\Delta$ be an effective $\bR$-Weil divisor on $R$.
    Then for every $e$, the function $\fp\mapsto s_e(R_\fp,\Delta_\fp,\fa^t_\fp)$ is constructible.
    If $R$ is catenary,
    then the function $\fp\mapsto s_e(R_\fp,\Delta_\fp,\fa^t_\fp)$ is lower semi-continuous.
\end{Lem}

We conclude the following theorem.
\begin{Thm}\label{thm:sTRIPLEsSemicont}
Let $R$ be a Noetherian $\bF_p$-algebra that satisfies 
Condition \ref{condit:excellence}.
Let $\fa$ be an ideal of $R$ not contained in any minimal prime of $R$ and let $t\in\bR_{>0}$.
Then the function $\fp\mapsto s(R_\fp,\fa_\fp^t)$ is lower semi-continuous.

Assume $R$ is normal and let $\Delta$ be an effective $\bR$-Weil divisor on $R$.
Then the function $\fp\mapsto s(R_\fp,\Delta_\fp,\fa_\fp^t)$ is lower semi-continuous.
\end{Thm}
\begin{proof}
    By Theorem \ref{thm:sSemicont} we may assume $s(R_\fp)>0$ for all $\fp\in \Spec R$, in particular $R$ is catenary and locally equidimensional.
    The result follows from uniform convergence (Proposition \ref{prop:SeTRIPLEsUnifConv}) and individual semicontinuity (Lemma \ref{lem:SeTRIPLEsSemicont}).
\end{proof}

\appendix
\section{Generic local duality and $F$-splitting ideals}\label{sec:seSemicont}

In this apprendix, we aim to establish the spreading-out of the $F$-splitting ideals to a constructible neighborhood,
Theorem \ref{thm:IFsigeConstr},
and its consequences, Theorems \ref{thm:ecSemicont}, \ref{thm:seSemicont}, and \ref{thm:SFRopenGENERAL}.
The results should be attributed to Hochster--Yao \cite{HochsterYaoSe}.
What we will do here is to remove the ``$(S_2)$ and quotient of Cohen--Macaulay'' condition in \cite{HochsterYaoSe} using \cite{Lyu-dual-complex-lift},
and to include the perspective on $F$-splitting ideals.
In fact, this generalization is predicted by \cite[Remark 7.14]{HochsterYaoSe},
except that instead of a module, we are using a bounded object in the derived category instead.
For this reason, we will use notations compatible with \cite{HochsterYaoSe}.\\

We use standard notations for derived category of rings.
In particular, we use $D^+$ (bounded below), $D^-$ (bounded above), $D^b$ (bounded), and $D_{Coh}$ (having finitely generated cohomology modules) for Noetherian rings.

\subsection{Generic Local Duality {\cite[\S 2 and \S 5]{HochsterYaoSe}}}\label{subsec:GenLocDual}
In this subsection we work in
\begin{Situ}\label{situ:ExistsDual}    
Let $R$ be a Noetherian ring, $\omega\in D^+_{Coh}(R)$, $P\in\Spec(R)$, $A=R/P$.
Fix a multiplicative subset
$W\subseteq R\setminus P$ mapping surjectively to $A\setminus \{0\}$.
Assume that $\omega_P$ is a dualizing complex for $R_P$, placed so that $\omega_P\in D^{\geq 0}(R_P)$ and  $H^0(\omega_P)\neq 0$.
Let $d=\HT(P)$ and write $E=H^d_P(\omega)$.
\end{Situ}

We will use the following facts without explicit reference.
\begin{Fact}
    Let $S$ be a multiplicative subset of $R$.
    Then for all $X\in D^-_{Coh}(R)$ and all $Y\in D^+(R)$,
    $S^{-1}\Ext^j_R(X,Y)=\Ext^j_{S^{-1}R}(S^{-1}X,S^{-1}Y)$,
    see \citestacks{0A6A}.
    Moreover, for all $X\in D(R)$,
    $S^{-1}R\Gamma_P(X)=R\Gamma_{S^{-1}P}(S^{-1}X)$, \citestacks{0957}.
    In particular,
    the results we proved in this subsection are preserved by further localizations.
\end{Fact}
\begin{Fact}
For all $X\in D^-_{Coh}(R),Y\in D^+(R)$,
there is a canonical isomorphism
$$R\Hom_{R}(X,R\Gamma_P(Y))=R\Gamma_P(R\Hom_{R}(X,Y)). $$
Indeed, there is a canonical map
$R\Hom_{R}(X,R\Gamma_P(Y))\to R\Hom_{R}(X,Y)$
given by the canonical map $R\Gamma_P(Y)\to Y$.
As $X\in D^-_{Coh}(R),Y\in D^+(R)$,
cohomologies of
$R\Hom_{R}(X,R\Gamma_P(Y))$ are $P^\infty$-torsion,
so we get a canonical map 
$$R\Hom_{R}(X,R\Gamma_P(Y))\to R\Gamma_P(R\Hom_{R}(X,Y)).$$
This map is an isomorphism as it is when $X$ is a free module of finite rank placed in any degree,
and since $X\in D^-_{Coh}(R),Y\in D^+(R)$,
and $R\Gamma_P$ has bounded cohomological dimension.
\end{Fact}
\begin{Lem}\label{lem:HomAE}
    In Situation \ref{situ:ExistsDual},
    for every $i$,
    after localizing at some $f=f(i)\in W$,
    we have $$H^i(R\Hom_{R}(A,R\Gamma_P(\omega)))=
    \begin{cases}0 &i\neq d \\
    A &i=d\end{cases}
    $$
\end{Lem}
\begin{proof}
    $R\Hom_{R}(A,R\Gamma_P(\omega))=R\Gamma_PR\Hom_{R}(A,\omega)=R\Hom_{R}(A,\omega)$ is in $D^+_{Coh}(R)$,
    and $R\Hom_{R}(A,\omega)_P=R\Hom_{R_P}(A_P,\omega_P)\cong A_P[-d]$,
    where the last isomorphism follows from the placement of the dualizing complex $\omega_P$.
    Hence the result holds when $W=R\setminus P$.
    A smaller $W$ works just fine as the modules of concern are $P^\infty$-torsion \cite[Remark 2.16]{HochsterYaoSe}.
\end{proof}

\begin{Lem}\label{lem:EisInjective}
In Situation \ref{situ:ExistsDual},
let $i\neq d$.
    Let $M$ be a finite $R$-module.
    Then after localizing at some $h=h(M,i)\in W$,    
    we have $$H^i(R\Hom_{R}(M,R\Gamma_P(\omega)))=0.$$
    \end{Lem}
\begin{proof}    
    By Noetherian induction, we may assume the lemma is true for all proper quotients of $M$.
    If $P\in\Ass_R(M)$,
    then there exists a submodule of $M$ isomorphic to $A$,
    and we win by induction and Lemma \ref{lem:HomAE}.
 Otherwise, after localizing at any $g=g(M)\in W$ which lies in all associated primes that contain $P$, there exists $x\in P$ that is a nonzerodivisor on $M$.
    By induction,
    we have $H^i(R\Hom_{R}(M/xM,R\Gamma_P(\omega)))_f=0$,
    so $H:=H^i(R\Hom_{R}(M,R\Gamma_P(\omega)))_f$ has zero $x$-torsion.
    But $H$ is $P^\infty$-torsion,
    so $H=0$.
\end{proof}
\begin{Rem}
    The proof of Lemma \ref{lem:EisInjective} tells us the element $h(M,i)$ may take the product form $g(M)f(i)$, where $g(M)$ is independent of $i$ and $f(i)$ is as in Lemma \ref{lem:HomAE}. 
\end{Rem}

Note that if $\omega\in D^b_{Coh}(R)$,
then by Lemma \ref{lem:EisInjective} applied to $M=R$,
after localizing at some $f\in W$,
we have $R\Gamma_P(\omega)=E[-d]$.
\begin{Lem}\label{lem:HomAE-noRGamma}
    In Situation \ref{situ:ExistsDual},
    for every $i$,
    after localizing at some $g=g(i)\in W$,
    we have $$\Ext^i_{R}(A,E)=
    \begin{cases}0 &i\neq 0 \\
    A &i=0\end{cases}
    $$
\end{Lem}
\begin{proof}
Replacing $\omega$ by a truncation, we may assume $\omega$ is bounded, so we may assume $R\Gamma_P(\omega)=E[-d]$ as noted above.
Then $\Ext^i_{R}(A,E)=H^{d+i}R\Hom_{R}(A,R\Gamma_P(\omega))$,
    so Lemma \ref{lem:HomAE} applies.
\end{proof}

We say an $R$-module $X$ is \emph{$A$-filterable} if $X=\bigcup_{t=0}^{\infty} X_t$ with each $X_{t+1}/X_t$ a free $A$-module of finite rank.
\begin{Lem}\label{lem:freeFilterHomE}
In Situation \ref{situ:ExistsDual},
let $M$ be a finite $R$-module.
    Then after localizing at some $h=h(M)\in W$,
    $N:=\Hom_R(M,E)$ is $A$-filterable.
    Indeed, the filtration is given by $N_t=\Hom_R(M_t,E)=\operatorname{Ann}_{N}P^t$
    where $M_t=M/P^tM$.
\end{Lem}
\begin{proof}
Note that $\operatorname{Ann}_{N}P^t=\Hom_R(R/P^t,N)=\Hom_R((R/P^t)\otimes_R M,E)=\Hom_R(M_t,E)$.
    We may assume $G_t:=P^tM/P^{t+1}M$ is free over $A$ for all $t\geq 0$,
    see \cite[Lemma 2.10]{HochsterYaoSe}.
We may also assume, by Lemma \ref{lem:HomAE-noRGamma},
that $\Ext^1_R(A,E)=0$ and $\Hom_R(A,E)\cong A$.
Therefore $\Ext^1_R(G_t,E)=0$ and $\Hom_R(G_t,E)$ is a free $A$-module (of finite rank).
    
    Let $M_t=M/P^tM$.
    Then $N=\colim_t\Hom_R(M_t,E)$ as $E$ is $P^\infty$-torsion and $M$ is finite.
As $\Ext^1_R(G_t,E)=0$ we have an exact sequence
\[
\begin{CD}
    0@>>>\Hom_R(M_{t},E)@>>>\Hom_R(M_{t+1},E)@>>>\Hom_R(G_{t},E)\\
    @>>>\Ext^1_R(M_{t},E)@>>>\Ext^1_R(M_{t+1},E)@>>>0.
\end{CD}
\]
As $M_0=0$, we see inductively that $\Ext^1_R(M_{t},E)=0$,
and as $\Hom_R(G_{t},E)$ is free,
we see $N=\colim_t\Hom_R(M_t,E)$
is $A$-filterable.
\end{proof}

\begin{Lem}\label{lem:genericLocalDual}
In Situation \ref{situ:ExistsDual},
    for all $M\in D^-_{Coh}(R)$ and all $j$, after localizing at some $f=f(M,j)\in W$,
    there is a canonical isomorphism
    \begin{align}\label{genericLocalDual}
    \Hom_R(\Ext^{d-j}_R(M,\omega),E)= H^j_P(M)
    \end{align}
    of $R$-modules.
    If $A_h$ is Gorenstein  and $\omega_h,M_h\in D^b_{Coh}(R)$ for some $h\in W$,
    then $f=f(M)$ can be chosen independently of $j$.
\end{Lem}
\begin{proof}
As $M\in D^-_{Coh}(R)$, 
we may replace $\omega$ by a truncation to assume $\omega\in D^b_{Coh}(R)$.
We may assume $E[-d]=R\Gamma_P(\omega)$ as noted above.

Assume for the moment $R\Hom_R(M,\omega)$ is bounded.
    We have a canonical map 
    \begin{align}\label{E20qtolimit}
           \Hom_R(\Ext^{d-j}_R(M,\omega),E)\to H^jR\Hom_R(R\Hom_R(M,\omega),E[-d])
    \end{align}
    given by the spectral sequence
    $E_2^{p,q}=\Ext^p_R(H^{d-q}R\Hom_R(M,\omega),E)$
    converging to $H^{p+q}(R\Hom_R(M,\omega),E[-d])$.
    Since $E_2^{p,q}$ is nonzero only in a bounded range of $q$,
    and since Lemma \ref{lem:EisInjective}
    tells us when $p\neq 0$, $E_2^{p,q}=0$ after localization,
    we see \eqref{E20qtolimit}
    is an isomorphism after localization
    at some $g=g(M,j)$.

    In particular, for $M=R$,
    \eqref{E20qtolimit} is well-defined and its target is $H^jR\Hom_R(\omega,R\Gamma_P(\omega))=H^jR\Gamma_PR\Hom_R(\omega,\omega).$
    We know $S:=R\Hom_R(\omega,\omega)$ is in $D^+_{Coh}(R)$ and we know $S_P=R_P$ as $\omega_P$ is a dualizing complex.
    Therefore, 
    after localizing at another $h=h(j)\in W$,
    we have $H^jR\Gamma_P(S)=H^j_P(R)$.
    This gives our desired isomorphism in the case $M=R$,
    and therefore for all perfect $M$.

    For a general $M\in D^-_{Coh}(R)$,
    and any given integer $t$,
    there exists a perfect $K$ and a map $K\to M$ so that $\operatorname{Cone}(K\to M)\in D^{\leq t}(R)$.
    As $\omega$ is bounded and as $R\Gamma_P$
    has bounded cohomological dimension,
    for any $j$ we can fix a $t$ so that for such a $K$, both sides of \eqref{genericLocalDual} are the same for $K$ and $M$.
    Therefore we get \eqref{genericLocalDual} for $M$.

Finally, after possibly localizing, assume $A$ is Gorenstein and $\omega,M\in D^b_{Coh}(R)$.
We claim that
after localizing at an element of $W$ depending only on $M$,
for all but finitely many $j$,
the left hand side of \eqref{genericLocalDual} is zero.
Since we also have $R\Gamma_P(M)\in D^b(R)$ for all $M\in D^b_{Coh}(R)$,
we see for all but finitely many $j$ both sides of \eqref{genericLocalDual} are zero.
Therefore we can take $f=f(M)$ independently of $j$.

To see the claim, let $(-)'$ denotes the $P$-adic completion.
Then $R',P',\omega'$ put us in Situation \ref{situ:ExistsDual} with the image of $W$ in $R'$, the same $d$, and the same $E$, considered as an $R'$-module as it is $P^\infty$-torsion.
Moreover, the left hand side of \eqref{genericLocalDual} is also unchanged
as, generally, $\Hom_R(X,Y)=\Hom_{R'}(X',Y)$ for a finite $R$-module $X$ and a $P^\infty$-torsion $R$-module $Y$.
Therefore, we may replace $R$ by $R'$ to assume $R$ is $P$-adically complete.
As $A$ is Gorenstein, $R$ admits a pseudo-dualizing complex by Theorem \ref{thm:dualCPLX}\ref{duCX:formallift},
which after localizing at some $s\in R\setminus P$ is quasi-isomorphic to a shift of $\omega$, cf. \cite[Lemma 2.1]{Lyu-dual-complex-lift}.
By Theorem \ref{thm:dualCPLX}\ref{duCX:preserveBDD}
$R\Hom_R(M,\omega)_s$ is bounded.
Therefore, after localizing at $s$,
for all but finitely many $j$,
the left hand side of \eqref{genericLocalDual} is zero.
Since the left hand side of \eqref{genericLocalDual} is always $P^\infty$-torsion,
we can replace $s$ by an element in $W$, see \cite[Remark 2.16]{HochsterYaoSe}.
\end{proof}

\begin{Thm}\label{thm:freefilterLocCoh}
In Situation \ref{situ:ExistsDual},
    for $M\in D_{Coh}(R)$ and all $j$, after localizing at some $g=g(M,j)\in W$,
$H^j_P(M)$
is $A$-filterable.
    If $M\in D^b_{Coh}(R)$,
    then $g=g(M)$ can be chosen independently of $j$.    
\end{Thm}
\begin{proof}
We may assume $M\in D^b_{Coh}(R)$ as $R\Gamma_P$ has bounded cohomological dimension, which also gives independence of $j$.
    The existence of $g=g(M,j)\in W$ follows from Lemmas \ref{lem:genericLocalDual}
    and \ref{lem:freeFilterHomE}.
\end{proof}

\subsection{The $e$th $F$-splitting ideal}
\begin{Def}[{cf. \cite[Definition 4.3]{Tucker-SRexists}}]\label{def:SplittingIdeals}
Let $e\in\bZ_{\geq 0}$.
    Let $(T,\fm)$ be a Noetherian local $\bF_p$-algebra.
Let $M$ be a finite $F^e_*T$-module and let $\varphi:T\to M$ be a $T$-linear map.
The \emph{splitting ideal of $\varphi$}, denoted $I_{\varphi}$, is 
    \[
    I_{\varphi}=\{c\in T\mid 1\mapsto (F^e_*c)\varphi(1): T\to M\text{ is not pure}\}.
    \]
For  the injective hull $E_T$  of $k=T/\fm$,
\cite[Proposition 1.3(5)]{Fed83-complete}
tells us $c\in I_{\varphi}$ if and only if $k$ is killed by
$x\mapsto x\otimes (F^e_*c)\varphi(1):E_T\to E_T\otimes_T M$.
Therefore $I_{\varphi}$ is an ideal
and $I_{\varphi}^\wedge=I_{\varphi^\wedge}$,
where $(-)^\wedge$ is $\fm$-adic completion.

    Specializing to the case where $M=F^e_*T$ and $\varphi$ the canonical map, we have \emph{$e$th $F$-splitting ideal of $T$}, denoted $\IFsig{e}{T}$.
    In other words,
    \[
    \IFsig{e}{T}=\{c\in T\mid 1\mapsto F^e_*c: T\to F^e_*T\text{ is not pure}\}.
    \]
\end{Def}
The module $K_e$ as in \cite[Definition 1.1]{EY-splittingnumbers}
 is identified with $F^e_*T/F^e_*\IFsig{e}{T}$,
 so $l(T/\IFsig{e}{T})=s_e(T)p^{e\dim T}.$

\begin{Thm}[{cf. \cite[Lemma 7.13]{HochsterYaoSe}}]\label{thm:IFsigeConstr}
    Let $R$ be a Noetherian $\bF_p$-algebra, $P\in\Spec(R)$.
    Assume that $R/P$ is regular.
    
    Let $M$ be a finite $F^e_*R$-module for some $e\in\bZ_{\geq 0}$ and let $\varphi:R\to M$ be an $R$-linear map.
    Then there exists an ideal $I=I(\varphi)$ of $R$ and an $f=f(\varphi)\in R\setminus P$
    such that for all $Q\in V(P)\cap D(f)$,
    and every sequence of elements $\underline{y}$ of $Q$ mapping to a regular system of parameters of $(R/P)_Q$,
    the splitting ideal of $\varphi_Q:R_Q\to M_Q$ is equal to
    $IR_Q+(\underline{y}^{p^e})R_Q$.

    In particular, for $M=F^e_*R$ and $\varphi$ the canonical map, we get an ideal $I=I(e)$ and an $f=f(e)$ such that $\IFsig{e}{R_Q}=IR_Q+(\underline{y}^{p^e})R_Q$ for all $Q\in V(P)\cap D(f)$. 
\end{Thm}
\begin{proof}    

By Discussion \ref{discu:localize-complete},
we may replace $R$ by the $P$-adic completion $\lim_t R/P^t$,
so $R$ admits a pseudo-dualizing complex $\omega\in D^b_{Coh}(R)$.
This puts us in Situation \ref{situ:ExistsDual}
with $W=R\setminus P$,
after possibly shifting $\omega$.
We use the notations $A,d,E$ there.
We may assume $E$ is $A$-filterable, Lemma \ref{lem:freeFilterHomE},
and that $R\Gamma_P(\omega)=E[-d]$,
Lemma \ref{lem:EisInjective}.

The $R$-algebra $F^e_*R$ is isomorphic to $R$ as an abstract ring.
We have $R\Gamma_Q(N)=R\Gamma_{F^e_*Q}(N)$ for all $N\in D(F^e_*R)$ and all $Q\in\Spec(R)$,
as $\sqrt{QF^e_*R}=F^e_*Q$.
Consider $X=M\otimes^L_R \omega\in D^{-}_{Coh}(F^e_*R)$.
One has $R\Gamma_{P}(X)=M\otimes^L_R E[-d]\in D^{\leq d}(F^e_*R)$,
so $H^{d-j}R\Gamma_{P}(X)=\Tor_j^R(M,E)$.
In particular, after localization, $M\otimes_R E=H^{d}R\Gamma_{F^e_*P}(X)$ is $(F^e_*A)$-filterable by Theorem \ref{thm:freefilterLocCoh}.

Let $Q\in V(P)$ and let $\underline{y}$ be elements of $Q$ mapping to a regular system of parameters of $A_Q$.
Let $h=\dim A_Q$.
As $M\otimes_R E$ is $(F^e_*A)$-filterable and as $F^e_*A$ is flat over $A$,
$\underline{y}$ is a regular sequence on $M\otimes_R E$.
Moreover,
 as $M\otimes_R E=H^{d}R\Gamma_{P}(X)$ and as $R\Gamma_{P}(X)\in D^{\leq d}$,
$H^h_{(\underline{y})}(M\otimes_R E)=H^{d+h}R\Gamma_{P+(\underline{y})}(X)$,
so $H^h_{(\underline{y})}((M\otimes_R E)_Q)=H^{d+h}R\Gamma_{QR_Q}(X_Q)$.
However $R\Gamma_{QR_Q}(X_Q)=M_Q\otimes^L_{R_Q}R\Gamma_{QR_Q}(\omega_Q)$.
As $E$ is $A$-filterable we have $R\Gamma_{QR_Q}(\omega_Q)=H^h_{(\underline{y})}(E_Q)[-d-h]$
and thus
$H^{d+h}R\Gamma_{QR_Q}(X_Q)=M_Q\otimes_{R_Q}H^h_{(\underline{y})}(E_Q)$.
As $\omega_Q$ is a dualizing complex $H^h_{(\underline{y})}(E_Q)=E_{R_Q}$,
where $E_{R_Q}$ denotes the injective hull of the residue field of the local ring $R_Q$.
In other words, 
the map $\psi:=\varphi\otimes {\operatorname{id}_E}:E\to M\otimes_R E$, after taking
$H^{h}_{(\underline{y})}(-)_Q$,
becomes the canonical map $\varphi_Q\otimes {\operatorname{id}_{E_{R_Q}}}:E_{R_Q}\to M_Q\otimes_{R_Q}E_{R_Q}$.

At this point, we localize further to assume that $\operatorname{Ann}_E P=A$ and that $E/(\operatorname{Ann}_E P)$
is $A$-filterable, possible by Lemmas \ref{lem:HomAE-noRGamma} and \ref{lem:freeFilterHomE}.
Let $u$ generate $\operatorname{Ann}_E P$.
Then \cite[Corollary 4.5]{HochsterYaoSe}
tells us the element $[u;\underline{y}]\in H^{h}_{(\underline{y})}(E)_Q=E_{R_Q}$
has annihilator $QR_Q$, thus a socle generator of $E_{R_Q}$.
Let $\psi(u)\in M\otimes_R E$ be the image of $u$.
Let $I$ be the unique ideal of $R$ such that $F^e_*I=\operatorname{Ann}_{F^e_*R}\psi(u)$.
We shall show that $I$ is what we want.

The image of the socle generator $[u;\underline{y}]$
in $H^{h}_{(\underline{y})}(M\otimes_R E)$
is $[\psi(u);\underline{y}]$.
By \cite[Proposition 2.5(c)]{HochsterYaoSe},
we may assume $(M\otimes_R E)/((F^e_*R)\psi(u))$ is $(F^e_*A)$-filterable.
Then \cite[Corollary 4.5]{HochsterYaoSe}
tells us
the annihilator of $[\psi(u);\underline{y}]$
is $F^e_*(IR_Q)+(\underline{y})F^e_*R_Q=F^e_*(IR_Q+(\underline{y}^{p^e})R_Q)$.
Therefore, $I_{\varphi_Q}=IR_Q+(\underline{y}^{p^e})R_Q$,
as desired.
\end{proof}

\begin{Rem}\label{rem:IFsigAscent}
    By the same argument, we see if $(T,\fm)\to (T',\fm')$ is a flat local map of Noetherian local $\bF_p$-algebras where $T'/\fm T'$ is regular,
    then $\IFsig{e}{T'}=\IFsig{e}{T}T'+(\underline{y}^{p^e})T'$ for every sequence $\underline{y}$ that maps to a regular system of parameters of $T'/\fm T'$,
    so $s_e(T)=s_e(T')$.
    Indeed, we may assume $T,T'$ complete, and run the argument above for $R=T',P=\fm T'$,
    and $\omega$ base changed from $T$ (permissible by for example \cite[Lemma 7.1]{Lyu-dual-complex-lift}).
    Then the conditions $R\Gamma_P(\omega)=E[-d]$, $\operatorname{Ann}_E P=A$,
    and all the ``filterable'' conditions are trivial as everything is base changed from $T$ to $T'$ or from $F^e_*T$ to $F^e_*T'$,
    so no localizations need to be taken.
    In this context, this idea appeared in \cite[\S 4]{Schwede-Zhang-Bertini-F-pure-regular} and \cite[Chapter 7]{Ma-Polstra-book},
    specifically for $F$-purity and strong $F$-regularity, and many other texts for other $F$-singularity notions.
\end{Rem}

We are able to derive several consequences.

The \emph{$F$-purity exponent} of $c\in R$,
where $R$ is a Noetherian $\bF_p$-algebra,
is the smallest integer $e\geq 1$ such that the map $1\mapsto F^e_*c: R\to F^e_*R$ is pure,  or $\infty$.
The function $\mathfrak{e}_c(P)$ on $\Spec(R)$ is then defined as the $F$-purity exponent of $c\in R_P$.
\begin{Thm}\label{thm:ecSemicont}
    Let $R$ be a Noetherian $\bF_p$-algebra so that $R/P$ is J-0 for all $P\in\Spec(R)$.
    Then for every $c\in R$, the function $P\mapsto \mathfrak{e}_c(P)$ is constructible and upper semi-continuous.
    In particular, the $F$-pure locus of $R$ is open.
\end{Thm}
\begin{proof}
The ``in particular'' statement follows from the case $c=1$.

By \cite[Proposition 7.3]{HochsterYaoSe} 
it suffices to show for $P\in\Spec(R)$ fixed so that $e=\mathfrak{e}_c(P)<\infty$,
there exists a $g\not\in P$ so that $\mathfrak{e}_c(Q)\leq e$ for all $Q\in V(P)\cap D(g)$.
We may therefore localize $R$ near $P$ and assume $R/P$ regular.
Let $I,f$ be as in Theorem \ref{thm:IFsigeConstr} for $M=F^e_*R$.
    By definition, for   all $Q\in V(P)\cap D(f)$,
    $\mathfrak{e}_c(Q)\leq e$ if and only if $c\not\in \IFsig{e}{R_Q}=IR_Q+(\underline{y}^{p^e})R_Q$,
    where $\underline{y}$ is any sequence of elements of $Q$ that becomes a regular system of paramters of $(R/P)_Q$.

As $(R/I)_P$ is of finite length,
Discussion \ref{discu:localize} tells us
after localization $R_Q/(IR_Q+(\underline{y}^{p^e})R_Q)$
and 
and $R_Q/(IR_Q+cR_Q+(\underline{y}^{p^e})R_Q)$
have different lengths,
as $(R/I)_P$ and $(R/(I+cR))_P$ have different lengths, since $c\not\in IR_P$.
Therefore $c\not\in IR_Q+(\underline{y}^{p^e})R_Q$, as desired.
\end{proof}
\begin{Rem}
    \cite{Epstein-Datta-Schwede-Tucker-FROBENIUS-OHM-RUSH} includes 
    openness of the $F$-pure and related purity loci in another setting.
\end{Rem}
\begin{Thm}\label{thm:seSemicont}
    Let $R$ be a Noetherian $\bF_p$-algebra so that $R/P$ is J-0 for all $P\in\Spec(R)$.
    Then for every $e$, the function $P\mapsto s_e(R_P)$ is constructible.
    If $R$ is catenary and locally equidimensional,
    then the function $P\mapsto s_e(R_P)$ is lower semi-continuous.
\end{Thm}
\begin{proof}
    Constructibility follows directly from Theorem \ref{thm:IFsigeConstr}, Discussion \ref{discu:localize},
    and Theorem \ref{thm:uniftame}\ref{uniftame_height}.
     If $R$ is catenary and locally equidimensional, then $\HT(P_1)=\HT(P_2)+\HT(P_1/P_2)$ for all primes $P_1\supseteq P_2$.
     \cite[Proposition 5.2]{YaoSe}
     tells us $s_e(R_{P_1})\leq s_e(R_{P_2})$.
     Therefore the function $P\mapsto s_e(R_P)$ is lower semi-continuous.

\end{proof}

\begin{Rem}
    A constructible function only takes finitely many values.
    
    Therefore Theorem \ref{thm:ecSemicont} tells us the finite $F$-purity indices for a fixed $c$ are bounded.
    Theorem \ref{thm:seSemicont} tells us for a fixed $e$, the nonzero $F$-splitting numbers are bounded away from zero and have a common denominator, which is nontrivial when the ring has infinite Krull dimension.
\end{Rem}

\begin{Lem}\label{lem:SFRcriteria}
    Let $T$ be a Noetherian local $\bF_p$-algebra.    
Then the following hold.
\begin{enumerate}[label=$(\roman*)$]
\item\label{SFR:ONEc} Let $c\in T$ not in any minimal prime of $T$ be such that $T_c$ is strongly $F$-regular.
Assume $T\to T^\wedge$ has regular fibers.
    Then $T$ is strongly $F$-regular if and only if $1\mapsto F^e_*c:T\to F^e_*T$ is pure for some $e\geq 0$. 
    \item\label{SFR:ascent} Let $T\to T'$ be a flat local map Noetherian local $\bF_p$-algebras with regular generic and special fibers.
    Assume $T'\to T'^\wedge$ has regular fibers.
    If $T$ is strongly $F$-regular, so is
    $T'$.
\end{enumerate}
\end{Lem}
\begin{proof}
Consider \ref{SFR:ONEc}.
   ``Only if'' follows from the definition. 
   For ``if'', as $T^\wedge$ is excellent,
   the result is true for $T^\wedge$ by \cite[Proposition 3.34]{Has10}.
   As $\IFsig{e}{T}^\wedge=\IFsig{e}{T^\wedge}$,
   we only need to show $(T^\wedge)_c$ is strongly $F$-regular.
   This tells us \ref{SFR:ONEc} is true in the special case $T_c$ is regular,
   and the general case follows from \ref{SFR:ascent}.
   Note that the special case is enough for Theorem \ref{thm:SFRopenGENERAL}.

   It remains to show \ref{SFR:ascent}.
   We know the strongly $F$-regular ring $T$ is normal,   
   in particular an integral domain.
   Therefore $(T\setminus\{0\})^{-1}T'$ is the only generic fiber of $T\to T'$,
   thus regular.
   As $T'\to T'^\wedge$ has regular fibers,
   the regular locus of $T'$ is open (same proof as \cite[(33.D) Theorem 76]{Mat-CA}),
   so there exists $c\in T\setminus\{0\}$ such that $T'_c$ is regular.
   Since \ref{SFR:ONEc} is true in the special case,
   it suffices to show $c\not\in \IFsig{e}{T'}$ for some $e$.
   
   As $T$ is strongly $F$-regular,
   we know $c\not\in \IFsig{e}{T}$ for some $e$.
By Remark \ref{rem:IFsigAscent}, $\IFsig{e}{T'}=\IFsig{e}{T}T'+(\underline{y}^{p^e})T'$ for every sequence $\underline{y}$ that maps to a regular system of parameters of the closed fiber of $T\to T'$.
  Thus $\IFsig{e}{T'}\cap T=\IFsig{e}{T}$,
   as $T'/(\underline{y}^{p^e})T'$ is a faithfully flat $T$-algebra \citestacks{00MG}.
   Hence $c\not\in \IFsig{e}{T'}.$
\end{proof} 
\begin{Rem}
    The condition that the generic fiber of $T\to T'$ is regular in \ref{SFR:ascent} cannot be removed.
    Indeed, there exists a Gorenstein local ring $T$ that is weakly $F$-regular,
    and whose completion $T^\wedge$ is not weakly $F$-regular,
    see \cite[\S 5]{Loepp-Rotthaus-01-complete-bad-F-regular}.
    A careful inspection of the proof shows that their $T$ is strongly $F$-regular;
    alternatively,
     every Gorenstein weakly $F$-regular ring is strongly $F$-regular, by the same proof as the $F$-finite case \cite[Theorem 3.1(f)]{HH89-sFreg}.
\end{Rem}

\begin{Thm}\label{thm:SFRopenGENERAL}
    Let $R$ be a Noetherian $\bF_p$-algebra so that $R/P$ is J-0 for all $P\in\Spec(R)$.
    If $R_P\to R_P^\wedge$ has regular fibers for all $P\in\Spec(R)$,
    then the strongly $F$-regular locus of $R$ is open.
\end{Thm}
\begin{proof}
    Since the strongly $F$-regular locus of $R$ is contained in the normal locus of $R$ which is open,
    we may assume $R$ normal.
    
    By our assumption, there exists a nonzerodivisor $c\in R$
    such that  $R_c$ is regular.
    Then by Lemma \ref{lem:SFRcriteria}\ref{SFR:ONEc}, for every $P,$ $R_P$ is strongly $F$-regular if and only if $\mathfrak{e}_c(P)<\infty$.
    We conclude by Theorem \ref{thm:ecSemicont}.
\end{proof}

\subsection{Pairs and triples}
See Definition \ref{def:SplittingIdealTriples} for the splitting ideals.
\begin{Thm}
\label{thm:IFsigeTRIPLEsConstr}
    Let $R$ be a Noetherian $\bF_p$-algebra, $P\in\Spec(R)$.
    Assume that $R/P$ is regular.
    
    Let $\fa$ be an ideal of $R$ not contained in any minimal prime of $R$.
    Let $t\in\bR_{>0}$, $e\in\bZ_{\geq 0}.$
    Then there exists an ideal $I=I(e,\fa,t)$ of $R$ and an $f=f(e,\fa,t)\in R\setminus P$
    such that for all $Q\in V(P)\cap D(f)$,
    and every sequence of elements $\underline{y}$ of $Q$ mapping to a regular system of parameters of $(R/P)_Q$, $I^{\fa_Q^t}_{e}=IR_Q+(\underline{y}^{p^e})R_Q$.

Assume $R$ is normal and let $\Delta$ be an effective $\bR$-Weil divisor on $R$.
Then there exists an ideal $I=I(e,\Delta,\fa,t)$ of $R$ and an $f=f(e,\Delta,\fa,t)\in R\setminus P$
    such that for all $Q\in V(P)\cap D(f)$,
    and every sequence of elements $\underline{y}$ of $Q$ mapping to a regular system of parameters of $(R/P)_Q$, $I^{\Delta_Q,\fa_Q^t}_{e}=IR_Q+(\underline{y}^{p^e})R_Q$.
    
\end{Thm}
\begin{proof}
    If there is no $\fa$ (\emph{i.e.} $\fa=R$) then the result follows from Theorem \ref{thm:IFsigeConstr} for $M=F^e_*R$ and $M=F^e_*R(\lceil(p^e-1)\Delta\rceil)$.
    Since $e,t$ are fixed it suffices to show for all ideals $I,J$ of $R$ with $l(R_P/IR_P)<\infty$,
    we may localize near $P$ so that 
    \begin{align}\label{colonPreservedNearby}
        (IR_Q+(\underline{y}^{p^e})R_Q:JR_Q)=(IR_Q:JR_Q)+(\underline{y}^{p^e})R_Q.
    \end{align}
    This is because we may then apply this to the ideal $I=I(e,\Delta,R,t)$ and $J=\fa^{\lceil t(p^e-1)\rceil}$.
    If $x_1,\ldots,x_l$ generates $J$ then we have an injective map 
    \begin{align*}
        R/(I:J)&\to (R/I)^{\oplus l}\\
        a&\mapsto (ax_1,\ldots,ax_l),
    \end{align*}
    so \eqref{colonPreservedNearby} follows from Discussion \ref{discu:localize} (cf. Step \ref{step:Conductor} of the proof of Theorem \ref{thm:uniftame}).
\end{proof}

\begin{Thm}\label{thm:seTRIPLEsSemicont}
    Let $R$ be a Noetherian $\bF_p$-algebra so that $R/P$ is J-0 for all $P\in\Spec(R)$.
    Let $\fa$ be an ideal of $R$ not contained in any minimal prime of $R$ and let $t\in\bR_{>0}$.
    Then for every $e$, the function $P\mapsto s_e(R_P,\fa^t_P)$ is constructible.
    If $R$ is catenary and locally equidimensional,
    then the function $P\mapsto s_e(R_P,\fa^t_P)$ is lower semi-continuous.

    Assume $R$ is normal and let $\Delta$ be an effective $\bR$-Weil divisor on $R$.
    Then for every $e$, the function $P\mapsto s_e(R_P,\Delta_P,\fa^t_P)$ is constructible.
    If $R$ is catenary,
    then the function $P\mapsto s_e(R_P,\Delta_P,\fa^t_P)$ is lower semi-continuous.
\end{Thm}
\begin{proof}
    Again, constructibility follows directly from Theorem \ref{thm:IFsigeTRIPLEsConstr}, Discussion \ref{discu:localize},
    and Theorem \ref{thm:uniftame}\ref{uniftame_height}.

     For lower-semicontinuity, we only need to show the function is non-decreasing along generalizations.
     As $R$ is catenary and locally equidimensional, $\HT(P_1)=\HT(P_2)+\HT(P_1/P_2)$ for all primes $P_1\supseteq P_2$.
     Let us show, more generally, that for every ideal $\fb$ not contained in any minimal prime of $R$ and any $\varphi:R\to M$,
     where $M$ is a finite $F^e_*R$-module receiving a map $\varphi:R\to M$,
     we have $l(R_{P_1}/(I_{\varphi_{P_1}}:\fb_{P_1}))\leq p^{\HT(P_1/P_2)}l(R_{P_2}/(I_{\varphi_{P_2}}:\fb_{P_2}))$.
     
     \cite[Proposition 5.2]{YaoSe}
     is the case $\fb=R$, and similar to its proof we may assume $R=R_{P_1}$ is complete local with algebraically closed residue field,
     in particular $F$-finite.
     We rewrite $P_1=\fm,P_2=\fp$.
     In this case, 
     \begin{align*}
     (I_\varphi:\fb) &= \{x\in R\mid \psi (F^e_*bx.\varphi(1))\in \fm,\forall b\in\fb,\forall\psi\in \Hom_R(M,R)\}\\
     &= \{x\in R\mid \psi (F^e_*x.\varphi(1))\in \fm,\forall\psi\in F^e_*\fb.\Hom_R(M,R)\}
     \end{align*}
     Therefore, by \cite[Lemma 3.6]{BST-pairs-Cartier-algebra},
     $l_R(R/(I_\varphi:\fb))$ is the maximal number of $R$-direct summands of $M$ whose projection maps lie in $D:=F^e_*\fb.\Hom_R(M,R)$,
     and $[\kappa(\fp):\kappa(\fp)^p]l_{R_\fp}(R_\fp/(I_{\varphi_\fp}:\fb_\fp)))$ is the maximal number of $R_\fp$-direct summands of $M_\fp$ whose projection maps lie in $D_\fp$,
     hence $l_R(R/(I_\varphi:\fb))\leq [\kappa(\fp):\kappa(\fp)^p]l_{R_\fp}(R_\fp/(I_{\varphi_\fp}:\fb_\fp)))$.
     Since $[\kappa(\fp):\kappa(\fp)^p]=p^{\dim(R/\fp)}$ this is what we want.
\end{proof}

\printbibliography

\end{document}